\documentclass[11pt]{amsart}
\usepackage{amsmath} \usepackage{amsxtra}
\usepackage{amscd} \usepackage{amsthm}  
\usepackage{amsfonts}\usepackage{amssymb} 
\usepackage{eucal}\usepackage{bm}
\usepackage{graphicx} 
\newtheorem{Cor}[subsubsection]{Corollary}
\newtheorem{Lm}[subsubsection]{Lemma}
\newtheorem{Pp}[subsubsection]{Proposition}

\newtheorem{Thm}[subsubsection]{Theorem}
\newtheorem{Def}[subsubsection]{Definition}
\newtheorem{Rem}[subsubsection]{Remark}

\theoremstyle{definition}

\theoremstyle{remark}

\newcommand{\nc}{\newcommand}
\nc{\renc}{\renewcommand}
\nc{\ssec}{\subsection}
\nc{\sssec}{\subsubsection}
\nc{\on}{\operatorname}

\newcommand{\cA}{{\mathcal A}}
\newcommand{\cB}{{\mathcal B}}

\newcommand{\cD}{{\mathcal D}}
\newcommand{\cH}{{\mathcal H}}
\newcommand{\cE}{{\mathcal E}}
\newcommand{\cG}{{\mathcal G}}
\newcommand{\cI}{{\mathcal I}}
\newcommand{\cJ}{{\mathcal J}}
\newcommand{\cO}{{\mathcal O}}
\newcommand{\cL}{{\mathcal L}}
\newcommand{\cM}{{\mathcal M}}
\newcommand{\cN}{{\mathcal N}}
\newcommand{\cF}{{\mathcal F}}
\newcommand{\cK}{{\mathcal K}}
\newcommand{\cP}{{\mathcal P}}
\newcommand{\cQ}{{\mathcal Q}}
\newcommand{\cR}{{\mathcal R}}
\newcommand{\cS}{{\mathcal S}}
\newcommand{\cT}{{\mathcal T}}
\newcommand{\cU}{{\mathcal U}}
\newcommand{\cV}{{\mathcal V}}
\newcommand{\cW}{{\mathcal W}}
\newcommand{\cX}{{\mathcal X}}
\newcommand{\cY}{{\mathcal Y}}
\newcommand{\cZ}{{\mathcal Z}}

\newcommand{\GG}{{\mathbb G}}

\newcommand{\ZZ}{{\mathbb Z}}
\newcommand{\QQ}{{\mathbb Q}}

\newcommand{\HH}{{\mathbb H}}

\newcommand{\gm}{\mathfrak{m}}

\newcommand{\gL}{\mathfrak{L}}

\newcommand{\Sh}{\on{Sh}}

\newcommand{\Fl}{{\mathcal F}l}

\newcommand{\Qlb}{\mathbb{\bar Q}_\ell}
\newcommand{\Gm}{\mathbb{G}_m}

\newcommand{\A}{\mathbb{A}}

\newcommand{\toup}[1]{\stackrel{#1}{\to}}
\newcommand{\hook}[1]{\stackrel{#1}{\hookrightarrow}}

\newcommand{\getsup}[1]{\stackrel{#1}{\gets}}
\newcommand{\Sp}{\on{\mathbb{S}p}}

\newcommand{\GSp}{\on{G\mathbb{S}p}}
\newcommand{\IC}{\on{IC}}

\newcommand{\Hom}{\on{Hom}}

\newcommand{\Ext}{\on{Ext}}

\newcommand{\Sym}{\on{Sym}}
\newcommand{\SO}{\on{S\mathbb{O}}}
\newcommand{\GO}{\on{G\mathbb{O}}}

\newcommand{\Aut}{\on{Aut}}

\newcommand{\RG}{\on{R\Gamma}}

\newcommand{\triv}{\on{triv}}

\newcommand{\Spr}{{{\mathcal S}pr}}

\newcommand{\Pic}{\on{Pic}}

\newcommand{\Bun}{\on{Bun}}

\newcommand{\rk}{\on{rk}}
\newcommand{\Spec}{\on{Spec}}

\newcommand{\Gr}{\on{Gr}}

\newcommand{\GL}{\on{GL}}

\newcommand{\Fr}{{\on{Fr}}}

\newcommand{\pr}{\on{pr}}
\newcommand{\id}{\on{id}}
\newcommand{\norm}{{norm}} 
\newcommand{\tr}{\on{tr}}
\newcommand{\QED}{$\square$} 
\newcommand{\Fq}{\mathbb{F}_q}  
\newcommand{\Fp}{\mathbb{F}_p}  

\newcommand{\iso}{{\widetilde\to}}

\newcommand{\comp}{\circ}
\newcommand{\Four}{\on{Four}}
\renewcommand{\H}{{\on{H}}}   
\newcommand{\R}{\on{R}\!}   
\renewcommand{\L}{\on{L}}     

\newcommand{\D}{\on{D}}       
\newcommand{\wt}{\widetilde}
\newcommand{\ov}[1]{\overline{#1}}
\newcommand{\select}[1]{{\it{#1}}}

\newcommand{\und}[1]{\underline{#1}}

\renewcommand{\div}{\on{div}}

\newcommand{\<}{\langle}
\renewcommand{\>}{\rangle}

\newcommand{\ev}{\mathit{ev}}

\newcommand{\Res}{\on{Res}}

\newcommand{\dimrel}{\on{dim.rel}}

\newcommand{\sign}{\on{sign}}

\newcommand{\Vect}{\on{Vect}}

\newcommand{\Ind}{\on{Ind}}

\newcommand{\ssum}{\on{sum}}

\nc{\Hi}{\on{Hi}}
\nc{\Lo}{\on{Lo}}
\nc{\Inv}{\on{Inv}}
\nc{\Fix}{\on{Fix}}
\nc{\oV}{\overset{\scriptscriptstyle\circ}{V}}
\nc{\ssharp}{{\scriptstyle\sharp}}
\nc{\Irr}{\on{Irr}}

\newenvironment{Prf}{\par\noindent {\it Proof }}{\QED}

\begin{document}

\title{Linear periods of automorphic sheaves for $\GL_{2n}$}
\author{S. Lysenko}
\address{Institut \'Elie Cartan Nancy (Math\'ematiques), Universit\'e de Lorraine, B.P. 239, F-54506 Vandoeuvre-l\`es-Nancy Cedex, France}
\email{Sergey.Lysenko@univ-lorraine.fr}

\begin{abstract} We calculate, in the framework of the geometric Langlands program, the periods of cuspidal automorphic sheaves for $\GL_{2n}$ along the Levi subgroup $\GL_n\times\GL_n$. We also solve the corresponding local problem.
\end{abstract}
\thanks{\select{Acknowledgements.} I am very greatful to M. Finkelberg for numerous fruitful discussions on the subject. I also thank the referees for useful comments.}
\maketitle
\tableofcontents

\section{Introduction} 
\sssec{} Let $X$ be a smooth projective connected curve over an algebraically closed field $k$. Write $\Bun_n$ for the moduli stack of rank $n$ vector bundles on $X$. In this paper we calculate the geometric  periods of cuspidal automorphic sheaves on $\Bun_{2n}$ for the Levi subgroup $\GL_n\times\GL_n$. Our results geometrise a similar calculation at the level of functions from \cite{FJ}, but our proof does not follow any existing argument of the theory of automorphic forms.

 More generally, if $G$ is a connected reductive group over $k$, $H\subset G$ is a spherical subgroup, one may ask, in the framework of the geometric Langlands program, about the periods of automorphic sheaves on $\Bun_G$ with respect to the subgroup $H$. The corresponding problem at the level of functions has been intensively studied in the theory of automorphic forms (\cite{FJ, JSh, JLR, SV} and more references in \cite{SV}). 
 
 In the case $G=\GL_n\times\GL_n$, $H$ is the diagonally embedded $\GL_n$ this problem was solved in \cite{L2, L5}, and in \select{loc.cit.} it was naturally divided into local and global parts. In our case $G=\GL_{2n}$ and $H$ is the Levi subgroup $\GL_n\times\GL_n$, it is spherical. We divide this problem into local and global parts also. 
 
  If $E$ is an irreducible rank $2n$ local system on $X$, one has the automorphic $E$-Hecke eigensheaf $\Aut_E$ on $\Bun_{2n}$ constructed in \cite{FGV, G}. The linear periods of $\Aut_E$ along $\GL_n\times\GL_n$ allow to distinguish the local systems $E$ on $X$, which admit a symplectic form (so, conjecturally, then $\Aut_E$ comes by the geometric Langlands functoriality from a smaller group). 
 
 This paper is also motivated by \cite{L4}, where for $G=\GSp_4$ and a $\check{G}$-local system $E_{\check{G}}$ on $X$ whose standard representation is irreducible, we constructed a $E_{\check{G}}$-Hecke eigensheaf $\cK_{E_{\check{G}}}$ on $\Bun_G$. In \cite{L4} we have not checked that $\cK_{E_{\check{G}}}$ is always non zero. This nonvanishing is established in the present paper. 
 
\sssec{} In this section we informally describe our main results and the ideas involved in the proofs.  

 Write $\Pic X$ for the Picard stack of $X$. Consider the diagram
\begin{equation}
\label{diag_for_introd_periods}
\Bun_{2n}\getsup{\nu_n} \Bun_n\times\Bun_n\toup{\det\times\det} \Pic X\times\Pic X,
\end{equation}
where $\nu_n(L_1, L_2)=L_1\oplus L_2$. The linear period of $\Aut_E$ is defined as
$$
Per_E=(\det\times\det)_!\nu_n^*\Aut_E[\dimrel(\nu_n)]
$$
For $d\ge 0$ let $X^{(d)}$ denote the $d$-th symmetric power of $X$.
Write $\Omega$ for the canonical line bundle on $X$. Let $\epsilon_X: X^{(d)}\times X^{(d'-d)}\to \Pic X\times\Pic X$ be the map sending $(D,D_1)$ to 
$$
(\Omega^{(2n-1)+(2n-3)+\ldots+1}(D), \; \Omega^{(2n-2)+(2n-4)+\ldots+2}(D+D_1))
$$  
For $d,d'\in\ZZ$ let $(\Pic X\times\Pic X)^{d,d'}$ be the connected component of $\Pic X\times\Pic X$ given for $(\cA, \cA')\in \Pic X\times \Pic X$ by 
$$
\deg \cA=d+\deg(\Omega^{(2n-1)+(2n-3)+\ldots+1}),\;\; 
\deg \cA'=d'+\deg(\Omega^{(2n-2)+(2n-4)+\ldots+2})
$$

Our main global result is Theorem~\ref{Thm_2.2.2}, which describes all the left truncations $\tau_{\ge N} Per_E$ for $N\in\ZZ$ over $(\Pic X\times\Pic X)^{d,d'}$. The answer is of local nature with respect to $X$, it makes sense for any, not necessarily irreducible, local system $E$ on $X$. The answer is given essentially by a suitable left truncation of
\begin{equation}
\label{direct_image_under_epsilon_X}
(\epsilon_X)_!((\wedge^2 E)^{(d)}\boxtimes E^{(d'-d)})
\end{equation}
(cf. Theorem~\ref{Thm_2.2.2} for a precise claim). This is an easy consequence of Theorem~\ref{Th_2.1.2}, which is our main local result. 

 Let $\cM_n$ be the stack classifying $L\in\Bun_n$ with a nonzero section $\Omega^{n-1}\hook{} L$, $q_n: \cM_n\to\Bun_n$ be the map sending this point to $L$. Our formulation of Theorem~\ref{Th_2.1.2} uses special features of the $\GL_n$ case, the existence of the mirabolic subgroup used in Laumon's construction of automorphic sheaves on $\Bun_n$ via descent under $\cM_n\to \Bun_n$. Namely, we replace (\ref{diag_for_introd_periods}) by the diagram
$$
{\cM_{2n}}\,\getsup{\nu_{\cZ}}\, \cZ_n\, \toup{\pi}\, \Pic X\times\Pic X, 
$$
where $\cZ_n$ is the stack classifying $L,L'\in\Bun_n$ with a nonzero section $\Omega^{2n-1}\hook{} L$. The map $\nu_{\cZ}$ sends this point to $M=L\oplus L'$ with the induced section $\Omega^{2n-1}\hook{} M$, and $\pi$ sends the above point to $(\det L, \det L')$. 

  Write $\cM_{n,d}$ for the connected component of $\cM_n$ given by 
$$
\deg L-\deg(\Omega^{(n-1)+(n-2)+\ldots+1})=d
$$ 
To a local system $E$ on $X$ Laumon has associated a complex $\cK_{n, E}^d$ on $\cM_{n,d}$ for $d\ge 0$ (cf. Section~\ref{Sect_2.3.1_now}). If $E$ is irreducible of rank $n$ then one has $q_n^*\Aut_E[\dimrel(q_n)]\,\iso\, \cK_{n, E}^d$  for $d\ge 0$ over $\cM_{n,d}$ by \cite{FGV, G}.

 We define the local linear period of $\cK_{2n,E}$ as 
\begin{equation}
\label{local_period_introd}
\pi_!(\nu_{\cZ}^*\cK_{2n, E})[\dimrel(\nu_{\cZ})]
\end{equation}
Theorem~\ref{Th_2.1.2} claims that for $d+d'\ge 0$ and any local system $E$ on $X$ of rank $2n$ the restriction of (\ref{local_period_introd}) to $(\Pic X\times\Pic X)^{d,d'}$ vanishes unless $0\le d\le d'$, and in the latter case (assuming for simplicity $\rk E=2n$) identifies with (\ref{direct_image_under_epsilon_X}) up to a shift. 

 In Section~\ref{Sect_2.3} we explain our plan of the proof of Theorem~\ref{Th_2.1.2} reducing it to Propositions~\ref{Pp_2.2.4}, \ref{Pp_2.2.5}. 
 
For $d\ge 0$ in Section~\ref{Sect_2.3.1_now} we define the morphism $q_{\cY}: \cY_{n,d}\to\cM_{n,d}$ and a perverse sheaf $\cP^d_{n, E}$ on $\cY_{n, d}$ appearing in Laumon's construction. By definition, if $E$ is a local system on $X$ then $\cK_{n,E}^d\,\iso\, (q_{\cY})_!\cP^d_{n,E}$. So, Theorem~\ref{Th_2.1.2} calculates certain direct image under the composition 
$$
\cY_{2n}\times_{\cM_{2n}}\cZ_n \toup{\pr_2}\cZ_n\toup{\pi}\Pic X\times\Pic X
$$ 

 In Section~\ref{Sect_2.3.3_now} we get a commutative diagram
$$ 
\begin{array}{ccccc}
\label{closed_substack_first_intro}
{\wt\cY_{n,d}}\times{\wt\cY'_{n, d'}} & \toup{\bar\pi} & X^{(d)}\times X^{(d')}& \getsup{i_X} & X^{(d)}\times X^{(d'-d)}\\
\downarrow\lefteqn{\scriptstyle q_{\cY\cZ}} && \downarrow\lefteqn{\scriptstyle \epsilon} & \swarrow\lefteqn{\scriptstyle \epsilon_X}\\
\cY_{2n}\times_{\cM_{2n}}\cZ_{2n} & \toup{\pi\comp \pr_2}& \Pic X\times\Pic X,
\end{array}
$$
where $q_{\cY\cZ}$ is a closed immersion. Here $i_X$ sends $(D, D_1)$ to $(D, D+D_1)$. Our Proposition~\ref{Pp_2.2.4} close in spirit to (\cite{L2}, Theorem~A) allows to replace the integration over $\cY_{2n}\times_{\cM_{2n}}\cZ_n$ by the integration over its closed substack ${\wt\cY_{n,d}}\times{\wt\cY'_{n, d'}}$. Theorem~\ref{Th_2.1.2} is so reduced to a calculation of the direct image 
\begin{equation}
\label{complex_one_introduction}
\bar\pi_!q_{\cY\cZ}^*(\cP_{2n, E}\boxtimes\Qlb)
\end{equation}
in Proposition~\ref{Pp_2.2.5}. Using the natural stratification of $\cY_{2n}$ from \cite{L2} and the description of the $*$-restrictions of $\cP_{2n, E}$ to the strata obtained in \select{loc.cit.}, we first calculate (\ref{complex_one_introduction}) via this stratification. This shows that (\ref{complex_one_introduction}) is placed in one cohomological degree, and provides a filtration on the corresponding constructible sheaf. Besides, we see that this sheaf is the extension by zero under $i_X$. It remains to identify this constructible sheaf with $(i_X)_!(\wedge^2 E)^{(d)}\boxtimes E^{(d'-d)}$ (in the case $\rk(E)=2n$). 

 The latter problem is absent at the level of functions. At the geometric level on the contrary this remaining part is the most difficult and involves new ideas. We divide it into two steps:
 
\medskip 
\noindent
{\scshape STEP 1)} is Proposition~\ref{Pp_1.2.10}. Unwinding Laumon's construction of $\cP_{2n, E}$, we see that the problem is to calculate certain highest direct image with compact support (for the usual t-structure). We introduce a diagram 
$$ 
{\wt\cQ_{n,d}}\times{\wt\cQ'_{n, d'}}\;\toup{\tilde\varphi\times\tilde\varphi'}\; {\wt\cY_{n,d}}\times{\wt\cY'_{n, d'}}\;\toup{\bar\pi} \; X^{(d)}\times X^{(d')}
$$
and identify the cohomology sheaf of (\ref{complex_one_introduction}) in the corresponding degree with some highest direct image under $\bar\pi\comp (\tilde\varphi\times\tilde\varphi')$. 
This is obtained by a reduction to some local statement with respect to $X$. We generalize this local result for an arbitrary reductive group and a Levi subgroup that we call \select{antistandard}. This generalization is perfomed in Appendix A independent of the rest of the paper. We consider this generalization as an important result of independent interest (Propositions~\ref{Pp_2.1.1} and \ref{Pp_2.1.15}). 
 
\medskip 
\noindent
{\scshape STEP 2)}  is Proposition~\ref{Pp_4.2.8}, which is in turn reduced to Theorem~\ref{Thm_4.2.11}. Their content is to decompose $\bar\pi\comp (\tilde\varphi\times\tilde\varphi')$ as 
$$
{\wt\cQ_{n,d}}\times{\wt\cQ'_{n,d'}}\;\toup{\tilde\beta\times\tilde\beta'} \;
{^{\le n}\Sh^d_0}\times{^{\le n}\Sh^{d'}_0}\;\toup{^{\le n}(\div\times\div)}\; X^{(d)}\times X^{(d')}
$$
and reduce the calculation of the highest direct image under $\bar\pi\comp (\tilde\varphi\times\tilde\varphi')$ to a calculation of the highest direct image under $^{\le n}(\div\times\div)$ over the closed subscheme 
$$
i_X: X^{(d)}\times X^{(d'-d)}\hook{} X^{(d)}\times X^{(d')}
$$ 
The last calculation given in Theorem~\ref{Thm_4.2.11} is of local nature with respect to $X$. We consider Theorem~\ref{Thm_4.2.11} as one of our main results, which is of independent interest.  
 
 For $0\le d\le d'$ in Section~\ref{Sect_4.3.5_now} we introduce the scheme $V^{d,d'}_-$ classifying $D\in X^{(d)}, D'\in X^{(d')}$ with $D\le D'$ and $(x_i)\in X^{d+d'}$ with $\sum_i x_i=D+D'$. The $\IC$-sheaf of 
$V^{d,d'}_-$ is described in Section~\ref{Sect_1.2.8_normalization}, it is a constructible sheaf placed in usual cohomological degree $-d'$. Let $\Sh^{d,d'}_-$ be the stack classifying torsions sheaves $F,F'$ on $X$ with $\div(F)\le \div(F')$, $d=\deg F, d'=\deg F'$, and a complete flag of torsion sheaves $F_1\subset\ldots\subset F_{d+d'}=F\oplus F'$. We have the map $\div^{\nu}_-: \Sh^{d,d'}_-\to V^{d,d'}_-$ defined in Section~\ref{Sect_4.3.8_now}. Theorem~\ref{Thm_4.2.11} claims, in particular, that each fibre of $\div^{\nu}_-$ is nonempty of dimension $-d'$, and 
$$
\R^{-2d'}(\div^{\nu}_-)_!\Qlb\,\iso\, \IC[-d']
$$ 
naturally. To prove Theorem~\ref{Thm_4.2.11}, we use two different stratifications of $\Sh^{d,d'}_-$. The definition of the second stratification is inspired by the results of Richardson and Springer from \cite{RS}. Write $\cE$ for the set of irreducible components of $V^{d,d'}_-$. We get an inclusion of $\cE$ in the set indexing the strata of $\Sh^{d,d'}_-$. It turns out that only these strata contribute to the highest direct image under $\div^{\nu}_-$, and this is the contribution of the corresponding irreducible components of $V^{d,d'}_-$ into $\IC[-d']$. We finish the proof of Theorem~\ref{Thm_4.2.11} using general Remarks~\ref{Rem_4.4.13} and \ref{Rem_2.4.6}. This completes the proof of Theorem~\ref{Th_2.1.2}. 
   
\sssec{Organization} Our main results are formulated in Section~\ref{Sect_Main_results}.  In Section~\ref{Sect_Other_results} we state some additional results (one may think of Proposition~\ref{Pp_2.3.3} as a version of Pieri's rule for $\GL_m$ over some configuration spaces of coweights). In Section~\ref{Sect_classical_theory} we derive some consequences from Theorem~\ref{Th_2.1.2} for the classical theory of automorphic forms. We obtain an integral representation of the $L$-function attached to $\wedge^2 E$ similar to the one given by Friedberg and Jacquet \cite{FJ}. The proofs are given in the remaining sections.
  
     
\ssec{Notation}
\label{Sect_Notation} Work over an algebraically closed ground field $k$ of characteristic $p>0$ (except in Section~\ref{Sect_2.4.3}, where $k=\Fq$ of characteristic $p>0$). All the schemes and stacks we consider are defined over $k$. Fix a prime $\ell$ and an algebraic closure $\Qlb$ of $\QQ_{\ell}$. Let $X$ be a smooth projective connected curve of genus $g$. Write $\Omega$ for the canonical line bundles on $X$. 
 
 We work with algebraic stacks in smooth topology and \'etale $\Qlb$-sheaves on them. We adopt the conventions of (\cite{L3}, Section~2.1). So, for an algebraic stack $S$ locally of finite type we have the derived category $\D(S)$ of $\ell$-adic sheaves on $S$. For a morphism $f: \cX\to\cY$ of algebraic stacks, the functors $f^*, f_*, f_!$ between the corresponding derived categories are understood in the derived sense. We use the notion of a \select{generalized affine fibration} from (\cite{L2}, Section~0.1.1). For a morphism of stacks $f: \cX\to\cY$ write $\dimrel(f)$ for the function of a connected component $C$ of $\cX$ given by $\dim C-\dim C'$, where $C'$ is the connected component of $\cY$ containing $f(C)$. 
 
 Fix a nontrivial character $\psi: \Fp\to \Qlb^*$ and denote by $\cL_{\psi}$ the Artin-Schreier sheaf on $\A^1$ associated to $\psi$. We will ignore the Tate twists everywhere (they are easy to recover if necessary). 

 For $n\ge 1$ write $\Bun_n$ for the stack of rank $n$ vector bundles on $X$. For $n=1$ we also write $\Pic X=\Bun_1$ for the Picard stack of $X$. By a modification $M'\subset M''$ of a locally free $\cO_X$-module $M''$ on $X$ we mean a quasi-coherent subsheaf such that $M'=M''$ over the generic point of $X$. For $d\ge 0$ denote by $\Sh^d_0$ the stack of torsion sheaves on $X$ of generic rank zero and length $d$. Let $^{\le n}\Sh^d_0\subset\Sh^d_0$ be the open substack given by the property: for a scheme $S$ an object $F$ of $\Hom(S,\Sh^d_0)$ lies in $\Hom(S, {^{\le n}\Sh^d_0})$ if the geometric fibre of $F$ at any point of $S\times X$ is of dimension $\le n$. 
 
 For $d\ge 0$ write $X^{(d)}$ for the $d$-th symmetric power of $X$.
We think of it as the scheme of effective divisors of degree $d$ on $X$. Let $\div: \Sh^d_0\to X^{(d)}$ be the morphism norm as in (\cite{L2}, Section~0.1.2). 

\sssec{} 
\label{Sect_1.1.1}
Fix the maximal torus of diagonal matrices in $\GL_n$ and the Borel subgroup of upper triangular matrices. The set of weights of $\GL_n$ is identified with $\ZZ^n$. 

We use the following (nonstandard) notations. Let $\Lambda_m=\ZZ^m$. For $a,b\in\ZZ^m$ we write $\<a,b\>=\sum_i a_ib_i$. Let $\Lambda_m^+=\{(\lambda_1,\ldots\lambda_m)\in\ZZ^m\mid \lambda_1\ge\ldots\ge\lambda_m\}$, $\Lambda_m^-=-\Lambda_m^+$. 
Let 
$$
\Lambda^{\succ}_m=\{\lambda\in\ZZ^m\mid \lambda_i\ge 0\;\;\mbox{for all}\; i\}
$$ 
Let $\Lambda_m^{\succ+}=\Lambda_m^+\cap \Lambda_m^{\succ}$. Let $\Lambda_m^{\succ-}=\Lambda_m^-\cap \Lambda_m^{\succ}$. For $d\in \ZZ$ let 
$$
\Lambda^{\succ+}_{m,d}=\{\lambda\in\Lambda^{\succ+}_m\mid \sum_i\lambda_i=d\}
$$ 
The second subindex $d$ on the right will mean that we impose the above condition $\sum_i\lambda_i=d$. This way we get, in particular, $\Lambda_{m, d}$, $\Lambda^{\succ-}_{m, d}$ and $\Lambda^{\succ}_{m,d}$. 

 Let 
$$
\Lambda_m^{pos}=\{\lambda\in\Lambda_{m, 0}\mid \lambda_1+\ldots+\lambda_i\ge 0\;\; \mbox{for}\;\; 1\le i\le m\}
$$ 
Set $\Lambda^{\succ pos}_m=\Lambda^{\succ}_m+\Lambda_m^{pos}$, we also get $\Lambda^{\succ pos}_{m,d}=\Lambda^{\succ}_{m,d}+\Lambda_m^{pos}$.

\sssec{} For $\lambda\in\Lambda^{\succ pos}_m$ set
$$
X^{\lambda}_{pos}=\prod_{i=1}^m X^{(\lambda_1+\ldots+\lambda_i)}
$$
This is the scheme of $\Lambda^{\succ pos}_m$-valued divisors of degree $\lambda$. A point of $X^{\lambda}_{pos}$ is a collection $(D_1,\ldots, D_m)$ with $D_1+\ldots+D_i\in X^{(\lambda_1+\ldots+\lambda_i)}$ for $1\le i\le m$. 

 Let $X^{\lambda}\subset X^{\lambda}_{pos}$ be the closed subscheme given by the property that $D_i\ge 0$ for all $i$. 
This is the scheme of $\Lambda^{\succ}_m$-valued divisors of degree $\lambda$. 

 Let $X^{\lambda}_-\subset X^{\lambda}$ be the closed subscheme given by $D_1\le \ldots\le D_m$. Let also $X^{\lambda}_+\subset X^{\lambda}$ be the closed subscheme given by $D_1\ge \ldots\ge D_m$. 

 For a local system $E$ on $X$ and $\lambda\in\Lambda^{\succ -}_n$ the sheaf $E^{\lambda}_-$ on $X^{\lambda}_-$ is introduced in (\cite{L2}, Definition 1). We denote also by $\cL^d_E$ (resp., $\Spr^d_E$) Laumon's (resp., Springer's) perverse sheaf on $\Sh^d_0$ defined in (\cite{L2}, Section~1). The group $S_d$ acts on $\Spr^d_E$, and $\cL^d_E$ is the perverse sheaf of $S_d$-invariants of $\Spr^d_E$.

\section{Main results}
\label{Sect_Main_results}

\ssec{} 
\label{sect_2.1}
Fix $n>0$. For $d\in\ZZ$ as in (\cite{L2}, Section~2.1), write $\cM_n$ for the stack classifying $L\in\Bun_n$ together with a subsheaf $\Omega^{n-1}\hook{} L$. Write $\cM_{n,d}$ for the connected component of $\cM_n$ given by $\deg L-\deg(\Omega^{(n-1)+(n-2)+\ldots+1})=d$.  Let $q_n: {\cM_{n,d}}\to\Bun_n$ be the projection sending the above point to $L$. 

 For a local system $E$ on $X$ and $d\ge 0$ one has a complex $\cK^d_{n, E}$ over $\cM_{n,d}$ defined in (\cite{L2}, Remark~1). If $E$ is irreducible of rank $n$ then for $d\ge 0$ there is an isomorphism $q_n^*\Aut_E[\dimrel(q_n)]\,\iso\, {\cK^d_{n,E}}$ over $\cM_{n,d}$ established in \cite{FGV, G}. Here $\Aut_E$ is the $E$-Hecke eigensheaf on $\Bun_n$ defined in \cite{FGV}, it is perverse. 
Though $\cK_{n,E}$ may depend on $\psi$, we do not express this in our notation.

\sssec{} 
\label{Sect_2.1.1}
Let $\cZ_n$ be the stack classifying $L, L'\in\Bun_n$ and a nonzero section $s: \Omega^{2n-1}\hook{} L$. Let $\nu_{\cZ}: \cZ_n\to {\cM_{2n}}$ be the map sending this point to $M=L\oplus L'$ with the induced section $s: \Omega^{2n-1}\hook{} M$. We get the diagram
$$
{\cM_{2n}}\,\getsup{\nu_{\cZ}}\, \cZ_n\, \toup{\pi}\, \Pic X\times\Pic X, 
$$
where $\pi$ is the map sending $(L,L',s)$ to $(\det L, \det L')$. 
Let $E$ be a local system on $X$. The local linear period of $\cK_{2n, E}$ is 
\begin{equation}
\label{period_linear_GL_2n}
\pi_!(\nu_{\cZ}^*\cK_{2n, E})[\dimrel(\nu_{\cZ})]
\end{equation}
For $d,d'\ge 0$ write $\epsilon: X^{(d)}\times X^{(d')}\to \Pic X\times\Pic X$ for the map sending $(D,D')$ to
$$
(\Omega^{(2n-1)+(2n-3)+\ldots+1}(D), \; \Omega^{(2n-2)+(2n-4)+\ldots+2}(D'))
$$  
For $d,d'\in\ZZ$ write $(\Pic X\times\Pic X)^{d,d'}$ for the connected component of $\Pic X\times\Pic X$ given for $(\cA, \cA')\in \Pic X\times \Pic X$ by 
$$
\deg \cA=d+\deg(\Omega^{(2n-1)+(2n-3)+\ldots+1}),\;\; 
\deg \cA'=d'+\deg(\Omega^{(2n-2)+(2n-4)+\ldots+2})
$$
If $0\le d\le d'$ we have the closed immersion
\begin{equation}
\label{map_i_X}
i_X: X^{(d)}\times X^{(d'-d)}\to X^{(d)}\times X^{(d')}
\end{equation}
sending $(D, D_1)$ to $(D, D+D_1)$. 
Our main result is the following.
\begin{Thm} 
\label{Th_2.1.2}
Let $E$ be any local system on $X$ and $d+d'\ge 0$. Over $(\Pic X\times\Pic X)^{d,d'}$ the complex (\ref{period_linear_GL_2n}) vanishes unless $0\le d\le d'$. If $0\le d\le d'$ and the rank of $E$ is $2n$ then (\ref{period_linear_GL_2n}) identifies canonically with
$$
\epsilon_!(i_X)_!((\wedge^2 E)^{(d)}\boxtimes E^{(d'-d)})[d']
$$
\end{Thm}
We actually prove a more precise claim for any local system $E$ (cf. Proposition~\ref{Pp_2.2.5} and Remark~\ref{Rem_2.2.6}). 

\ssec{Global linear periods}

\sssec{} Let $E$ be an irreducible rank $2n$ local system on $X$, recall the perverse sheaf $\Aut_E$ on $\Bun_{2n}$ from Section~\ref{sect_2.1}. Consider the diagram
\begin{equation}
\label{diag_linear_periods_main}
\Bun_{2n}\getsup{\nu_n} \Bun_n\times\Bun_n\toup{\det\times\det} \Pic X\times\Pic X,
\end{equation}
where $\nu_n(L_1, L_2)=L_1\oplus L_2$. The linear period of $\Aut_E$ is defined as
$$
Per_E=(\det\times\det)_!\nu_n^*\Aut_E[\dimrel(\nu_n)]
$$
For $d,d'\in\ZZ$ write $Per^{d,d'}_E$ for its restriction to the component $(\Pic X\times\Pic X)^{d,d'}$.

  For a local system $\cA$ of rank one on $X$, write $A\cA$ for the corresponding automorphic local system on $\Pic X$. This is a character local system such that for any $d\ge 0$ the restriction of $A\cA$ under $X^{(d)}\to\Pic X$, $D\mapsto \cO(D)$ identifies canonically with $\cA^{(d)}$. 
 
 Let $\Bun_1$ act on $\Pic X\times\Pic X$ by $\bar a: \Bun_1\times \Pic X\times\Pic X\to \Pic X\times\Pic X$ sending $(\cA, L_1, L_2)$ to $(L_1\otimes \cA^n, L_2\otimes \cA^n)$. Then $Per_E$ satisfies the following equivariance property for this action
\begin{equation}
\label{equiv_property_P_E}
\bar a^* Per_E\,\iso\, A(\det E)\boxtimes Per_E
\end{equation}
Note that $Per_E$ is also $S_2$-equivariant, where $S_2$ permutes the two copies of $\Pic X$.  

 The following is derived from Theorem~\ref{Th_2.1.2}. 
\begin{Thm} 
\label{Thm_2.2.2}
i) If $d,d'\in\ZZ$ then $Per^{d,d'}_E$ vanishes unless $4n(1-g)\le d-d'\le 0$. \\
ii) There is $N\in\ZZ$ such that for any $d,d'\in\ZZ$, $Per^{d,d'}_E$ is placed in usual cohomological degrees $\le N$. If $0\le d\le d'$ and $d\in\ZZ$ is large enough then one has
$$
\tau_{\ge N'} Per^{d,d'}_E\,\iso\, \tau_{\ge N'} (\epsilon_!(i_X)_!((\wedge^2 E)^{(d)}\boxtimes E^{(d'-d)})[d'+d+(g-1)(n-2n^2)])
$$
over $(\Pic X\times\Pic X)^{d,d'}$ with 
$$
N'=N-2d+(2g-2)(2n^2-n)+2
$$
\end{Thm}

 Taking into account the equivariance property (\ref{equiv_property_P_E}), this gives a description of all the left truncations of $Per^{d,d'}_E$ for all $d,d'\in\ZZ$. 
 
\begin{Rem} We did not look for the constant $N$ in general. One checks that for any $d\in\ZZ$ the complex $Per^{d,d}_E$ is placed in usual degrees $\le 2+(g-1)(2n^2-n-2)$.
\end{Rem} 

\ssec{Plan of the proof of Theorem~\ref{Th_2.1.2}} 
\label{Sect_2.3}

\sssec{}  
\label{Sect_2.3.1_now}
For $d\ge 0$ consider the stack $\cY_{m,d}$ classifying $L\in\Bun_m$ together with sections
\begin{equation}
\label{sections_t_i}
t_i: \Omega^{(m-1)+(m-2)+\ldots+(m-i)}\hook{}\wedge^i L 
\end{equation}
for $i=1,\ldots, m$ satisfying the Pl\"ucker relations, $D\in X^{(d)}$ such that 
$$
t_m: \Omega^{(m-1)+(m-2)+\ldots+1}(D)\,\iso\, \det L
$$ 
is an isomorphism (\cite{L2}, Section~4.1). For a local system $E$ on $X$ we have a perverse sheaf $\cP_{m,E}^d$ on $\cY_{m,d}$ introduced in (\cite{L2}, 4.2, Def. 2), see also Section~\ref{Sect_3.1.1}. Let $q_{\cY}: {\cY_{m,d}}\to {\cM_{m,d}}$ be the map sending the above point to $(L, t_1)$. By definition, $\cK_{m, E}^d\,\iso\, (q_{\cY})_!(\cP_{m,E}^d)$ for $d\ge 0$. 

\sssec{}  Let $\wt\cY_n$ be the stack classifying $L\in\Bun_n$ with regular sections 
$$
s_i: \Omega^{(2n-1)+(2n-3)+\ldots+(2n-2i+1)}\hook{} \wedge^i L
$$
for $1\le i\le n$ over $X$ satisfying the Pl\"ucker relations. For $d\ge 0$
the connected component $\wt\cY_{n,d}$ of $\wt\cY_n$ is defined by the property that there is $D\in X^{(d)}$ such that 
$$
s_n:\Omega^{(2n-1)+\ldots+1}(D)\,\iso\, \det L
$$ 
is an isomorphism. Let similarly $\wt\cY'_n$ be the stack classifying $L'\in\Bun_n$ with regular sections 
$$
s'_i: \Omega^{(2n-2)+(2n-4)+\ldots+(2n-2i)}\hook{} \wedge^i L'
$$
for $1\le i\le n$ satisfying the Pl\"ucker relations. For $d'\ge 0$ the connected component $\wt\cY'_{n,d'}$ of $\wt\cY'_n$ is defined by the property that there is $D'\in X^{(d')}$ such that 
$$
s'_n: \Omega^{(2n-2)+\ldots+2}(D')\,\iso\, \det L'
$$ 
is an isomorphism. Let 
$$
\bar\pi: {\wt\cY_{n,d}}\times{\wt\cY'_{n,d'}}\to X^{(d)}\times X^{(d')}
$$ 
be the map sending the above point to $(D,D')$. 

 Let $\iota: {\wt\cY_{n,d}}\times{\wt\cY'_{n,d'}}\to {\cY_{2n}}$ be the map sending the above point to $(M, (t_k))$, where $M=L\oplus L'$ and $t_k$ are defined as follows. For $1\le 2k\le 2n$ we get maps
$$
s_k\otimes s'_k: \Omega^{(2n-1)+\ldots+(2n-2k)}\hook{} (\wedge^k L)\otimes(\wedge^k L')\subset \wedge^{2k} M
$$
For $1< 2k+1\le 2n$ we get maps
$$
s_{k+1}\otimes s'_k: \Omega^{(2n-1)+\ldots+(2n-2k-1)}\hook{} (\wedge^{k+1} L)\otimes(\wedge^k L')\subset\wedge^{2k+1} M
$$
Set
\begin{equation}
\label{t_i_as_a_function_of_s_and_s'}
\begin{array}{ll}
t_{2k}=(-1)^{1+\ldots+(k-1)}s_k\otimes s'_k, & \mbox{for}\;\; 1\le 2k\le 2n\\  \\
t_{2k+1}=(-1)^{1+\ldots+k}s_{k+1}\otimes s'_k, & \mbox{for}\;\; 1< 2k+1\le 2n\\ \\
t_1=s_1
\end{array}
\end{equation} 
These $(t_k)$ automatically satisfy the Pl\"ucker relations, so that $\iota$ takes values in $\cY_{2n}$. 

\sssec{} 
\label{Sect_2.3.3_now}
We get the commutative diagram
\begin{equation}
\label{diag_the_first}
\begin{array}{ccccc}
{\wt\cY_{n,d}}\times{\wt\cY'_{n, d'}} & \toup{\bar\pi} & X^{(d)}\times X^{(d')}\\
\downarrow\lefteqn{\scriptstyle q_{\cY\cZ}} &&& \searrow{\lefteqn{\scriptstyle \epsilon}}\\
\cY_{2n}\times_{\cM_{2n}}\cZ_n & \toup{\pr_2} & \cZ_n & \toup{\pi} & \Pic X\times\Pic X\\
\downarrow\lefteqn{\scriptstyle\pr_1} && \downarrow\lefteqn{\scriptstyle\nu_{\cZ}}\\
\cY_{2n} & \to & {\cM_{2n}}
\end{array} 
\end{equation}
Here $q_{\cY\cZ}$ is the map whose projection to $\cY_{2n}$ is $\iota$, and its projection to $\cZ_n$ sends the corresponding point to $(L, L', t_1)$. 

For any local system $E$ on $X$ we will define a canonical constructible subsheaf 
\begin{equation}
\label{subsheaf_Pp_2.2.5}
^{\le n}((\wedge^2 E)^{(d)}\boxtimes E^{(d'-d)})\subset (\wedge^2 E)^{(d)}\boxtimes E^{(d'-d)}
\end{equation}
on $X^{(d)}\times X^{(d'-d)}$ in Definition~\ref{Def_4.3.11}. We refer the reader to \select{loc.cit.} for the precise construction. As $n$ varies, they form an increasing filtration on $(\wedge^2 E)^{(d)}\boxtimes E^{(d'-d)}$.

 Theorem~\ref{Th_2.1.2} is immediately reduced to Propositions~\ref{Pp_2.2.4}, \ref{Pp_2.2.5} below.

\begin{Pp} 
\label{Pp_2.2.4} Let $d+d'\ge 0$. For any local system $E$ on $X$ the natural map
$$
\pi_!\pr_{2!}(\cP_{2n, E}^{d+d'}\boxtimes\Qlb)\to \pi_!\pr_{2!} (q_{\cY\cZ})_*
 q_{\cY\cZ}^*(\cP_{2n, E}^{d+d'}\boxtimes\Qlb)
$$
is an isomorphism. In particular, it vanishes unless $d,d'\ge 0$.
Besides, $q_{\cY\cZ}$ is a closed immersion. 
\end{Pp}
\begin{Pp}
\label{Pp_2.2.5}
Let $d,d'\ge 0$. For any local system $E$ on $X$ there is a canonical isomorphism
\begin{equation}
\label{complex_Pp_2.2.5}
\bar\pi_! q_{\cY\cZ}^*(\cP_{2n, E}^{d+d'}\boxtimes\Qlb)[\dimrel(\nu_{\cZ})]\,\iso\, (i_X)_!(^{\le n}((\wedge^2 E)^{(d)}\boxtimes E^{(d'-d)}))[d']
\end{equation}
If $E$ is of rank $2n$ then (\ref{subsheaf_Pp_2.2.5}) is an equality.
In particular, (\ref{complex_Pp_2.2.5}) vanishes unless $d\le d'$.
\end{Pp}
 
 The proof of Proposition~\ref{Pp_2.2.4} (resp., \ref{Pp_2.2.5}) is the purpose of Section~\ref{Sect_Passing to a closed substack} (resp., \ref{Section_Direct_image_barpi}).
 
\begin{Rem} 
\label{Rem_2.2.6}
The situation here is similar to The Main local theorem of \cite{L2}, where in calculating some period for a pair of local systems $E, E'$ on $X$ we obtained as an answer the subsheaf $^{\le n}(E\otimes E')^{(d)}\subset (E\otimes E')^{(d)}$.
\end{Rem}

\ssec{Other results}
\label{Sect_Other_results}

\sssec{} 
\label{Sect_2.4.1_revision}
For a finite-dimensional $\Qlb$-vector space $V$ and $\lambda\in \Lambda^{\succ+}_{2n}$ consider the polynomial functor $V^{\lambda}$ of $V$ defined in (\cite{L2}, Section~0.1.4). 
For $\lambda\in\Lambda_{2n}$ let $\lambda^{odd}=(\lambda_1,\lambda_3,\ldots,\lambda_{2n-1})$, $\lambda^{even}=(\lambda_2, \lambda_4,\ldots,\lambda_{2n})$. 

\begin{Lm} 
\label{Lm_2.3.2}
Assume $\dim V=2n$. Then for $0\le d\le d'$ one has
$$
\Sym^d(\wedge^2 V)\otimes \Sym^{d'-d} V\,\iso\, \oplus_{\lambda} V^{\lambda}
$$
as representations of $\GL(V)$, the sum being taken over $\lambda\in\Lambda^{\succ+}_{2n, d+d'}$ such that $\lambda^{odd}\in\Lambda^{\succ+}_{n,d'}$ and 
$\lambda^{even}\in\Lambda^{\succ+}_{n,d}$. In particular, it is multiplicity free. 
\end{Lm}
\begin{proof}
Set temporarily
$$
\cJ_{d',d}=\{\lambda\in\Lambda^{\succ+}_{2n, d+d'}\mid \lambda^{odd}\in\Lambda^{\succ+}_{n,d'}, \lambda^{even}\in\Lambda^{\succ+}_{n,d}\}
$$
Consider the map $f: \cJ_{d',d}\to \cJ_{d,d}$ sending $\lambda$ to the unique $\mu\in \cJ_{d,d}$ such that $\mu^{odd}=\mu^{even}=\lambda^{even}$. 

\noindent
1) The case $d=d'$ is (\cite{JSh}, Proposition~1). Moreover, in this case for $\lambda\in \cJ_{d,d}$ we get $\lambda^{odd}=\lambda^{even}$. Indeed, for such $\lambda$ one has $\lambda_{2i-1}\ge\lambda_{2i}$ for $1\le i\le n$, and 
$$
\sum_{i=1}^n \lambda_{2i-1}=\sum_{i=1}^n \lambda_{2n}=d,
$$
so $\lambda_{2i-1}=\lambda_{2i}$ for all $i$. 

\smallskip
\noindent
2) The general case reduces to the case $d=d'$ using Pieri's rule for representations of $\GL_{2n}$ found in (\cite{O}, Proposition~2.1). Namely, if $\lambda\in\Lambda^{\succ+}_{2n, 2d}$ with $\lambda^{odd}=\lambda^{even}$ then Pieri's rule gives
$V^{\lambda}\otimes \Sym^{d'-d} V\,\iso\,\oplus_{\mu} V^{\mu}$, the sum being over $\mu\in \cJ_{d',d}$ with $f(\mu)=\lambda$. So,
$$
\Sym^d(\wedge^2 V)\otimes \Sym^{d'-d} V\,\iso\, \mathop{\oplus}\limits_{\lambda\in \cJ_{d,d}} V^{\lambda}\otimes \Sym^{d'-d} V\,\iso\, 
\mathop{\oplus}\limits_{\lambda\in \cJ_{d,d}}\,(
\mathop{\oplus}\limits_{\substack{
\mu\in \cJ_{d',d}\\
f(\mu)=\lambda}}
V_{\mu})\;\iso\, 
\mathop{\oplus}\limits_{\mu\in \cJ_{d',d}} V_{\mu}
$$
\end{proof} 

\sssec{} 
\label{Sect_2.3.3}
Let $0\le d\le d'$. For $\lambda\in \Lambda^{\succ-}_{2n, d+d'}$ with $\lambda^{odd}\in \Lambda^{\succ-}_{n,d}$, $\lambda^{even}\in \Lambda^{\succ-}_{n, d'}$ let 
$$
\ssum^{\lambda}: X^{\lambda}_-\to X^{(d)}\times X^{(d')}
$$ 
be the map sending $(D_1,\ldots, D_{2n})$ to $(D, D')$ with $D=D_1+D_3+\ldots+D_{2n-1}$ and $D'=D_2+D_4+\ldots+D_{2n}$. 

 We think of the following informally as a version of Lemma~\ref{Lm_2.3.2}, where a point is replaced by the curve $X$. The proof of Proposition~\ref{Pp_2.3.3} is given in Section~\ref{Sect_4.3.13}.  
\begin{Pp}
\label{Pp_2.3.3}
For any local system $E$ on $X$, the sheaf $^{\le n}((\wedge^2 E)^{(d)}\boxtimes E^{(d'-d)})$ admits a filtration with the associated graded 
$$
\mathop{\oplus}\limits_{\lambda} \;\ssum^{\lambda}_!E^{\lambda}_-,
$$
the sum being taken over $\lambda\in \Lambda^{\succ-}_{2n,d+d'}$ such that $\lambda^{odd}\in \Lambda^{\succ-}_{n,d}$ and $\lambda^{even}\in\Lambda^{\succ-}_{n,d'}$. 
Moreover, if $\rk(E)=2n$ then (\ref{subsheaf_Pp_2.2.5}) is an equality.
\end{Pp}

\sssec{} Among other results, we underline Theorem~\ref{Thm_4.2.11}, which is of independent interest. We refer the reader to Section~\ref{Sect_Highest direct image} for its formulation. One could think of it as an analog of (\cite{L2}, Theorem~C) in our setting. 

\begin{Rem}
\label{Rem_2.4.6}
The following phenomenon appears in Theorem~\ref{Thm_4.2.11}, it has also appeared in the local Rankin-Selberg method for $\GL_n$ in \cite{L2}. Let $f: Y\to Z$ be a morphism of stacks, $Z$ a scheme of finite type, $\cE$ the set of irreducible components of $Z$. Assume the normalization of $Z$ is of the form $\norm: \sqcup_{i\in\cE} \, Z_i\to Z$, where $Z_i$ is smooth irreducible, the image $\norm(Z_i)$ is the $i$-th irreducible component of $Z$. Assume all the fibres of $f$ are of dimension $\le d$ for some $d\in\ZZ$. To calculate $\R^{2d}f_!\Qlb$, we find a stratification of $Y$ by locally closed substacks $Y_i$, $i\in\cE$ with the property that the map $Y_i\to Z$ factors naturally as $Y_i\toup{f_i} Z_i\toup{\norm} Z$, and establish isomorphisms $\R^{2d}(f_i)_!\Qlb\,\iso\,\Qlb$ for $i\in\cE$. Then 
$$
\R^{2d}f_!\Qlb\,\iso\, \norm_!\Qlb
$$ 
Indeed, this direct image has a filtration with the associated graded $\oplus_{i\in \cE} \,(\norm_i)_!\Qlb$, and any such filtration splits canonically.
\end{Rem}

\ssec{Relation with the classical theory}
\label{Sect_classical_theory}
   
\sssec{} For $d,d'\in\ZZ$ let $\cZ_n^{d,d'}\subset\cZ_n$ be the component given by $(\det L, \det L')\in (\Pic X\times\Pic X)^{d,d'}$ for $(L, L', s:\Omega^{2n-1}\hook{}L)\in \cZ_n$. 

From Theorem~\ref{Th_2.1.2} one derives the following.
\begin{Cor}
\label{Cor_2.4.2}
Let $0\le d\le d'$ and $E$ be a local system of rank $2n$ on $X$. Let $V_1, V_2$ be local systems of rank one on $X$. 
One has 
\begin{multline}
\label{complex_Cor_2.4.2}
\RG_c(\cZ_n^{d,d'}, \nu_{\cZ}^*\cK^{d+d'}_{2n, E}\otimes \pi^*(AV_1\boxtimes AV_2))[n(g-1)]\,\iso\\  V_0\otimes
\RG(X^{(d)}\times X^{(d'-d)}, (V_1\otimes V_2\otimes \wedge^2 E)^{(d)}\boxtimes (V_2\otimes E)^{(d'-d)})[2d']
\end{multline}
Here 
$$
V_0=(AV_1)_{\Omega^{(2n-1)+(2n-3)+\ldots+1}}\otimes (AV_2)_{\Omega^{(2n-2)+(2n-4)+\ldots+2}}
$$
If $d,d'\in\ZZ$ with $d+d'\ge 0$ then the left hand side of (\ref{complex_Cor_2.4.2}) vanishes unless $0\le d\le d'$. \QED
\end{Cor}
 
\sssec{} 
\label{Sect_2.4.3}
For this subsection, assume the base field is $\Fq$. For a local system $E$ on $X$, the L-function of $E$ is defined as the formal series in $\Qlb[[t]]$
$$
L(E,t)=\sum_{d\ge 0} \sum_{D\in X^{(d)}(\Fq)} \tr(\Fr_D, E^{(d)})t^d
$$
According to the Grothendieck's trace formula, 
$$
L(E,t)=\prod_{r=0}^2 \det(1-t\Fr, \H^r(X\otimes_{\Fq} k, E))^{(-1)^{r+1}}
$$
Here $\Fr$ is the geometric Frobenius endomorphism. For an irreducible local system $E$ of rank $2n$ on $X$ let 
$$
\varphi_E: \Bun_{2n}(\Fq)\to\Qlb
$$ 
be the function trace of Frobenius of $\Aut_E$. 
Let $\varphi_{V_i}: (\Pic X)(\Fq)\to\Qlb$ be the function trace of Frobenius of $AV_i$ for rank one local systems $V_i$. Then Corollary~\ref{Cor_2.4.2} yields the equality in $\Qlb[[t]]$
\begin{multline*}
\sum_{d\ge 0} \sum_{(s,L, L')\in \cZ^{d,d}_n(\Fq)} \frac{1}{\ssharp\Aut(\Omega^{2n-1}\hook{} L)\ssharp\Aut(L')} \varphi_E(L\oplus L')\varphi_{V_1}(\det L)\varphi_{V_2}(\det L')t^d=\\ 
q^{(n-4n^2)(g-1)/2}\tr(\Fr, V_0)
L(V_1\otimes V_2\otimes\wedge^2 E, t)
\end{multline*}
Here $\ssharp A$ denotes the number of element of the set $A$.

 If $d$ is large enough and $\varphi_E(L\oplus L')\ne 0$ then $\Ext^1(\Omega^{2n-1}, L)=0$ and one has $\dim\Hom(\Omega^{2n-1}, L)=d+(g-1)(n-2n^2)$. One may derive some results of \cite{FJ} about linear periods by passing to the residue at $t=q^{-1}$ in the above equality.
 
\ssec{Application: automorphic sheaves for $\GSp_4$}

\sssec{} 
\label{Sect_2.6.1}
Use the following notations from (\cite{L4}, Section~3.3.1). For a reductive group $G$ over $k$, $\check{G}$ denotes its Langlands dual over $\Qlb$. Let $\GG=\GSp_4$, $E_{\check{\GG}}$ be a $\check{\GG}$-local system on $X$ viewed as a pair $(E,\chi)$, where $E$ (resp., $\chi$) is of rank 4 (resp., 1) local system on $X$ equipped with a symplectic form $\wedge^2 E\to\chi$. 

 Let $\HH=\GO^0_6$ and $\kappa: \check{\GG}\hook{}\check{\HH}$ be as in \select{loc.cit}. So, $\check{\HH}\,\iso\, \{(c,b)\in \Gm\times\GL_4\mid \det b=c^2\}$, and $\kappa$ is the natural inclusion. Let $E_{\check{\HH}}$ be the $\check{\HH}$-local system on $X$ obtained via extension of scalars via $\kappa$. Thus, $E_{\check{\HH}}$ is given by the pair $(E,\chi)$ with the induced isomorphism $\det E\,\iso\, \chi^2$ on $X$. Let $F_{\GG}: \D^-(\Bun_{\HH})\to \D^{\prec}(\Bun_{\GG})$ be the theta-lifting functor from (\cite{L4}, Section~2.2.6).  
 
 Assume $E$ irreducible on $X$. Under these assumptions we constructed the $E_{\check{\GG}}$-Hecke eigensheaf denoted $F_{\GG}(K_{E^*, \chi^*, \HH})\in\D^{\prec}(\Bun_{\GG})$ in (\cite{L4}, Theorem~3.3.4). 
\begin{Thm} 
\label{Thm_2.6.2}
Under the assumptions of Section~\ref{Sect_2.6.1} the complex $F_{\GG}(K_{E^*, \chi^*, \HH})$ is nonzero.
\end{Thm}

\section{Passing to a closed substack}
\label{Sect_Passing to a closed substack}

\noindent
In this section we prove Proposition~\ref{Pp_2.2.4}. In Section~\ref{Sect_3.1.1} we describe some general properties of the perverse sheaf $\cP^d_{n,E}$, namely its $*$-restriction to the natural stratification, and related results. After this preparation, the proof of Proposition~\ref{Pp_2.2.4} given in Section~\ref{Sect_Proof_of_Pp2.2.4} is a combination of Propositions~\ref{Con_1}, Lemma~\ref{Lm_3.3.4} and Corollary~\ref{Pp_three}. Namely, we will decompose $q_{\cY\cZ}$ into closed immersions
$$
{\wt\cY_{n,d}}\times{\wt\cY'_{n,d'}}\;\toup{i_{\cY}} \;{^{2n}\cZ^{d,d'}}\;\toup{^{2n}\kappa} \; {\cY_{2n}}\times_{\cM_{2n}} \cZ_n,
$$
and successively replace the integration over ${\cY_{2n}}\times_{\cM_{2n}} \cZ_n$ by the integration over a smaller substack in the above diagram.

\ssec{Generalities about $\cP_{n,E}^d$}
\label{Sect_3.1}

\sssec{} 
\label{Sect_3.1.1}
For $r\ge 0$ let $\cQ_{n,r}$ be the stack classifying collections
$$
(0=L_0\subset\ldots\subset L_n\subset L, (s_i)),
$$
where $L_n\subset L$ is a modification of rank $n$ vector bundles on $X$ with $\deg(L/L_n)=r$, $(L_i)$ is a complete flag of subbundles on $L_n$, and $s_i: \Omega^{n-i}\,\iso\, L_i/L_{i-1}$ is an isomorphism for $i=1,\ldots, n$. We have the diagram $\A^1\getsup{\mu}{\cQ_{n,r}}\toup{\beta} {^{\le n}\Sh^r_0}$ as in (\cite{L2}, Section~2.1), here $\beta$ sends the above point to $L/L_n$. As in \select{loc.cit.}, for a local system $E$ on $X$ one defines the perverse sheaf 
$$
\cF^r_{n,E}=\beta^*\cL_E^r\otimes\mu^*\cL_{\psi}[b(n,r)]
$$ 
with 
\begin{equation}
\label{def_of_b(n,r)}
b(n,r)=nr+(1-g)\sum_{i=1}^{n-1} i^2=\dim(\cQ_{n,r})
\end{equation}
Let $\varphi:{\cQ_{n,r}}\to {\cY_{n,r}}$ be the map sending the above point to $(L, (t_i))$, where 
$$
t_i:\Omega^{(n-1)+\ldots+(n-i)}\,\iso\,\wedge^i L_i\hook{}\wedge^i L
$$
are the induced maps for $i=1,\ldots, n$. By definition, $\cP^r_{n,E}\,\iso\,\varphi_!(\cF^r_{n,E})$. Write $\wt\cP_{n,E}^r$ for the complex on $\cY_{n,r}$ obtained from $\cP^r_{n,E}$ by replacing in its definition $\cL^r_E$ by $\Spr^r_E$ (as in \cite{L2}, p. 489 after Lemma~12). 

\sssec{} As in (\cite{L2}, Section~4.1), we stratify $\cY_{n,d}$ by locally closed substacks $S(\cY)^{\lambda}_{pos}$ indexed by $\lambda\in \Lambda^{\succ pos}_{n,d}$. The stratum $S(\cY)^{\lambda}_{pos}$ is given by requiring that the degree of zeros of $t_i: \Omega^{(n-1)+\ldots+(n-i)}\hook{}\wedge^i L$ equals $\lambda_1+\ldots+\lambda_i$ for $i=1,\ldots,n$. So, $S(\cY)^{\lambda}_{pos}$ classifies collections: $(D_i)\in X^{\lambda}_{pos}$, a complete flag of vector bundles $(0=L_0\subset L_1\subset\ldots\subset L_n=L)$ on $L\in\Bun_n$ with trivializations 
$$
\Omega^{(n-1)+\ldots+(n-i)}(D_1+\ldots+D_i)\,\iso\, \wedge^i L_i
$$
Let $S(\cY)^{\lambda}_-=S(\cY)^{\lambda}_{pos}\times_{X^{\lambda}_{pos}} X^{\lambda}_-$. We have a map $\mu_{\lambda}: S(\cY)^{\lambda}_-\to\A^1$ sending the above point to the sum in $k\;\iso\;\Ext^1(\Omega^{n-i-1}(D_i), \Omega^{n-i}(D_i))$ over $1\le i<n$ of the pull-backs of 
$$
0\to L_i/L_{i-1}\to L_{i+1}/L_{i-1}\to L_{i+1}/L_i\to 0
$$ 
under 
$$
\Omega^{n-i-1}(D_i)\hook{} \Omega^{n-i-1}(D_{i+1})\,\iso\,
L_{i+1}/L_i
$$ 

 Let $q_-^{\lambda}: S(\cY)^{\lambda}_-\to X^{\lambda}_-$ be the projection sending the above point to $(D_i)$. For $\lambda\in \ZZ^n$ set
$$
a_n(\lambda)=\<\lambda, (n-1,n-2,\ldots, 0)\>
$$
as in \cite{L2}. For $\lambda\in \Lambda_{n,d}^{\succ-}$
recall the definition of the sheaf $E^{\lambda}_-$ on $X^{\lambda}_-$ from (\cite{L2}, Definition~1). Let $f_{\lambda}: X^{\lambda}_-\to {^{\le n}\Sh^d_0}$ be the map sending $(D_1,\ldots, D_n)$ to
$$
\Omega^{n-1}(D_1)/\Omega^{n-1}\oplus\Omega^{n-2}(D_2)/\Omega^{n-2}\oplus\ldots\oplus \cO(D_n)/\cO
$$
For any smooth $\Qlb$-sheaf $E$ on $X$ set 
$$
E^{\lambda}_-=\cH^{2a_n(\lambda)}(f_{\lambda}^*\cL^d_E),
$$ 
where $\cH^{i}$ denotes the $i$-cohomology sheaf for the usual t-structure. 

 We will use the following.
\begin{Pp}[\cite{L2}, Proposition~2] 
\label{Pp_3.1.2}
Let $E$ be a smooth $\Qlb$-sheaf on $X$ of rank $m$, $\lambda\in\Lambda_{n,d}^{\succ pos}$. The $*$-restriction of $\cP^d_{n,E}$ to $S(\cY)^{\lambda}_{pos}$ vanishes unless $\lambda_1=\ldots=\lambda_{n-m}=0$ and $\lambda\in \Lambda_{n,d}^{\succ-}$. In the latter case it is the extension by zero under $S(\cY)^{\lambda}_-\hook{}S(\cY)^{\lambda}_{pos}$ of 
$$
(q_-^{\lambda})^*E^{\lambda}_-\otimes\mu_{\lambda}^*\cL_{\psi}[b(n,d)-2a_n(\lambda)]
$$ 
\end{Pp}

\sssec{}  
\label{Sect_3.1.4}
For $1\le j\le m$ let $_{j}\cY_m$ be the stack classifying $L\in\Bun_m$ together with sections (\ref{sections_t_i}) for $1\le i\le j$ satisfying the Pl\"ucker relations. For $j<m$ let $\delta_j: {\cY_m}\to {_{j+1}\cY_m}$ be the projection forgetting $t_{j+2},\ldots, t_m$. 
 
  Given $\lambda\in \Lambda^{\succ pos}_j$ consider the locally closed substack $_{j+1}\cY_m^{\lambda}\hook{} {_{j+1}\cY_m}$ given by the property that there is $(D_1,\ldots, D_j)\in X^{\lambda}_{pos}$ such that
$$
t_i: \Omega^{(m-1)+\ldots+(m-i)}(D_1+\ldots+D_i)\hook{} \wedge^i L
$$
is a subbundle for $1\le i\le j$. A point of $_{j+1}\cY_m^{\lambda}$ gives a flag $(L_1\subset\ldots\subset L_j\subset L)$ of vector bundles on $X$ with trivializations for $1\le i\le j$
$$
L_i/L_{i-1}\,\iso\, \Omega^{m-i}(D_i)
$$
and a section $t_{j+1}: \Omega^{(m-1)+\ldots+(m-j-1)}\hook{}
\wedge^jL_j\otimes L/L_j$. 
 
 Let $_{j+1}\cY^{\lambda}_{m,-}$ be the closed substack of $_{j+1}\cY^{\lambda}_m$ given by the properties: 
$0\le D_1\le\ldots\le D_j$, and there is a regular section $s: \Omega^{m-j-1}(D_j)\hook{} L/L_j$ such that 
$$
t_{j+1}=t_j\otimes s: \Omega^{(m-1)+\ldots+(m-j-1)}\hook{} \wedge^j L_j\otimes L/L_j\subset \wedge^{j+1}L
$$
We used the fact that $t_j: \Omega^{(m-1)+\ldots+(m-j)}\hook{} \wedge^j L_j$. 

 Let $_j\cW^{\lambda}$ be the stack classifying $(D_1,\ldots, D_j)\in X^{\lambda}_-$, $\bar L\in \Bun_{m-j}$ with a section $s: \Omega^{m-j-1}(D_j)\hook{} \bar L$. Let 
$$
\tau^{\lambda}:{_{j+1}\cY^{\lambda}_{m,-}}\to  {_j\cW^{\lambda}}
$$ 
be the map sending the above point to $(\bar L, (D_i), s)$ with $\bar L=L/L_j$. Let $\ev: {_{j+1}\cY^{\lambda}_{m,-}}\to\A^1$ be the map sending a point of $_{j+1}\cY^{\lambda}_{m,-}$ to the sum in $\H^1(X,\Omega)\,\iso\,\A^1$ of the pull-back extensions
$$
\begin{array}{ccc}
0\to L_i/L_{i-1}\to L_{i+1}/L_i\to  & L_{i+1}/L_i & \to 0\\
&\uparrow\\
& \Omega^{m-i-1}(D_i)
\end{array}
$$
for $1\le i\le j-1$ and the pull-back of $0\to L_j/L_{j-1}\to L/L_{j-1}\to L/L_j\to 0$ under $s: \Omega^{m-j-1}(D_j)\hook{} L/L_j$.

\begin{Lm} 
\label{Lm_P_E_on_m_j+1_cY}
Let $1\le j<m$, $E$ be any local system on $X$. The $*$-restriction of $\delta_{j !}(\cP_{m,E})$ to $_{j+1}\cY^{\lambda}_m$ is the extension by zero from $_{j+1}\cY^{\lambda}_{m,-}$ of 
$
\ev^*\cL_{\psi}\otimes (\tau^{\lambda})^*K'
$
for some complex $K'\in \D(_j\cW^{\lambda})$. \QED
\end{Lm} 
\begin{proof}
This follows from the equivariance properties of Whittaker sheaves (\cite{G}, Proposition~4.13) and (\cite{L2}, Proposition~2). A similar argument is also used in (\cite{G2}, Section~4.4-4.5).  

 Namely, denote temporarily by $\cP_j\subset \GL(\Omega^{m-1}\oplus\ldots\oplus\Omega^{m-j-1}\oplus \cO^{\oplus j-1})$ the parabolic group subscheme on $X$ preserving for each $1\le i\le j$ the subbundle $\Omega^{m-1}\oplus\ldots\oplus\Omega^{m-i}$. Let $\cN_j\subset \cP_j$ be its unipotent radical. Let $M_j=\cP_j/\cN_j$.  
 
 For a finite collection $\und{y}=\{y_1,\ldots, y_r\}\in X$ let $\gL_{\und{y}}(\cN_j)$ and 
$\gL_{\und{y}}^+(\cN_j)$ be the corresponding loop and arc groups as in (\cite{G2}, 4.4.1). As above, one gets a canonical character
$$
\ev_{\und{y}}: \gL_{\und{y}}(\cN_j)\to \prod_{s=1}^j \gL_{\und{y}}(\Omega)\to\A^1,
$$
where the last map is the sum of residues over all $s$ and all points of $\und{y}$. The map $\ev_{\und{y}}$ is trivial on $\gL^+_{\und{y}}(\cN_j)$. Write $D_{\und{y}}$ for the formal neighbourhood of $\und{y}$ in $X$. 

 Let $_{j+1}\cY^{\lambda}_{m,  \; \mbox{{\tiny good at}}\, \und{y}}\subset 
{_{j+1}\cY^{\lambda}_m}$ be the open substack given by the property that $\und{y}$ does not intersect $D_i$ for any $i$, and $t_{j+1}$ is a subbundle in a neighbourhood of $\und{y}$. 

 For a point of $_{j+1}\cY^{\lambda}_{m,  \; \mbox{{\tiny good at}}\, \und{y}}$ as above, the restriction of $L$ to $D_{\und{y}}$ naturally gives rise to a $\cP_j$-torsor on $D_{\und{y}}$. Moreover, $L/L_j$ is equipped with a section $s_{\und{y}}: \Omega^{m-j-1}\to L/L_j$ over $D_{\und{y}}$ such that $t_{j+1}=t_j\otimes s_{\und{y}}$ over $D_{\und{y}}$. 
 
   Write $\cH(_{j+1}\cY^{\lambda}_{m,  \; \mbox{{\tiny good at}}\, \und{y}})$ for the stack classifying a point of $_{j+1}\cY^{\lambda}_{m,  \; \mbox{{\tiny good at}}\, \und{y}}$ as above together with an isomorphism $L\,\iso\, \Omega^{m-1}\oplus\ldots\oplus \Omega^{m-j}\oplus L/L_j$ of $\cP_j$-torsors over $D_{\und{y}}$ inducing the identity on the corresponding $M_j$-torsors. The projection $\cH(_{j+1}\cY^{\lambda}_{m,  \; \mbox{{\tiny good at}}\, \und{y}})\to {_{j+1}\cY^{\lambda}_{m,  \; \mbox{{\tiny good at}}\, \und{y}}}$ is a $\gL_{\und{y}}^+(\cN_j)$-torsor. 
   
   As in (\cite{G2}, 4.4.3), the action of $\gL_{\und{y}}^+(\cN_j)$ on $\cH(_{j+1}\cY^{\lambda}_{m,  \; \mbox{{\tiny good at}}\, \und{y}})$ naturally extends to a $\gL_{\und{y}}(\cN_j)$-action. 
   
  Let $_j\cW^{\lambda}_{pos}$ be the stack classifying $(D_1,\ldots, D_j)\in X^{\lambda}_{pos}$, $\bar L\in \Bun_{m-j}$ and a section 
$$
s: \Omega^{m-j-1}\to \bar L(D_1+\ldots+D_j)
$$ 
We have the projection 
$\tau^{\lambda}_{pos}: {_{j+1}\cY^{\lambda}_m}\to {_j\cW^{\lambda}_{pos}}$ extending $\tau^{\lambda}$. Let $_j\cW^{\lambda}_{pos, \; \mbox{{\tiny good at}}\, \und{y}}\subset {_j\cW^{\lambda}_{pos}}$ be the open substack giving by a similar condition: $\und{y}$ does not intersect $D_i$ for all $i$, and $s$ is a subbundle in a neighbourhood of $\und{y}$. 

 The group ind-scheme $\gL_{\und{y}}(\cN_j)$ acts transitively on each fibre of the composition
$$
\cH(_{j+1}\cY^{\lambda}_{m,  \; \mbox{{\tiny good at}}\, \und{y}})\to {_{j+1}\cY^{\lambda}_{m,  \; \mbox{{\tiny good at}}\, \und{y}}}\toup{\tau^{\lambda}_{pos}}  {_j\cW^{\lambda}_{pos, \; \mbox{{\tiny good at}}\, \und{y}}},
$$
compare with (\cite{G2}, Lemma~4.4.6). One finishes the proof as in (\cite{G2}, Section~4.4-4.5).
\end{proof}

\ssec{Passing to $^{2n}\cY\cZ$} 
\label{Sect_passing_2n_YZ}

\sssec{} The purpose of Section~\ref{Sect_passing_2n_YZ} is to establish Proposition~\ref{Con_1} below. Its proof after several reductions boils down to a calculation of $\RG_c(V, f^*\cL_{\psi})$ for a finite-dimensional vector space $V$ and a linear morphism $f: V\to\A^1$. A reader may compare with the proof of (\cite{L2}, Theorem~A) which uses a similar strategy in a less involved setting.

\sssec{}
Write a point of $\cY_{2n}\times_{\cM_{2n}}\cZ_n$ as a collection $L,L'\in\Bun_n$ for which we set $M=L\oplus L'$, and sections
\begin{equation}
\label{section_t_i_for_M}
t_i: \Omega^{(2n-1)+\ldots+(2n-i)}\hook{}\wedge^i M
\end{equation}
for $1\le i\le 2n$ satisfying the Pl\"ucker relations. Here $t_1: \Omega^{2n-1}\hook{} L$. 

For $1\le j\le 2n$ let 
$$
^j\kappa: {^j\cY\cZ} \hook{} {\cY_{2n}\times_{\cM_{2n}}\cZ_n}
$$ 
be the closed substack given by the following property. For each $1\le i\le j$ we require the following condition (A):
\begin{itemize}
\item[A1)] if $i=2k$ then $t_{2k}$ factors as 
$$
t_{2k}: \Omega^{(2n-1)+\ldots+(2n-2k)}\hook{} (\wedge^k L)\otimes(\wedge^k L')\subset \wedge^{i} M
$$
\item[A2)] if $i=2k+1$ then $t_{2k+1}$ factors as 
$$
t_{2k+1}:\Omega^{(2n-1)+\ldots+(2n-2k-1)}\hook{} (\wedge^{k+1} L)\otimes(\wedge^k L')\subset \wedge^i M
$$
\end{itemize}
We get a diagram of closed embeddings
$$
^{2n}\cY\cZ\hook{}\ldots \hook{}{^2\cY\cZ}\hook{}{^1\cY\cZ}={\cY_{2n}\times_{\cM_{2n}}\cZ_n}
$$
Our first step is the following.
\begin{Pp} 
\label{Con_1}
Let $E$ be a local system on $X$. For each $1\le j<2n$ the natural map 
\begin{equation}
\label{map_for_step_i}
\pi_!\pr_{2!}(^{j}\kappa)_*(^{j}\kappa)^*(\cP_{2n, E}\boxtimes\Qlb)  
\to \pi_!\pr_{2!}(^{j+1}\kappa)_*(^{j+1}\kappa)^*(\cP_{2n, E}\boxtimes\Qlb)  
\end{equation}
is an isomorphism.
\end{Pp}

\sssec{} In this subsection we reduce Proposition~\ref{Con_1} to Proposition~\ref{Pp_one} below. For $1\le j\le 2n$ let $^j\wt\cZ$ be the stack classifying $L,L'\in\Bun_n$ for which we set $M=L\oplus L'$, and sections (\ref{section_t_i_for_M}) for $1\le i\le j$ satisfying the Plucker relations such that $t_1: \Omega^{2n-1}\hook{} L\subset M$.

 Let $^j\cZ\hook{}{^j\wt\cZ}$ be the closed substack given by the condition (A) on $t_i$ for each $1\le i\le j$. Note that $^1\cZ={^1\wt\cZ}$. Set 
$$
^{j+1}\cZ'={^{j+1}\wt\cZ}\times_{^j\wt\cZ} {^j\cZ},
$$
here the projection $^{j+1}\wt\cZ\to {^j\wt\cZ}$ forgets $t_{j+1}$. For $1\le j<2n$ denote by 
$$
^j\zeta: {^j\cY\cZ}\to {^{j+1}\cZ'}
$$ 
the projection that forgets $t_{j+2},\ldots, t_{2n}$. Set
$$
^{j+1}K_E={^j\zeta_!}(^j\kappa)^*(\cP_{2n, E}\boxtimes\Qlb)
$$
We have a cartesian square
$$
\begin{array}{ccc}
^{j+1}\cY\cZ & \hook{} & ^j\cY\cZ\\
\downarrow && \downarrow\lefteqn{\scriptstyle ^j\zeta}\\
^{j+1}\cZ & \toup{\bar\kappa} & ^{j+1}\cZ',
\end{array}
$$ 
and $\bar\kappa$ is a closed immersion. Let $\pi_{\cZ}: {^{j+1}\cZ'}\to \Pic X\times\Pic X$ be the map sending a point of $^{j+1}\cZ'$ to $(\det L, \det L')$. Proposition~\ref{Con_1} is reduced to the following.
\begin{Pp} 
\label{Pp_one}
For $1\le j<2n$ the natural map
$
\pi_{\cZ !}(^{j+1}K_E)\to \pi_{\cZ !}\bar\kappa_*\bar\kappa^*(^{j+1}K_E)
$
is an isomorphism.
\end{Pp}
Our idea of the proof of Proposition~\ref{Pp_one} is that it suffices to establish a similar result stratum by stratum after fixing the degrees of zeros of the sections (\ref{section_t_i_for_M}) for $1\le i\le j$. 
\sssec{} 
\label{section_3.2.4}
Our purpose now is to reduce Proposition~\ref{Pp_one} to Proposition~\ref{Pp_two} below. 

The stack $^j\cZ$ is stratified by locally closed substacks $^j\cZ^{\lambda}$ indexed by $\lambda\in\Lambda^{\succ pos}_j$. Let $j=2k$ for $j$ even (resp., $j=2k+1$ for $j$ odd). Denote by $^j\cZ^{\lambda}$ the stack classifying collections: 
\begin{itemize}
\item $L,L'\in\Bun_n$ and $(D_1,\ldots, D_j)\in X^{\lambda}_{pos}$ for which we set $M=L\oplus L'$;
\item a flag of subbundles $(0=L'_0\subset L'_1\subset \ldots\subset L'_k\subset L')$ together with trivializations
\begin{equation}
\label{trivilization_L'_i/L'_{i-1}}
\sigma'_i: L'_i/L'_{i-1}\,\iso\, \Omega^{2n-2i}(D_{2i})
\end{equation}
for $1\le i\le k$;
\item if $j=2k$ a flag of subbundles $(0=L_0\subset L_1\subset\ldots\subset L_k\subset L)$ with trivializations
\begin{equation}
\label{trivialization_L_i/L_{i-1}}
\sigma_i: L_i/L_{i-1}\,\iso\, \Omega^{2n-2i+1}(D_{2i-1})
\end{equation}
for $1\le i\le k$;
\item if $j=2k+1$ a flag of subbundles $(0=L_0\subset L_1\subset\ldots\subset L_{k+1}\subset L)$ with trivializations (\ref{trivialization_L_i/L_{i-1}})
for $1\le i\le k+1$.
\end{itemize}
The locally closed immersion $^j\cZ^{\lambda}\hook{} {^j\cZ}$ is given by the formulas: 
\begin{itemize}
\item if $1\le 2s\le j$ then $t_{2s}=\sigma_1\otimes\sigma'_1\otimes\ldots\otimes\sigma_s\otimes\sigma'_s$ is the map
$$
t_{2s}: \Omega^{(2n-1)+\ldots+(2n-2s)}\hook{} (L_1/L_0)\otimes (L'_1/L'_0)\otimes \ldots\otimes (L_s/L_{s-1})\otimes (L'_s/L'_{s-1})\hook{} \wedge^{2s} M
$$
\item if $1\le 2s+1\le j$ then $t_{2s+1}=\sigma_1\otimes\sigma'_1\otimes\ldots\otimes\sigma_{s+1}$ is the map
$$
t_{2s+1}: \Omega^{(2n-1)+\ldots+(2n-2s-1)}\hook{} (L_1/L_0)\otimes (L'_1/L'_0)\otimes \ldots\otimes (L_{s+1}/L_s)\hook{} \wedge^{2s+1}M
$$
\end{itemize}
This stratification is given by fixing the degrees of the divisors of zeros of $t_i$ for all $1\le i\le j$. 

 Denote by $^j\cZ^{\lambda}_-\hook{} {^j\cZ^{\lambda}}$ the closed substack given by $(D_1,\ldots, D_j)\in X^{\lambda}_-$. 
 
\sssec{} For $\lambda\in \Lambda^{\succ pos}_j$ consider the diagram
$$
\begin{array}{ccc}
{^{j+1}\cZ}\times_{^j\cZ}{^j\cZ^{\lambda}} & \toup{\bar\kappa^{\lambda}} & {^{j+1}\cZ'}\times_{^j\cZ}{^j\cZ^{\lambda}}\\
 & \searrow & \downarrow\lefteqn{\scriptstyle \pi_{\cZ}^{\lambda}}\\
 && \Pic X\times\Pic X,
\end{array}
$$ 
where $\bar\kappa^{\lambda}$ is the restriction of $\bar\kappa$, and $\pi_{\cZ}^{\lambda}$ is the restriction of $\pi_{\cZ}$. Note that $\bar\kappa^{\lambda}$ is a closed embedding.

Denote by $^{j+1}K^{\lambda}_E$ the $*$-restriction of $^{j+1}K_E$ to ${^{j+1}\cZ'}\times_{^j\cZ}{^j\cZ^{\lambda}}$. 
By Proposition~\ref{Pp_3.1.2}, the $*$-restriction of $\cP^d_{2n, E}$ to $S(\cY)^{\nu}_{pos}$ vanishes for all but a finite number of these strata for $\nu\in \Lambda^{\succ pos}_{2n,d}$. This reduces Proposition~\ref{Pp_one} to the following.
\begin{Pp} 
\label{Pp_two}
For $\lambda\in \Lambda^{\succ pos}_j$ the natural map 
$(\pi^{\lambda}_{\cZ})_!(^{j+1}K^{\lambda}_E)\to (\pi^{\lambda}_{\cZ})_!\bar\kappa^{\lambda}_*(\bar\kappa^{\lambda})^*(^{j+1}K^{\lambda}_E)$ is an isomorphism.
\end{Pp}

\sssec{Proof of Proposition~\ref{Pp_two}}
Let $j=2k$ for $j$ even (resp., $j=2k+1$ for $j$ odd). The stack ${^{j+1}\cZ'}\times_{^j\cZ}{^j\cZ^{\lambda}}$ classifies collections: a point of $^{j}\cZ^{\lambda}$ as described in Section~\ref{section_3.2.4}, 
\begin{itemize}
\item if $j=2k$ then we are also given a section 
$$
s: \Omega^{2n-2k-1}\to (M/(L_k\oplus L'_k))(D_1+\ldots+D_{2k}),
$$ 
for which we define $t_{j+1}$ as the map
$$
t_j\otimes s: \Omega^{(2n-1)+\ldots+(2n-2k-1)}\to \det(L_k\oplus L'_k)\otimes (M/(L_k\oplus L'_k))\subset \wedge^{j+1}M
$$
\item if $j=2k+1$ then we are given a section 
$$
s: \Omega^{2n-2k-2}\to (M/(L_{k+1}\oplus L'_k))(D_1+\ldots+D_{2k+1}),
$$
and we define $t_{j+1}$ as the map
$$
t_j\otimes s: \Omega^{(2n-1)+\ldots+(2n-2k-2)}\to \det(L_{k+1}\oplus L'_k)\otimes (M/(L_{k+1}\oplus L'_k))\subset \wedge^{j+1}M
$$
\end{itemize}

Let 
$$
^{j+1}\bar\cZ^{'\lambda}\hook{} {^{j+1}\cZ'}\times_{^j\cZ}{^j\cZ^{\lambda}}
$$ 
be the closed substack given by the propeties that $D_1\le \ldots\le D_j$, and 
\begin{itemize}
\item if $j=2k$ then $s: \Omega^{2n-2k-1}(D_{2k})\to M/(L_k\oplus L'_k)$ is regular
\item if $j=2k+1$ then $s: \Omega^{2n-2k-2}(D_{2k+1})\to M/(L_{k+1}\oplus L'_k)$ is regular.
\end{itemize}

 By Lemma~\ref{Lm_P_E_on_m_j+1_cY}, one knows that $^{j+1}K^{\lambda}_E$ is the extension by zero from $^{j+1}\bar\cZ^{'\lambda}$. Let $^{j+1}\bar\cZ^{\lambda}$ be the preimage of $^{j+1}\bar\cZ^{'\lambda}$ under $\bar\kappa^{\lambda}$. So, in Proposition~\ref{Pp_two} we actually deal with the diagram
$$
\begin{array}{ccc}
^{j+1}\bar\cZ^{\lambda} & \toup{\bar\kappa^{\lambda}} & ^{j+1}\bar\cZ^{'\lambda}\\
 & \searrow & \downarrow\lefteqn{\scriptstyle \pi_{\cZ}^{\lambda}}\\
 && \Pic X\times\Pic X,
\end{array}
$$   
By abuse of notations, here $\pi_{\cZ}^{\lambda}$ and $\bar\kappa^{\lambda}$ denote the restrictions of the corresponding maps. 

  The descripion of $^{j+1}K^{\lambda}_E$ on $^{j+1}\bar\cZ^{'\lambda}$ is as follows. Let $^{j+1}\bar\cW^{\lambda}$ be the stack classifying $(D_1,\ldots, D_j)\in X^{\lambda}_-$, 
\begin{itemize}
\item if $j=2k$ then we are given $\bar L,\bar L'\in\Bun_{n-k}$ and $s:\Omega^{2n-2k-1}(D_{2k})\hook{} \bar L\oplus\bar L'$;
\item if $j=2k+1$ then we are given $\bar L\in \Bun_{n-k-1}, \bar L'\in\Bun_{n-k}$ and 
$$
s: \Omega^{2n-2k-2}(D_{2k+1})\hook{} \bar L\oplus \bar L'
$$
\end{itemize}
We get a map $\bar\tau: {^{j+1}\bar\cZ^{'\lambda}}\to {^{j+1}\bar\cW^{\lambda}}$ given by the formulas $\bar L'=L'/L'_k$ and
$$
\bar L=\left\{
\begin{array}{ll}
L/L_k, & \mbox{if}\;\; j=2k\\
L/L_{k+1}, & \mbox{if}\;\; j=2k+1
\end{array}
\right.
$$

 Let $\ev^{\lambda}: {^{j+1}\bar\cZ^{'\lambda}}\to\A^1$ be the map sending a point of $^{j+1}\bar\cZ^{'\lambda}$ to the class of the pullback extension
$$
\begin{array}{ccc}
0\to L'_k/L'_{k-1}\to L'/L'_{k-1}\to  & L'/L'_k & \to 0\\
& \uparrow\lefteqn{\scriptstyle \bar s'} \\
& \Omega^{2n-2k-1}(D_{2k})
\end{array}
$$
for $j=2k$, and the class of the pullback extension
$$
\begin{array}{ccc}
0\to L_{k+1}/L_k\to L/L_k\to  & L/L_{k+1} & \to 0\\
& \uparrow\lefteqn{\scriptstyle \bar s}\\
& \Omega^{2n-2k-2}(D_{2k+1})
\end{array}
$$
for $j=2k+1$ respectively. Here $\bar s, \bar s'$ are the corresponding components of $s$. 
\begin{Lm} There is a complex $^{j+1}\bar K^{\lambda}_E$ on $^{j+1}\bar\cW^{\lambda}$ and an isomorphism over $^{j+1}\bar\cZ^{'\lambda}$
\begin{equation}
^{j+1}K^{\lambda}_E\,\iso\, \bar\tau^*(^{j+1}\bar K^{\lambda}_E)\otimes (\ev^{\lambda})^*\cL_{\psi}
\end{equation}
\end{Lm}
\begin{proof} This follows from Lemma~\ref{Lm_P_E_on_m_j+1_cY}. In more details, recall the stack $_{j+1}\cY^{\lambda}_{2n,-}$ from Section~\ref{Sect_3.1.4}. We have the map $^{j+1}\bar\cZ^{'\lambda}\to {_{j+1}\cY^{\lambda}_{2n,-}}$ sending a point of the source to the collection 
$$
(M=L\oplus L', (t_1,\ldots, t_{j+1}), (D_1,\ldots, D_j)\in X^{\lambda}_-)
$$ 
So, $M$ is equipped with the flag $0=M_0\subset M_1\subset\ldots\subset M_j\subset M$ of subbundles. Here $M_{2k}=L_k\oplus L'_k$ for $0\le 2k\le j$ 
and $M_{2k+1}=L_{k+1}\oplus L'_k$ for $0<2k+1\le j$. Our claim follows from the fact that, by construction, the exact sequence $0\to M_i/M_{i-1}\to M_{i+1}/M_i\to M_{i+1}/M_i\to 0$ splits canonically for $1\le i\le j-1$.
\end{proof}

The map $\pi_{\cZ}^{\lambda}$ factors as 
$$
^{j+1}\bar\cZ^{'\lambda}\toup{\bar\tau} {^{j+1}\bar\cW^{\lambda}}\to \Pic X\times \Pic X
$$ 
Let $^{j+1}\cW^{\lambda}\hook{} {^{j+1}\bar\cW^{\lambda}}$ be the closed substack given by the property
\begin{itemize}
\item if $j=2k$ then $s$ factors as $s: \Omega^{(2n-2k-1)}(D_{2k})\hook{} \bar L\subset \bar L\oplus \bar L'$;
\item if $j=2k+1$ then $s$ factors as $s: \Omega^{(2n-2k-2)}(D_{2k+1})\hook{} \bar L'\subset\bar L\oplus\bar L'$.
\end{itemize}
The square is cartesian
$$
\begin{array}{ccc}
^{j+1}\cW^{\lambda} & \hook{} &{^{j+1}\bar\cW^{\lambda}}\\
\uparrow && \uparrow\lefteqn{\scriptstyle\bar\tau}\\
^{j+1}\bar\cZ^{\lambda} & \toup{\bar\kappa^{\lambda}} & ^{j+1}\bar\cZ^{'\lambda}
\end{array}
$$
So, it suffices to show that $\bar\tau_!(\ev^{\lambda})^*\cL_{\psi}$ is the extension by zero from $^{j+1}\cW^{\lambda}$. The map $\bar\tau$ is a generalized affine fibration. Let $^{j+1}{\hat \cW}^{\lambda}$ be the stack classifying a point of $^{j+1}\bar\cW^{\lambda}$ as above together with an exact sequence 
\begin{equation}
\label{seq_for_proof_of_3.2.7_even}
0\to \Omega^{2n-2k}(D_{2k})\to ?\to \bar L'\to 0
\end{equation}
for $j=2k$, respectively with an exact sequence
\begin{equation}
\label{seq_for_proof_of_3.2.7_odd}
0\to \Omega^{2n-2k-1}(D_{2k+1})\to ?\to \bar L\to 0
\end{equation}
for $j=2k+1$. The map $\bar\tau$ decomposes naturally as 
$$
^{j+1}\bar\cZ^{'\lambda}\toup{\xi} {^{j+1}{\hat\cW}^{\lambda}}\toup{\hat\tau}{^{j+1}\bar\cW^{\lambda}}, 
$$
where $\xi$ is a generalized affine fibration. Namely, the sequence (\ref{seq_for_proof_of_3.2.7_even}) for $j=2k$ is 
$$
0\to L'_k/L'_{k-1}\to L'/L'_{k-1}\to L'/L'_k\to 0
$$ 
The sequence (\ref{seq_for_proof_of_3.2.7_odd}) for $j=2k+1$ is 
$$
0\to L_{k+1}/L_k\to L/L_k\to L/L_{k+1}\to 0
$$ 
Here $\hat \tau$ forgets the extension (\ref{seq_for_proof_of_3.2.7_even}) for $j=2k$ (respectively, the extension (\ref{seq_for_proof_of_3.2.7_odd}) for $j=2k+1$). 

 The map $\ev^{\lambda}$ factors naturally as $^{j+1}\bar\cZ^{'\lambda}\to {^{j+1}{\hat\cW}^{\lambda}}\toup{\hat\ev^{\lambda}}\A^1$, and $\xi_!\Qlb$ identifies with $\Qlb$ up to a shift. The $*$-fibre of $\hat\tau_!(\hat\ev^{\lambda})^*\cL_{\psi}$ at a given point vanishes unless $\bar s'=0$ for $j=2k$ (respectively, $\bar s=0$ for $j=2k+1$). Proposition~\ref{Pp_two} is proved. \QED

\medskip

 Thus, Propositions~\ref{Pp_one}, \ref{Con_1} are also proved.   
 
\ssec{Passing to $\wt\cY_n\times{\wt\cY'_n}$}
\label{Sect_3.3}

\sssec{} In Section~\ref{Sect_3.3} we finish the proof of  Proposition~\ref{Pp_2.2.4} using the results of Sections~\ref{Sect_3.1}-\ref{Sect_passing_2n_YZ}.  

\sssec{}  Note that $^{2n}\cZ={^{2n}\cY\cZ}$. For $d,d'\in\ZZ$ let $^{2n}\cZ^{d,d'}$ be the substack of $^{2n}\cZ$ given by 
$$
\deg L=\deg(\Omega^{(2n-1)+(2n-3)+\ldots+1})+d, \;\; \deg L'=\deg(\Omega^{(2n-2)+(2n-4)+\ldots+2})+d'
$$
The map $q_{\cY\cZ}: {\wt\cY_{n,d}}\times{\wt\cY'_{n,d'}}\to {\cY_{2n}}\times_{\cM_{2n}} \cZ_n$ factors uniquely through the closed substack
$$
{\wt\cY_{n,d}}\times{\wt\cY'_{n,d'}}\;\toup{i_{\cY}} \;{^{2n}\cZ^{d,d'}}\;\hook{^{2n}\kappa} \; {\cY_{2n}}\times_{\cM_{2n}} \cZ_n,
$$
this defines the map $i_{\cY}$.  

\sssec{} The stack $^{2n}\cZ$ admits a stratification by $^{2n}\cZ^{\lambda}$ indexed by $\lambda\in \Lambda^{\succ pos}_{2n}$, it is described in Section~\ref{section_3.2.4}. For $d,d'\in\ZZ$ the stack $^{2n}\cZ^{d,d'}$ admits a stratification by $^{2n}\cZ^{\lambda}$ indexed by $\lambda\in \Lambda^{\succ pos}_{2n}$ such that $\lambda_1+\lambda_3+\ldots+\lambda_{2n-1}=d$ and $\lambda_2+\ldots+\lambda_{2n}=d'$.

 For $\lambda\in \Lambda^{\succ-}_{2n}$ let 
$$
q_{\lambda}: {^{2n}\cZ^{\lambda}_-}\to X^{\lambda}_-
$$ 
be the map sending a point of $^{2n}\cZ^{\lambda}_-$ to $(D_1,\ldots, D_{2n})\in X^{\lambda}_-$. From Proposition~\ref{Pp_3.1.2} one gets the following.
\begin{Cor}
\label{Pp_three}
Let $\lambda\in \Lambda^{\succ pos}_{2n, r}$. The $*$-restriction of $(^{2n}\kappa)^*(\cP_{2n, E}^r\boxtimes\Qlb)$ to the stratum $^{2n}\cZ^{\lambda}$ vanishes unless $\lambda\in 
\Lambda^{\succ-}_{2n, r}$. In the latter case it is the extension by zero from $^{2n}\cZ^{\lambda}_-$ and identifies over $^{2n}\cZ^{\lambda}_-$ with 
$$
q_{\lambda}^*E^{\lambda}_-[b(2n, r)-2a_{2n}(\lambda)]
$$  
\end{Cor}

\sssec{} For a scheme $S$ and a vector bundle $V$ on $S\times X$, by a \select{rational section} of $V$ we mean a section defined over an open subscheme of the form $(S\times X)-D$, where $D\hook{} S\times X$ is a closed subscheme finite over $S$. 
\begin{Lm} 
\label{Lm_3.3.4}
A point of $^{2n}\cZ$ gives rise to uniquely defined rational sections 
\begin{equation}
\label{section_s_i_rational}
s_i: \Omega^{(2n-1)+\ldots+(2n-2i+1)}\hook{} \wedge^i L
\end{equation}
and
\begin{equation}
\label{section_s'_i_rational}
s'_i: \Omega^{(2n-2)+\ldots+(2n-2i)}\hook{} \wedge^i L'
\end{equation}
for $1\le i\le n$ such that (\ref{t_i_as_a_function_of_s_and_s'}) hold. 
\end{Lm}
\begin{Prf}
A point of of $^{2n}\cZ$ gives rise to non uniquely defined rational sections $u_i: \Omega^{2n-2i+1}\hook{} L$ and $u'_i: \Omega^{2n-2i}\hook{} L'$ for $1\le i\le n$ such that $t_{2k}=u_1\wedge u'_1\wedge\ldots\wedge u_k\wedge u'_k$ and 
$$
t_{2k+1}=u_1\wedge u'_1\wedge\ldots\wedge u_k\wedge u'_k\wedge u_{k+1}
$$ 
Set $s_k=u_1\wedge\ldots\wedge u_k$ and $s'_k=u'_1\wedge\ldots\wedge u'_k$. One checks that $s_k$ and $s'_k$ are independent of the choices of $u_i$, $u'_i$.
\end{Prf}
 
\sssec{Proof of Proposition~\ref{Pp_2.2.4}}  
\label{Sect_Proof_of_Pp2.2.4}
From Lemma~\ref{Lm_3.3.4} it follows that $i_{\cY}$ is a closed immersion, it is given by the property that the rational sections  (\ref{section_s_i_rational}) and (\ref{section_s'_i_rational}) are regular over the whole of $X$. 

 More precisely, this is seen as follows. For a point $(L, L', (t_i))\in {^{2n}}\cZ$ consider the exact sequence $0\to \Omega^{2n-1}\toup{t_1} L\to L/\Omega^{2n-1}\to 0$. Set $s_1=t_1$. As the first step one imposes the condition that the image of $t_2$ under the induced map 
$$
L\otimes L'\to (L/\Omega^{2n-1})\otimes L'
$$ 
vanishes, this yields a section $s_1':\Omega^{2n-2}\hook{} L'$ with $t_2=s_1\otimes s'_1$. Consider the exact sequence $0\to \Omega^{2n-2}\toup{s'_1} L'\to L'/\Omega^{2n-2}\to 0$. As the second step one imposes the condition that the image of $t_3$ under the induced map 
$$
(\wedge^2 L)\otimes L'\to (\wedge^2 L)\otimes (L'/\Omega^{2n-2})
$$ 
vanishes. This yields a section $s_2: \Omega^{(2n-1)+(2n-3)}\hook{} \wedge^2 L$ such that $t_3=-s_2\otimes s'_1$. Continuing this procedure, by induction one gets all the sections $s_i, s'_i$ by imposing similar closed conditions.

 From Corollary~\ref{Pp_three} it follows that $(^{2n}\kappa)^*(\cP_{2n, E}\boxtimes\Qlb)$ over the component $^{2n}\cZ^{d,d'}$ of $^{2n}\cZ$ is the extension by zero under the map $i_{\cY}$. Indeed, for each $\lambda\in\Lambda^{\succ pos}_{2n}$ such that $\lambda_1+\lambda_3+\ldots+\lambda_{2n-1}=d$ and $\lambda_2+\ldots+\lambda_{2n}=d'$, we have 
$$
^{2n}\cZ^{\lambda}_- \subset ({\wt\cY_{n,d}}\times{\wt\cY'_{n, d'}})\times_{(^{2n}\cZ^{d,d'})} {^{2n}\cZ^{\lambda}}
$$
Indeed, if a point of $^{2n}\cZ^{\lambda}_-$ is written as in Section~\ref{section_3.2.4} then the divisor of zeros of (\ref{section_s_i_rational}) is $D_1+D_3+\ldots+D_{2i-1}\ge 0$, the divisor of zeros of (\ref{section_s'_i_rational}) is $D_2+D_4+\ldots D_{2i}\ge 0$. Our claim follows now from Proposition~\ref{Con_1}.
\QED 

\begin{Rem} The map $i_{\cY}: {\wt\cY_{n,d}}\times{\wt\cY'_{n, d'}}\to {^{2n}\cZ^{d,d'}}$ is not an isomorphism already for $n=1$. In this case, 
the condition (A) is empty, and 
$^2\cZ^{d,d'}$ classifies: $L, L'\in\Bun_1$ with $\deg L=\deg\Omega+d, \deg L'=d'$ and sections 
$$
t_1: \Omega\hook{} L, t_2: \Omega\hook{} L\otimes L'
$$ 
The closed immersion $i_{\cY}$ is given in this case by requiring that there is $s'_1: \cO\hook{} L'$ such that $t_2=t_1\otimes s'_1$. 
\end{Rem}

\section{Direct image under $\bar\pi$}
\label{Section_Direct_image_barpi}

\noindent
The purpose of Section~\ref{Section_Direct_image_barpi}
is to prove Proposition~\ref{Pp_2.2.5}, which is a part of our strategy presented in Section~\ref{Sect_2.3}. We start with Proposition~\ref{Pp_graded_result} showing that the complex (\ref{complex_Pp_2.2.5}) that we must calculate is placed in the highest usual cohomological degree only. It also provides a filtration on the corresponding constructible sheaf whose associated graded is explicitely described. The corresponding highest direct image is calculated in two steps. 

 The first step is Proposition~\ref{Pp_1.2.10}. Namely, we get a diagram 
$$ 
{\wt\cQ_{n,d}}\times{\wt\cQ'_{n, d'}}\;\toup{\tilde\varphi\times\tilde\varphi'}\; {\wt\cY_{n,d}}\times{\wt\cY'_{n, d'}}\;\toup{\bar\pi} \; X^{(d)}\times X^{(d')}
$$
and replace the highest direct image under $\bar\pi$ by the highest direct image under $\bar\pi\comp (\tilde\varphi\times\tilde\varphi')$. 
This is obtained by a reduction to some local statement with respect to $X$. We generalize the corresponding local result for an arbitrary reductive group and a Levi subgroup that we call \select{antistandard}. This generalization is perfomed in Appendix A independent of the rest of the paper. We consider this generalization as an important result of independent interest (Propositions~\ref{Pp_2.1.1} and \ref{Pp_2.1.15}). 

The second step is Proposition~\ref{Pp_4.2.8}. The combination of both steps easily implies Proposition~\ref{Pp_2.2.5}. In turn, Proposition~\ref{Pp_4.2.8} is reduced to Theorem~\ref{Thm_4.2.11}, whose proof is the purpose of  Section~\ref{Section_proof_of_Thm_4.2.11}. 

 The mathematical content of Proposition~\ref{Pp_4.2.8} and Theorem~\ref{Thm_4.2.11} is to decompose $\bar\pi\comp (\tilde\varphi\times\tilde\varphi')$ as 
$$
{\wt\cQ_{n,d}}\times{\wt\cQ'_{n,d'}}\;\toup{\tilde\beta\times\tilde\beta'} \;
{^{\le n}\Sh^d_0}\times{^{\le n}\Sh^{d'}_0}\;\toup{^{\le n}(\div\times\div)}\; X^{(d)}\times X^{(d')}
$$
and reduce the calculation of the desired highest direct image under $\bar\pi\comp (\tilde\varphi\times\tilde\varphi')$ to the calculation of the highest direct image under $^{\le n}(\div\times\div)$ over the closed subscheme 
$i_X: X^{(d)}\times X^{(d'-d)}\hook{} X^{(d)}\times X^{(d')}$ given by (\ref{map_i_X}). The last calculation given in Theorem~\ref{Thm_4.2.11} is of local nature with respect to $X$.
We consider Theorem~\ref{Thm_4.2.11} as one of our main results, which is of independent interest. 

\ssec{Filtration}
\label{Sect_Filtration}

In Section~\ref{Sect_Filtration} we prove Proposition~\ref{Pp_graded_result}. It shows that (\ref{complex_Pp_2.2.5}) is a constructible sheaf (up to a shift) and calculates this sheaf in the Grothendieck group.

\sssec{} Recall the map (\ref{map_i_X}) from Section~\ref{Sect_2.1.1}. From Corollary~\ref{Pp_three} it follows that 
$$
\bar\pi_!q_{\cY\cZ}^*(\cP_{2n, E}\boxtimes\Qlb)
$$ 
is the extension by zero under $i_X$. In particular, it vanishes unless $0\le d\le d'$. Indeed, this holds for the contribution of each stratum $^{2n}\cZ^{\lambda}$ of $^{2n}\cZ^{d,d'}$. 

 For $\lambda\in\Lambda^{\succ pos}_{2n}$ recall the notation $\lambda^{odd}$, $\lambda^{even}$ and the map $\ssum^{\lambda}: X^{\lambda}_-\to X^{(d)}\times X^{(d')}$ from Sections~
\ref{Sect_2.4.1_revision}, \ref{Sect_2.3.3}. 
 
\begin{Pp} 
\label{Pp_graded_result}
The complex 
\begin{equation}
\label{complex_Pp_4.1.2}
\bar\pi_!q_{\cY\cZ}^*(\cP^{d+d'}_{2n, E}\boxtimes\Qlb)[\dimrel(\nu_{\cZ})-d']
\end{equation}
on $X^{(d)}\times X^{(d')}$ is a constructible sheaf placed in usual cohomological degree zero, it is an extension by zero under $i_X$. It admits a filtration with the associated graded
\begin{equation}
\label{complex_graded_by_lambda}
\mathop{\oplus}\limits_{\lambda} \;\ssum^{\lambda}_!E^{\lambda}_-,
\end{equation}
the sum being taken over $\lambda\in \Lambda^{\succ-}_{2n,d+d'}$ such that $\lambda^{odd}\in \Lambda^{\succ-}_{n,d}$ and $\lambda^{even}\in\Lambda^{\succ-}_{n,d'}$.
\end{Pp}
\begin{Prf}
For each $\lambda\in \Lambda^{\succ pos}_{2n}$ the projection $^{2n}\cZ^{\lambda}\to X^{\lambda}_{pos}$ is a generalized affine fibration of rank
\begin{equation}
\label{rank_of_q_lambda}
nd+(n-1)d'-a_{2n}(\lambda)-n(n-1)(g-1)-(4g-4)\sum_{i=1}^{n-1} i^2,
\end{equation}
The same holds for $q_{\lambda}: {^{2n}\cZ^{\lambda}_-}\to X^{\lambda}_-$. Since
$\sum_{i=1}^n i^2=\frac{n}{6}(n+1)(2n+1)$, from (\ref{def_of_b(n,r)})
one gets 
$$
b(2n, d+d')=2n(d+d')+(1-g)\frac{(2n-1)}{3}n(4n-1)
$$

 Using Corollary~\ref{Pp_three}, calculate (\ref{complex_Pp_4.1.2}) with respect to the stratification of $^{2n}\cZ^{d, d'}$ by the substacks $^{2n}\cZ^{\lambda}$. The strata that contribute are $^{2n}\cZ^{\lambda}_-$ for $\lambda\in \Lambda^{\succ-}_{2n,d+d'}$ such that $\lambda^{odd}\in \Lambda^{\succ-}_{n,d}$ and $\lambda^{even}\in\Lambda^{\succ-}_{n,d'}$. For such $\lambda$ we get
$$
q_{\lambda !}q_{\lambda}^*E^{\lambda}_-[b(2n, d+d')-2a_{2n}(\lambda)]\,\iso\, E^{\lambda}_-[2d'+(1-g)n],
$$
the shift in the right hand side is independent of $\lambda$. 

Now calculate the relative dimension of $\nu_{\cZ}$. One has $\dim(\cM_{2n,r})=r$. Let $\cZ_n^{d,d'}$ be the component of $\cZ_n$ given by 
$$
\deg L=d+\deg(\Omega^{(2n-1)+(2n-3)+\ldots+1}),\;\; \deg L'=d'+\deg(\Omega^{(2n-2)+(2n-4)+\ldots+0})
$$ 
Since $\dim \cZ_n^{d,d'}=d+n(g-1)$, for 
$$
\nu_{\cZ}: \cZ_n^{d,d'}\to {\cM_{2n, d+d'}}
$$ 
we get $\dimrel(\nu_{\cZ})=n(g-1)-d'$. So, the contribution of $^{2n}\cZ^{\lambda}$ to (\ref{complex_Pp_4.1.2}) is
$
\ssum^{\lambda}_!E^{\lambda}_-
$.
\end{Prf}

\begin{Rem} i) In the case $d=0$ the filtration of Proposition~\ref{Pp_graded_result} on (\ref{complex_Pp_4.1.2}) has only one term, and Proposition~\ref{Pp_2.2.5} follows from Proposition~\ref{Pp_graded_result}. Indeed, for $\lambda=(0,\ldots, 0, d')$ one has $E^{\lambda}_-=E^{(d')}$ on $X^{\lambda}_-=X^{(d')}$.  

\smallskip\noindent
ii) In view of Proposition~\ref{Pp_graded_result}, to prove Proposition~\ref{Pp_2.2.5} for $\rk(E)=2n$, it suffices to show that the left hand side of (\ref{complex_Pp_2.2.5}) is the intermediate extension of its restriction to any nonempty open subscheme of $X^{(d)}\times X^{(d'-d)}$. Indeed, if $(D,D')\in X^{(d)}\times X^{(d')}$ with $D'\ge D$, $D$ is reduced, $D'-D$ is reduced, and $D\cap (D'-D)=\emptyset$ then only one $\lambda$ as in Proposition~\ref{Pp_graded_result} contributes, namely 
$\lambda=(0,\ldots, 0, d, d')$. 
The restriction of $E^{\lambda}_-$ to the corresponding open subscheme of 
$$
i_X(X^{(d)}\times X^{(d'-d)})
$$ 
identifies with $(\wedge^2 E)^{(d)}\boxtimes E^{(d'-d)}$. 
\end{Rem}

\ssec{Highest direct image: first step}
\label{Sect_Highest direct image}

The purpose of Section~\ref{Sect_Highest direct image} is to prove Proposition~\ref{Pp_1.2.10} below. The plan of its proof is as follows. Stratifying the stack $\wt\cY_{n,d}\times{\wt\cY'_{n, d'}}$ naturally, our claim is easily reduced to Proposition~\ref{Pp_4.2.7}. The latter, in turn, is easily reduced to Proposition~\ref{Pp_4.2.10}, which is of local nature. Finally, Proposition~\ref{Pp_4.2.10} is obtained as a particular case of a more general Proposition~\ref{Pp_2.1.15}
valid for any reductive group $G$ and its antistandard Levi subgroup $M$. The latter is established in Appendix~\ref{Section_AppendixA} independent of the rest of the paper. 

\sssec{} For $d\ge 0$ let $\wt\cQ_{n, d}$ be the stack classifying: a modification of rank $n$ vector bundles $L_n\subset L$ on $X$ with $\deg(L/L_n)=d$, a complete flag of vector subbundles $0=L_1\subset\ldots\subset L_n$, where $L_i$ is a rank $i$ vector bundle, and isomorphisms
$$
\sigma_i: L_i/L_{i-1}\,\iso\, \Omega^{2n-2i+1}
$$ 
for $1\le i\le n$. We have a morphism $\tilde\varphi: {\wt\cQ_{n,d}}\to {\wt\cY_{n, d}}$ sending the above point to $(L, (s_i))$, where 
$$
s_i: \Omega^{(2n-1)+\ldots+(2n-2i+1)}\,\iso\, \wedge^i L_i \hook{} \wedge^i L
$$ 
for $1\le i\le n$.

 Similarly, for $d\ge 0$ let $\wt\cQ'_{n,d}$ be the stack classifying: a modification of rank $n$ vector bundles $L'_n\subset L'$ on $X$ with $\deg(L'/L'_n)=d$, a complete flag $(0=L'_0\subset L'_1\subset\ldots\subset L'_n)$ of vector subbundles on $L'_n$ with isomorphisms
$$
\sigma'_i: L'_i/L'_{i-1}\,\iso\, \Omega^{2n-2i}
$$ 
for $1\le i\le n$. We have a morphism $\tilde\varphi': {\wt\cQ'_{n,d}}\to {\wt\cY'_{n,d}}$ sending the above point to $(L', (s'_i))$, where 
$$
s'_i: \Omega^{(2n-2)+\ldots+(2n-2i)}\,\iso\, \wedge^i L'_i \hook{} \wedge^i L'
$$ 
for $1\le i\le n$.

Denote by 
$$
\eta: {\wt\cQ_{n,d}}\times {\wt\cQ'_{n,d'}}\to {\cQ_{2n, d+d'}}
$$ 
the map sending $(L, (L_i), \sigma_i), (L', (L'_i), \sigma'_i)$ to $(M, (M_i), (s_i))$, where $M=L\oplus L'$, $0=M_0\subset M_1\subset\ldots\subset M_{2n}$ is a complete flag of vector bundles on $X$ given by 
$$
\left\{ 
\begin{array}{ll} 
M_{2k+1}=L_{k+1}\oplus L'_k, & 0\le k<n\\
M_{2k}=L_k\oplus L'_k, & 1\le k\le n
\end{array}
\right.
$$
with the induced trivializations $s_i: M_i/M_{i-1}\,\iso\,\Omega^{2n-i}$ for $1\le i\le 2n$. 

\sssec{} Recall the map $\varphi: {\cQ_{n, r}}\to{\cY_{n,r}}$ from Section~\ref{Sect_3.1.1}. We define the commutative diagram
\begin{equation}
\label{diag_for_Sect_4.2.2}
\begin{array}{ccccc}
\cY_{2n, d+d'}&\getsup{\iota} & {\wt\cY_{n,d}}\times{\wt\cY'_{n, d'}}\\
\uparrow\lefteqn{\scriptstyle \varphi} && \uparrow & \searrow\lefteqn{\scriptstyle \bar\pi}\\
\cQ_{2n, d+d'} & \getsup{\iota_{\cT}}& \cT^{d, d'} &\toup{\bar\pi^{d,d'}}& X^{(d)}\times X^{(d')}\\
& \nwarrow\lefteqn{\scriptstyle\eta} & \uparrow\lefteqn{\scriptstyle \bar\eta} \\
&& \wt\cQ_{n,d}\times{\wt\cQ'_{n,d'}} 
\end{array} 
\end{equation}
as follows. We require the top left square in this diagram to be cartesian, so defining the stack $\cT^{d, d'}$ and maps $\iota_{\cT}, \bar\pi^{d,d'}$. We define $\bar\eta$ as the morphism whose composition with $\iota_{\cT}$ is $\eta$, and whose composition with $\cT^{d,d'}\to {\wt\cY_{n, d}}\times{\wt\cY'_{n, d'}}$ is $\tilde\varphi\times\tilde\varphi'$. 
\begin{Rem} The map $\bar\eta$ is a closed immersion. Indeed, let us be given a point of $\cT^{d,d'}$, which is a collection $(L, (s_i))\in {\tilde\cY_{n,d}}, (L', (s'_i))\in {\tilde\cY'_{n, d'}}$ and a flag 
$$
(M_1\subset\ldots\subset M_{2n}\subset L\oplus L')
$$ 
with trivializations $M_i/M_{i-1}\,\iso\, \Omega^{2n-i}$ for $1\le i\le 2n$ such that the corresponding maps $\wedge^i M_i\hook{} \wedge^i(L\oplus L')$ coincide with $t_i$ defined by (\ref{t_i_as_a_function_of_s_and_s'}). 

 Assume by induction on $k$ we have already constructed $(L_1\subset\ldots\subset L_k\subset L)$ and $(L'_1\subset\ldots\subset L'_k\subset L')$ with $M_{2i}=L_i\oplus L'_i$ for $0\le 2i\le 2k$ and $M_{2i+1}=L_{i+1}\oplus L_i$ for $0\le 2i+1< 2k$. Then our point of $\cT^{d,d'}$ gives rise to a section $\Omega^{2n-2k-1}\hook{}(L/L_k)\oplus (L'/L'_k)$. The condition that its component $\Omega^{2n-2k-1}\to L'/L'_k$ vanishes is closed, and gives rise to a subsheaf $L_{k+1}\subset L$. One gets a subsheaf $L'_{k+1}\subset L'$ similarly and then continues by induction on $k$.
\end{Rem}

 The following is the main result of Section~\ref{Sect_Highest direct image}. Its proof is given in Section~\ref{Section_4.2.8}.

\begin{Pp} 
\label{Pp_1.2.10}
For any local system $E$ on $X$ the natural map
\begin{equation}
\label{map_Pp_1.2.10}
\bar\pi^{d,d'}_!\iota_{\cT}^*(\cF^{d+d'}_{2n,E})\to \bar\pi^{d,d'}_!
\bar\eta_*\bar\eta^*\iota_{\cT}^*(\cF^{d+d'}_{2n, E})
\end{equation}
induces an isomorphism on the usual cohomology sheaves in degree $n(g-1)-2d'$. 
\end{Pp} 

\begin{Rem} By Proposition~\ref{Pp_graded_result}, the left hand side of  (\ref{map_Pp_1.2.10}) is placed in usual cohomological degree $n(g-1)-2d'$ only.
\end{Rem}

\sssec{} For $\lambda\in\Lambda^{\succ pos}_{2n, d+d'}$ with $\sum_{i=1}^n \lambda_{2i-1}=d$ and $\sum_{i=1}^n\lambda_{2i}=d'$ let $\cT^{\lambda}$ denote the preimage of $^{2n}\cZ^{\lambda}$ under the composition 
$$
\cT^{d,d'}\to\, {\wt\cY_{n, d}}\times{\wt\cY'_{n, d'}}\;\toup{i_{\cY}}\; {^{2n}\cZ^{d,d'}}
$$ 
Let $\bar\eta^{\lambda}: \wt\cQ\wt\cQ'^{\lambda}\hook{} \cT^{\lambda}$ be obtained from $\bar\eta$ by the base change $\cT^{\lambda}\to \cT^{d,d'}$. Let $\bar\pi^{\lambda}: \cT^{\lambda}\to X^{\lambda}_{pos}$ be the composition $\cT^{\lambda}\to {^{2n}\cZ^{\lambda}}\to X^{\lambda}_{pos}$. Clearly, $\bar\pi^{\lambda}$ factors through the closed subscheme $X^{\lambda}\hook{} X^{\lambda}_{pos}$. 

 Write $\iota^{\lambda}_{\cT}: \cT^{\lambda}\to {\cQ_{2n, d+d'}}$ for  the restriction of $\iota_{\cT}$. We get the diagram
$$
\begin{array}{ccccc}
\cQ_{2n, d+d'} & \getsup{\iota^{\lambda}_{\cT}} & \cT^{\lambda} & \toup{\bar\pi^{\lambda}} & X^{\lambda}\\
&& \uparrow\lefteqn{\scriptstyle \bar\eta^{\lambda}}\\
&& \wt\cQ\wt\cQ'^{\lambda}
\end{array}
$$

Denote by 
$$
\varphi^{\lambda}: \cT^{\lambda}\to {^{2n}\cZ^{\lambda}}\times_{X^{\lambda}_{pos}} X^{\lambda}
$$ 
the map obtained from $\varphi: {\cQ_{2n, d+d'}}\to {\cY_{2n, d+d'}}$ by the base change ${^{2n}\cZ^{\lambda}}\times_{X^{\lambda}_{pos}} X^{\lambda}\to {^{2n}\cZ^{\lambda}}\to{\cY_{2n, d+d'}}$. The map $\varphi^{\lambda}$ is a generalized affine fibration. 
 
We derive Proposition~\ref{Pp_1.2.10} from to the following result, whose proof is given in Section~\ref{Section_4.2.9}.
\begin{Pp} 
\label{Pp_4.2.7}
For $\lambda\in \Lambda^{\succ}_{2n, d+d'}$ the natural map 
$$
\varphi^{\lambda}_!(\iota^{\lambda}_{\cT})^*(\cF^{d+d'}_{2n, E})\to 
\varphi^{\lambda}_!(\bar\eta^{\lambda})_!
(\bar\eta^{\lambda})^*(\iota^{\lambda}_{\cT})^*(\cF^{d+d'}_{2n,E})
$$
induces an isomorphism on the (highest) cohomology sheaves in the usual degree 
$$
2a_{2n}(\lambda)-b(2n, d+d')
$$ 
Both complexes are placed in the usual cohomological degrees $\le 2a_{2n}(\lambda)-b(2n, d+d')$.
\end{Pp} 

\sssec{Proof of Proposition~\ref{Pp_1.2.10}} 
\label{Section_4.2.8}
It suffices to show that for any $\lambda\in\Lambda^{\succ}_{2n, d+d'}$ with $\sum_{i=1}^n \lambda_{2i-1}=d$ and $\sum_{i=1}^n\lambda_{2i}=d'$
the natural map
$$
\bar\pi^{\lambda}_!(\iota_{\cT}^{\lambda})^*(\cF^{d+d'}_{2n, E})\to 
\bar\pi^{\lambda}_!\bar\eta^{\lambda}_*(\bar\eta^{\lambda})^*
(\iota_{\cT}^{\lambda})^*(\cF^{d+d'}_{2n, E})
$$
induces an isomorphism on the usual cohomology sheaves in degree $n(g-1)-2d'$. Since $q_{\lambda}: {^{2n}\cZ^{\lambda}_-}\to X^{\lambda}_-$ is a generalized affine fibration of rank (\ref{rank_of_q_lambda}), our claim follows from Proposition~\ref{Pp_4.2.7}. \QED

\sssec{Proof of Proposition~\ref{Pp_4.2.7}}
\label{Section_4.2.9}
For $\lambda\in \Lambda^{\succ}_{2n, d+d'}$ the stack $\cT^{\lambda}$ classifies collections: 
\begin{itemize}
\item complete flags $(L_1\subset\ldots\subset L_n=L)$ and $(L'_1\subset\ldots\subset L'_n=L')$ of vector bundles on $X$, for which we set 
\begin{equation}
\label{def_cM_i_for_proof_of_Pp4.2.7}
\begin{array}{ll}
M_{2i}=L_i\oplus L'_i, &\mbox{for}\; 0\le 2i\le 2n,\\
M_{2i+1}=L_{i+1}\oplus L'_i, & \mbox{for}\; 0< 2i+1<2n
\end{array}
\end{equation}
\item $(D_i)\in X^{\lambda}$;
\item isomorphisms $M_i/M_{i-1}\,\iso\,\Omega^{2n-i}(D_i)$ for $1\le i\le 2n$;
\item lower modifications of $\cO_X$-modules $\cM_i\subset M_i$ for $1\le i\le 2n$ such that 
$$
\cM_1\subset \cM_2\subset\ldots\subset \cM_{2n}
$$ 
is a complete flag of vector bundles on $X$, and the image of $\cM_i/\cM_{i-1}\to M_i/M_{i-1}$ is $(M_i/M_{i-1})(-D_i)$ for $1\le i\le 2n$.
\end{itemize}

\medskip

 Pick a $k$-point $\xi$ of ${^{2n}\cZ^{\lambda}}\times_{X^{\lambda}_{pos}} X^{\lambda}$ given by a collection: $(D_i)\in X^{\lambda}$, complete flags of vector bundles on $X$
\begin{equation}
\label{collection_for_2n_cZ_lambda}
(L_1\subset\ldots\subset L_n=L, \;\; L'_1\subset\ldots\subset L'_n=L'),
\end{equation}
and isomorphisms (\ref{trivilization_L'_i/L'_{i-1}}) and (\ref{trivialization_L_i/L_{i-1}}) for $1\le i\le n$. It suffices to establish the desired result after the $*$-restriction to $\xi$. Write $S$ for the fibre of $\varphi^{\lambda}$ over $\xi$. Define $M_i$ by (\ref{def_cM_i_for_proof_of_Pp4.2.7}). 

 For $x\in X$ write $\cO_x$ for the completed local ring of $X$ at $x$, $F_x$ for the fraction field of $\cO_x$, set $D_x=\Spec\cO_x$. Let 
\begin{equation}
\label{flags_L_i^x_and_L'^i_x}
(L_{1,x}\subset\ldots\subset L_{n,x}=L_x, \;\; L'_{1,x}\subset\ldots\subset L'_{n,x}=L'_x)
\end{equation}
be the restriction of (\ref{collection_for_2n_cZ_lambda}) to $D_x$, similarly for $M_{i,x}$. Write 
$$
D_i=\sum_{x\in X} \lambda_{i,x} x, 
$$
so for $x\in X$ we get the $x$-component $\lambda_x\in \Lambda^{\succ}_{2n}$ of $D=(D_i)$, and $D=\sum_{x\in X} \lambda_x x$.
 
 For $x\in X$ let $S_x$ be the scheme classifying diagrams
$$
\begin{array}{ccc}
M_{1,x} & \subset \ldots\subset &M_{2n,x}\\
\cup && \cup\\
\cM_{1,x} & \subset\ldots\subset & \cM_{2n,x},
\end{array}
$$
where $\cM_{i,x}\subset M_{i,x}$ is a $\cO_x$-submodule such that for $1\le i\le 2n$ the natural map 
$$
\cM_{i,x}/\cM_{i-1,x}\to M_{i,x}/M_{i-1,x}
$$ 
is an inclusion with image $(M_{i,x}/M_{i-1,x})(-\lambda_{i,x} x)$. Note that $S_x$ is isomorphic to an affine space of dimension $a_{2n}(\lambda_x)$. We have canonically $S\,\iso\, \prod_{x\in X} S_x$. 

 For $x\in X$ define the map $\mu_x: S_x\to\A^1$ as follows. For $1\le i<2n$ we have 
$$
M_{i+1}/M_{i-1}\,\iso\, M_i/M_{i-1}\oplus M_{i+1}/M_i
$$ 
canonically. So, for $1\le i<2n$ a point of $S_x$ yields an inclusion
\begin{multline*}
\Omega^{2n-i-1}_x\,\iso\, \cM_{i+1,x}/\cM_{i,x}\subset  M_{i+1,x}/(\cM_{i,x}+M_{i-1,x})\,\iso\\
(\Omega_x^{2n-i}(\lambda_{i,x}x)/\Omega^{2n-i}_x)\oplus \Omega_x^{2n-i-1}(\lambda_{i+1,x}x)
\end{multline*}
whose projection to the second factor is the canonical inclusion. Let $\sigma_{i,x}$ be the first component of the latter map. Then $\mu_x$ sends a point of $S_x$ to $\sum_{i=1}^{2n-1}\Res \sigma_{i,x}$. 

 Recall the map $\mu: {\cQ_{2n, d+d'}}\to\A^1$ defined in Section~\ref{Sect_3.1.1}. The composition $S\subset  {\cQ_{2n, d+d'}}\toup{\mu}\A^1$ writes as $\sum_x \mu_x: \prod_{x\in X}S_x\to\A^1$. Set 
$$
d_x=\sum_{i=1}^n \lambda_{2i-1, x}, \;\;\;\; d'_x=\sum_{i=1}^n \lambda_{2i, x}
$$ 
Write $\beta_x: S_x\to \Sh_0^{d_x+d'_x}$ for the map sending the above point to $M_{2n,x}/\cM_{2n,x}$. 
 
 Write $\Gr(M_{2n,x})$ for the affine grassmanian classifying free 
$\cO_x$-submodules $\cR\subset M_{2n,x}(F_x)$ of rank $2n$ with $\cR\otimes_{\cO_x} F_x=M_{2n,x}(F_x)$. We view $S_x$ as a subscheme of $\Gr(M_{2n,x})$. For $\nu\in \Lambda^{\succ+}_{2n}$ let $\Gr^{\nu}(M_{2n,x})\subset \Gr(M_{2n,x})$ be the locally closed subscheme classifying lattices $\cR\subset M_{2n,x}$ such that
$$
M_{2n,x}/\cR\,\iso\, \cO_x/\gm_x^{\nu_1}\oplus\ldots\oplus \cO_x/\gm_x^{\nu_{2n}},
$$
here $\gm\subset \cO_x$ is the maximal ideal. The $\IC$-sheaf of $\Gr^{\nu}(M_{2n,x})$ on $\Gr(M_{2n,x})$ is denoted $\cA_{\nu, x}$. 

  By (\cite{FGKV}, Proposition~3.1 and Lemma~4.2), the $*$-restriction of $(\iota^{\lambda}_{\cT})^*\beta^*\cL^{d+d'}_E$ to $S$ identifies with $\mathop{\boxtimes}\limits_{x\in X} F_x$, where $F_x$ is the $*$-restriction of
\begin{equation}
\label{complex_F_x_description}
\mathop{\oplus}\limits_{\nu\in \Lambda^{\succ+}_{2n, d_x+d'_x}} \cA_{\nu, x}[-(d_x+d'_x)(2n-1)]\otimes E_x^{\nu}
\end{equation}
under $S_x\hook{} \Gr(M_{2n,x})$. Here $E_x^{\nu}$ is the polynomial functor of the vector space $E_x$ defined in (\cite{L2}, 0.1.4). 

 Let $i_M: S_{M, x}\subset S_x$ be the closed subscheme given by the property that there are diagrams of $\cO_x$-modules
$$
\begin{array}{ccc}
L_{1,x} &\subset\ldots\subset & L_{n,x}\\
\cup && \cup\\
\cL_{1,x} &\subset\ldots\subset & \cL_{n,x}
\end{array}
\;\;\;\;\;\;\;\;\;\;\;
\begin{array}{ccc}
L'_{1,x} &\subset\ldots\subset & L'_{n,x}\\
\cup && \cup\\
\cL'_{1,x} &\subset\ldots\subset & \cL'_{n,x}
\end{array}
$$
such that 
$$
\left\{
\begin{array}{ll}
\cM_{2s,x}=\cL_{s,x}\oplus \cL'_{s,x}, & \mbox{for}\; 1\le 2s\le 2n\\
\cM_{2s+1,x}=\cL_{s+1,x}\oplus\cL'_{s,x}, & \mbox{for}\; 1\le 2s+1\le 2n
\end{array}
\right.
$$
By Kunneth formula, Proposition~\ref{Pp_4.2.7} is reduced to Proposition~\ref{Pp_4.2.10} below, which is a local result. \QED

\begin{Pp} 
\label{Pp_4.2.10}
For $x\in X$ the complex $\RG_c(S_{M, x}, 
i_M^*(F_x\otimes \mu_x^*\cL_{\psi}))$ is placed in degrees $\le 2a_{2n}(\lambda_x)$. Moreover, the natural map
\begin{equation}
\label{morphism_for_Con_about_F_x}
\RG_c(S_x, F_x\otimes \mu_x^*\cL_{\psi})\to \RG_c(S_{M, x}, 
i_M^*(F_x\otimes \mu_x^*\cL_{\psi}))
\end{equation}
induces an isomorphism in the highest degree $2a_{2n}(\lambda_x)$.
\end{Pp}
\begin{proof}
Using (\ref{complex_F_x_description}), our claim is reduced to the following. For any $\nu\in\Lambda^{\succ+}_{2n, d_x+d'_x}$ the natural map 
$$
\RG_c(S_x, \cA_{\nu,x}\otimes \mu_x^*\cL_{\psi})\to \RG_c(S_{M, x}, 
i_M^*(\cA_{\nu,x}\otimes \mu_x^*\cL_{\psi}))
$$
is an isomorphism in the (highest) degree $\<\lambda_x, 2\check{\rho}_G\>$ for $G=\GL_{2n}$. Besides, both complexes are placed in degrees $\le \<\lambda_x, 2\check{\rho}_G\>$. This is a particular case of Proposition~\ref{Pp_2.1.15}.  Namely, we apply it for $G=\GL_{2n}$ with $T$ the group of diagonal matrices, $B$ the group of upper-triangular matrices. The antistandard Levi is the subgroup of $\GL_{2n}$ preserving the decomposition $\Vect(e_1, e_3,\ldots, e_{2n-1})\oplus \Vect(e_2, e_4,\ldots, e_{2n})$, here $(e_i)$ is the canonical base of the standard representation of $\GL_{2n}$.
\end{proof}

 The proof of Proposition~\ref{Pp_1.2.10} is completed.
  
\ssec{Highest direct image: second step}  
\label{Sect_4.3}

In this section we finish the proof of Proposition~\ref{Pp_2.2.5}. More precisely, we reduce it to Proposition~\ref{Pp_4.2.8} below, which in turn is reduced to Theorem~\ref{Thm_4.2.11}. This reduction is the content of Section~\ref{Sect_4.3}. The proof of Theorem~\ref{Thm_4.2.11} is then given in Section~\ref{Section_proof_of_Thm_4.2.11}. 
  
\sssec{} Write $\tilde\beta: {\wt\cQ_{n,d}}\to {^{\le n}\Sh^d_0}$ for the map sending a point of the source to $L/L_n$. Write $\tilde\beta': {\wt\cQ'_{n, d'}}\to {^{\le n}\Sh^{d'}_0}$ for the map sending a point of the source to $L'/L'_n$. 
  
  Consider the commutative diagram
\begin{equation}
\label{diag_for_Sect_4.2.7}
\begin{array}{ccccc}
\cQ_{2n, d+d'} & \getsup{\eta} & {\wt\cQ_{n,d}}\times{\wt\cQ'_{n,d'}} & \toup{\bar\pi^{d,d'}\bar\eta} & X^{(d)}\times X^{(d')}\\
\downarrow\lefteqn{\scriptstyle\beta} && \downarrow\lefteqn{\scriptstyle \tilde\beta\times\tilde\beta'} &\nearrow\lefteqn{\scriptstyle {^{\le n}(\div\times\div)}}\\
^{\le 2n}\Sh^{d+d'}_0 & \getsup{\eta_0} & {^{\le n}\Sh^d_0}\times{^{\le n}\Sh^{d'}_0}, 
\end{array}
\end{equation}
where $\eta_0$ sends $(F,F')$ to $F\oplus F'$, and $^{\le}(\div\times\div)$ is the restriction of $\div\times\div$. 

 For any local system $E$ on $X$ we define a canonical constructible subsheaf
$$
^{\le n}((\wedge^2 E)^{(d)}\boxtimes E^{(d'-d)})\subset ((\wedge^2 E)^{(d)}\boxtimes E^{(d'-d)})
$$
in Definition~\ref{Def_4.3.11} below. We will derive Proposition~\ref{Pp_2.2.5} from the following Proposition~\ref{Pp_4.2.8}, whose proof is given in Section~\ref{Sect_4.2.14}. 
\begin{Pp} 
\label{Pp_4.2.8}
For any local system $E$ on $X$ the complex
\begin{equation}
\label{complex_Pp_4.2.8}
i_X^*(^{\le n}(\div\times\div))_!(\tilde\beta\times\tilde\beta')_!(\tilde\beta\times\tilde\beta')^*\eta_0^*\cL^{d+d'}_E
\end{equation}
is placed in the usual cohomological degrees $\le b(2n, d+d')+n(g-1)-2d'=top$, and its cohomology sheaf in degree $top$ identifies canonically with
\begin{equation}
\label{sheaf_le_n_in_the_answer}
^{\le n}((\wedge^2 E)^{(d)}\boxtimes E^{(d'-d)})
\end{equation}
If $\rk(E)=2n$ then (\ref{subsheaf_Pp_2.2.5}) is an equality. 
\end{Pp}  

\sssec{Proof of Proposition~\ref{Pp_2.2.5}} Consider the diagram (\ref{diag_for_Sect_4.2.2}). One has canonically over ${\wt\cY_{n,d}}\times{\wt\cY'_{n, d'}}$
$$
\iota^*(\cP^{d+d'}_{2n, E})[n(g-1)-d']\,\iso\, q_{\cY\cZ}^*(\cP_{2n, E}\boxtimes\Qlb)[\dimrel(\nu_{\cZ})]
$$
So, in view of Proposition~\ref{Pp_graded_result}, to prove Proposition~\ref{Pp_2.2.5}, we must calculate the cohomology sheaf in the usual degree $n(g-1)-2d'$ of the complex
$$
\bar\pi^{d,d'}_!\iota_{\cT}^*(\cF^{d+d'}_{2n, E})
$$  
By Proposition~\ref{Pp_1.2.10}, this cohomology sheaf identifies with the cohomology sheaf in the same degree of
$$
\bar\pi^{d,d'}_!\bar\eta_!\eta^*(\cF^{d+d'}_{2n, E})
$$
From (\ref{diag_for_Sect_4.2.7}) one gets
$$
\eta^*(\cF^{d+d'}_{2n, E})\,\iso\, \eta^*\beta^*\cL^{d+d'}_E[b(2n, d+d')]
$$  
Our claim follows now from Proposition~\ref{Pp_4.2.8}. \QED

\begin{Rem} Recall the sheaf $E^{\lambda}_-$ defined in (\cite{L2}, Definition~1, Section~3.1). In the case $n=1$ the top cohomology sheaf of (\ref{complex_Pp_4.2.8}) identifies
with $E^{\lambda}_-$ on $X^{\lambda}_-=X^{(d)}\times X^{(d'-d)}$ for $\lambda=(d,d')$ by definition of $E^{\lambda}_-$. 
\end{Rem}

\sssec{} 
\label{Sect_4.3.5_now}
In the rest of Section~\ref{Sect_4.3} we reduce Proposition~\ref{Pp_4.2.8} to Theorem~\ref{Thm_4.2.11}. 

For the rest of Section~\ref{Section_Direct_image_barpi} assume $0\le d\le d'$ and set $I=\{1,\ldots, d+d'\}$. Let $S_I$ be the group of permutations of $I$. 

Set 
$$
V^{d,d'}=(X^{(d)}\times X^{(d')})\times_{X^{(d+d')}} X^{d+d'}
$$ 
This is the scheme classifying pairs of effective divisors $D, D'$ of degrees $d,d'$, and $(x_i)\in X^{d+d'}$ with $\sum_i x_i=D+D'$. Write $V^{d,d'}_-\hook{} V^{d,d'}$ for the closed subscheme given by $D\le D'$. It is equipped with an action of $S_I$. 

\sssec{Description of the $\IC$-sheaf}
\label{Sect_1.2.8_normalization}  
Consider the set $\cE$ of equivalence classes of surjections $\alpha: I\to \bar I$ such that there is a decomposition $\bar I=\bar I_1\sqcup \bar I_2$ with $\mid\bar I_2\mid=d'-d$ and the following property. For $i\in \bar I_2$ (resp., $i\in \bar I_1$), $\alpha^{-1}(i)$ has 1 (resp., 2) elements. In other words, $\cE$ is the set of equivalence relations on $I$ which have $d'-d$ equivalence classes consisting of 1 element, and $d$ equivalence classes consisting of 2 elements. We get a map 
$$
\norm: \mathop{\sqcup}\limits_{\alpha\in\cE} X^{\bar I}\to V^{d,d'}_-
$$
whose restriction to $X^{\bar I}$ sends $\gamma: \bar I\to X$ to 
$D=\sum_{i\in \bar I_1} \gamma(i)$, $D'=D+\sum_{i\in \bar I_2} \gamma(i)$ and the decomposition $I\to X$ of $D+D'$ is given by $\gamma\alpha$. This is the normalization of $V^{d,d'}_-$. So, $\norm_!\Qlb[d']\,\iso\,\IC$, and $V^{d,d'}_-$ is of pure dimension $d'$. The map $\norm$ is $S_I$-equivariant. 

 Denote by  ${^0V^{d,d'}_-}\subset V^{d,d'}_-$ the open subscheme given by the property that $D$ is reduced, and $(D'-D)\cap D=\emptyset$. Then $\norm$ is an isomorphism over this open subscheme, so $\IC\,\iso\, \Qlb[\dim]$ over ${^0V^{d,d'}_-}$.
 
 For each $(\bar I,\alpha)\in\cE$ the restriction $\norm_{\alpha}: X^{\bar I}\to V^{d,d'}_-$  of $\norm$ is a closed immersion. Indeed, the resulting closed subscheme is given by the property that a point $I\to X$ factors (uniquely) through $I\toup{\alpha} \bar I\toup{\gamma} X$, and $D=\sum_{i\in \bar I_1} \gamma(i)$. We denote by $V_{\alpha}$ this closed subscheme of $V^{d,d'}_-$. 
 
 By an involution of $I$ we mean a permutation $\sigma$ with $\sigma^2=\id$, so $\sigma$ could be the identity. Write $\Inv$ for the set of involutions in $S_I$. For $\alpha: I\to \bar I$ in $\cE$, let $w\in \Inv$ be the unique involution of $I$ whose orbits are precisely the fibres of $\alpha$. This defines an inclusion $\cE\hook{}\Inv$, which we use later.
 
 For $\alpha\in\cE$ denote by $S_{\alpha}$ the stabilizer of $\alpha$ in $S_I$. Set $I_i=\alpha^{-1}(\bar I_i)$ for $i=1,2$. Write $S_{I_i}$ for the group of permutations of $I_i$, so $S_{\alpha}\subset S_{I_1}\times S_{I_2}\subset S_I$. Let $\chi_{\alpha}: S_{\alpha}\to \Qlb^*$ be the restriction of the character $\sign\times\triv: S_{I_1}\times S_{I_2}\to\Qlb^*$. The group $S_{\alpha}$ fits into an exact sequence $1\to (S_2)^d\to S_{\alpha}\to S_{\bar I_1}\times S_{\bar I_2}\to 1$. 
 
\sssec{Twisted $\IC$-sheaf} We introduce the following \select{canonical induced representation} of $S_I$. The term \select{canonical} here is to express the fact that it is independent of the order on $I$ (and only depends on $I, d, d'$). 

Let $\{e_i\}_{i\in I}$ be the canonical base of $\Qlb^I$. 
For $\alpha: I\to \bar I$ in $\cE$ and $j\in\bar I_1$ set 
$$
V_{j,\alpha}=\mathop{\oplus}\limits_{i\in \alpha^{-1}(j)} \Qlb e_i\subset \Qlb^I
$$ 
Let
$$
V_{\alpha}=\mathop{\otimes}\limits_{j\in \bar I_1} \det V_{j,\alpha},\;\;\;
Ind_{\cE}=\mathop{\oplus}\limits_{\alpha\in\cE} V_{\alpha}
$$
It is understood that $\bar I_1$ is attached to an element $\alpha: I\to \bar I$ of $\cE$ as above. Equip $\Ind_{\cE}$ with the natural action of $S_I$. One may view $Ind_{\cE}$ as a subspace in $(\Qlb^I)^{\otimes 2d}$. For $\alpha\in\cE$ there exists a noncanonical isomorphism $Ind_{\cE}\,\iso\, ind_{S_{\alpha}}^{S_I} \chi_{\alpha}$ of $S_I$-representations. We underline that $Ind_{\cE}$ is independent of $\alpha\in\cE$. 

 Consider the sheaf $(\Qlb)_{\cE}$ on $\mathop{\sqcup}\limits_{\alpha\in\cE} X^{\bar I}$, whose restriction to the component $X^{\bar I}$ for a given $\alpha$ is the constant sheaf with fibre $V_{\alpha}$. Then $(\Qlb)_{\cE}$ is naturally $S_I$-equivariant.
Set
$$
\wt\IC[-d']=\norm_!(\Qlb)_{\cE}
$$ 
This is a constructible sheaf on $V^{d,d'}_-$, which inherits a natural action of $S_I$. 

We call $\wt\IC$ the \select{twisted version} of the $\IC$-sheaf of $V^{d,d'}_-$. We informally think of $\wt\IC$ as a version of the canonical induced representation of $S_I$ with a point replaced by the curve $X$. There is a noncanonical isomorphism $\wt\IC\,\iso\, \IC$. 

\sssec{} 
\label{Sect_4.3.8_now}
For the rest of Section~\ref{Section_Direct_image_barpi} set $\nu=(1,\ldots, 1)\in \Lambda_{d+d', d+d'}$.  Let $\Fl^{\nu}$ be the stack of flags $(F_1\subset\ldots\subset F_{d+d'})$, where $F_i$ is a coherent torsion sheaf on $X$ of length $i$ for $i=1,\ldots, d+d'$. Let $q_{\cF}: \Fl^{\nu}\to\Sh_0^{d+d'}$ be the map sending the above point to $F_{d+d'}$. 
 
 Consider the diagram
$$
\begin{array}{ccccc}
\Sh_0^{d+d'} & \getsup{\eta_0} & \Sh^d_0\times \Sh^{d'}_0 & \toup{\div\times\div} & X^{(d)}\times X^{(d')}\\
\uparrow\lefteqn{\scriptstyle q_{\cF}} && \uparrow && \uparrow\lefteqn{\scriptstyle q_V}\\
\Fl^{\nu} & \gets & \Sh^{d,d'} & \toup{\div^{\nu}} & V^{d,d'},
\end{array}
$$
where the left square is cartesian, so defining the stack $\Sh^{d,d'}$. The map $\div^{\nu}$ sends $(F, F', F_1\subset\ldots\subset F_{d+d'}=F\oplus F')$ to the collection $(D,D', (x_i))$, where $D=\div F, D'=\div F'$, and $x_i=\div(F_i/F_{i-1})$ for $1\le i\le d+d'$. The map $q_V$ sends $(D, D', (x_i))$ to $(D, D')$. 

 Let $\div^{\nu}_-: \Sh^{d,d'}_-\to V^{d,d'}_-$ be obtained from $\div^{\nu}$ by the base change $V^{d,d'}_-\hook{}  V^{d,d'}$. 
Let $^{\le 2n}\Fl^{\nu}$ be the preimage of $^{\le 2n}\Sh^{d+d'}_0$
under $q_{\cF}: \Fl^{\nu}\to\Sh^{d+d'}_0$. Consider the diagram
$$ 
\begin{array}{ccccc}
^{\le 2n}\Sh^{d+d'}_0 & \getsup{\eta_0} & {^{\le n}\Sh^d_0}\times{^{\le n}\Sh^{d'}_0}  & \getsup{\tilde\beta\times\tilde\beta'} & 
{\wt\cQ_{n,d}}\times{\wt\cQ'_{n, d'}} \\
\uparrow\lefteqn{\scriptstyle q_{\cF}} && \uparrow  && \uparrow\\ 
^{\le 2n}\Fl^{\nu} & \gets & ^{\le n}\Sh^{d,d'} & \getsup{\beta^{\nu}} & \wt\cQ\wt\cQ^{\nu}_n,
\end{array}
$$  
where both squares are cartesian, so defining the stacks in the low row and $\beta^{\nu}$. Write 
$$
\beta^{\nu}_-: {\wt\cQ\wt\cQ^{\nu}_{n,-}}\to {^{\le n}\Sh^{d,d'}_-}
$$ 
for the map obtained from $\beta^{\nu}$ by the base change $i_X: X^{(d)}\times X^{(d'-d)}\hook{} X^{(d)}\times X^{(d')}$. 

 Consider the diagram
$$
\begin{array}{ccc}
\wt\cQ\wt\cQ^{\nu}_{n,-}\;\,\toup{\beta^{\nu}_-}\;\, {^{\le n}\Sh^{d,d'}_-} \,\hook{} \,\Sh^{d,d'}_-\;\toup{\div^{\nu}_-} & V^{d,d'}_- & \toup{p_V} X^{d+d'}
\\
& \downarrow\lefteqn{\scriptstyle q_{V}^-} \\
& X^{(d)}\times X^{(d'-d)},
\end{array}
$$
where $p_V$ sends $(D,D',(x_i))$ to $(x_i)$, and $q_V^-$ is obtained from $q_V$ by the base change $i_X$. Let 
$$
^{\le n}\div^{\nu}_-: {^{\le n}\Sh^{d,d'}_-} \to V^{d,d'}_-
$$ 
be the restriction of $\div^{\nu}_-$ to this open substack. The group $S_I$ acts naturally on $V^{d,d'}_-$, $X^{d+d'}$, and the map $p_V$ is $S_I$-equivariant. 
 
 Proposition~\ref{Pp_4.2.8} will be derived from the following results, whose proof is given in Sections~\ref{Sect_1.2.8_normalization}-\ref{Sect_4.3.18}.  
\begin{Thm} 
\label{Thm_4.2.11} Recall our assumption $0\le d\le d'$.\\
i) The map $\div^{\nu}_-: \Sh^{d,d'}_-\to V^{d,d'}_-$ is surjective. Any of its fibres is of dimension $-d'$. \\
ii) One has canonically 
$$
\R^{-2d'}(\div^{\nu}_-)_!\Qlb\,\iso\IC[-d']
$$
So, for any local system $E$ on $X$,
$$
i_X^*(\div\times\div)_!\eta_0^*\Spr^{d+d'}_E
$$ 
is placed in usual degrees $\le -2d'$, and its cohomology sheaf in degree $-2d'$ is isomorphic to
\begin{equation}
\label{complex_with_difficult_action}
(q_V^-)_!(\wt\IC[-d']\otimes p_V^*E^{\boxtimes (d+d')})
\end{equation}
Moreover, this isomorphism is $S_I$-equivariant.
\end{Thm}
\begin{Pp} 
\label{Pp_4.2.12}
For any local system $E$ on $X$, the sheaf of $S_I$-invariants in (\ref{complex_with_difficult_action}) identifies with $(\wedge^2 E)^{(d)}\boxtimes E^{(d'-d)}$. 
\end{Pp}
\begin{proof} Write $\ssum: X^d\to X^{(d)}$ to the map sending $(x_i)$ to $\sum_i x_i$. For any $\alpha: I\to \bar I$ in $\cE$, the composition 
$$
X^{\bar I}\;\toup{\norm_{\alpha}}\; V^{d,d'}_-\;\toup{q_V^-}\; X^{(d)}\times X^{(d'-d)}
$$ 
is the map $\ssum\times\ssum: X^d\times X^{d'-d}\to X^{(d)}\times X^{(d'-d)}$. So, 
$$
(q_V^-)_!(\wt\IC[-d']\otimes p_V^*E^{\boxtimes (d+d')})\,\iso\,
\mathop{\oplus}\limits_{\alpha\in\cE} ((\ssum_! (E\otimes E)^{\boxtimes d})\boxtimes (\ssum_! E^{\boxtimes (d'-d)}))\otimes V_{\alpha}
$$
according to Section~\ref{Sect_1.2.8_normalization}. The group $S_{d+d'}$ acts transitively on $\cE$, and a stabilizer $S_{\alpha}$ of some $\alpha\in\cE$ was considered in Section~\ref{Sect_1.2.8_normalization}. It
fits into an exact sequence of groups $1\to (S_2)^d\to S_{\alpha}\to (S_d\times S_{d'-d})\to 1$. 
So, we are calculating $S_{\alpha}$-invariants on 
$$
((\ssum_! (E\otimes E)^{\boxtimes d})\boxtimes (\ssum_! E^{\boxtimes (d'-d)}))\otimes V_{\alpha}
$$
Taking the $(S_2)^d$-invariants first, one gets 
\begin{equation}
\label{sheaf_with_wedge^2}
(\ssum_! (\wedge^2 E)^{\boxtimes d})\boxtimes (\ssum_! E^{\boxtimes (d'-d)}),
\end{equation}
and this sheaf is equipped with the `induced' action of $S_d\times S_{d'-d}$. Let $S_d$ act naturally on $\ssum_! ((\wedge^2 E)^{\boxtimes d})$, and $S_{d'-d}$ act naturally on $\ssum_! (E^{\boxtimes (d'-d)})$. Then the product of these two actions coincides with the `induced' action of $S_d\times S_{d'-d}$ on (\ref{sheaf_with_wedge^2}). Our claim follows. 
\end{proof}

\begin{Def} 
\label{Def_4.3.11}
Recall our assumption $0\le d\le d'$. Let $^{\le n}K_V$ on $V^{d,d'}_-$ be given by 
$$
^{\le n}K_V=\R^{-2d'}(^{\le n}\div^{\nu}_-)_!\Qlb
$$  
Since $^{\le n}\Sh^{d,d'}_-$ is open in $\Sh^{d,d'}_-$, this is a constructible subsheaf of $\IC[-d']$ on $V^{d,d'}_-$. We get a $S_I$-equivariant filtration by constructible subsheaves
$$
^{\le 1}K_V\subset {^{\le 2}K_V}\subset\ldots\subset \IC[-d']
$$
For any local system $E$ on $X$ define the constructible sheaf
$$
^{\le n}((\wedge^2 E)^{(d)}\boxtimes E^{(d'-d)})
$$
on $X^{(d)}\times X^{(d'-d)}$ as the sheaf of $S_I$-invariants in 
$$
(q_V^-)_!(p_V^*E^{\boxtimes(d+d')}\otimes {^{\le n}K_V}),
$$
where the $S_I$-action comes from the corresonding action on the Springer sheaf. 

 Since $(q_V^-)_!$ is exact for the usual t-structures, from Proposition~\ref{Pp_4.2.12} we see that (\ref{sheaf_le_n_in_the_answer}) is a constructible subsheaf in $(\wedge^2 E)^{(d)}\boxtimes E^{(d'-d)}$. 
\end{Def}
\sssec{Proof of Proposition~\ref{Pp_4.2.8}} 
\label{Sect_4.2.14}
Recall the maps 
$\tilde\beta: {\wt\cQ_{n,d}}\to {^{\le n}\Sh^d_0}$ and $\tilde\beta': {\wt\cQ'_{n,d'}}\to {^{\le n}\Sh^{d'}_0}$ from (\ref{diag_for_Sect_4.2.7}). The map $\tilde\beta$ is surjective and smooth of relative dimension 
$$
nd-\frac{n}{6}(n-1)(1+4n)(g-1)
$$ 
The map $\tilde\beta'$ is surjective and smooth of relative dimension $nd'-\frac{n}{6}(n-1)(1+4n)(g-1)$. All the fibres of $\tilde\beta$ and of $\tilde\beta'$ are irreducible. Since
$$
\dim(\cQ_{2n, d+d'})=b(2n, d+d')=2n(d+d')+(1-g)\frac{(2n-1)}{3}n(4n-1),
$$
one has 
\begin{equation}
\label{dimrel_tildebeta_tildebeta'}
2\dimrel(\tilde\beta\times\tilde\beta')=b(2n, d+d')+n(g-1)
\end{equation}
  
By definition, (\ref{complex_Pp_4.2.8}) identifies canonically with the sheaf of $S_I$-invariants in 
\begin{multline*}
i_X^*(^{\le n}(\div\times\div))_!(\tilde\beta\times\tilde\beta)_!(\tilde\beta\times\tilde\beta)^*\eta_0^*\Spr^{d+d'}_E\,\iso\, \\ 
(q_V^-)_!(p_V^*E^{\boxtimes(d+d')}\otimes (^{\le n}\div^{\nu}_-)_!(\beta^{\nu}_-)_!\Qlb)
\end{multline*}
Here $S_I$ acts via its action on $\Spr^{d+d'}_E$. The map $\tilde\beta\times\tilde\beta'$ is surjective and smooth, all its fibres are irreducible. Since $\beta^{\nu}_-$ is obtained by base change from $\tilde\beta\times\tilde\beta'$, $(\beta^{\nu}_-)_!\Qlb$ is placed in degrees $\le 2\dimrel(\tilde\beta\times\tilde\beta')$ and
$$
\R^{2\dimrel(\tilde\beta\times\tilde\beta')}(\beta^{\nu}_-)_!\Qlb\,\iso\,\Qlb
$$  
By Theorem~\ref{Thm_4.2.11} i), the complex $(^{\le n}\div^{\nu}_-)_!(\beta^{\nu}_-)_!\Qlb$ is placed in usual cohomological degrees $\le 2\dimrel(\tilde\beta\times\tilde\beta')-2d'=top$, and its top cohomology sheaf is
$$
\R^{top}(^{\le n}\div^{\nu}_-)_!(\beta^{\nu}_-)_!\Qlb\,\iso\, \R^{-2d'}(^{\le n}\div^{\nu}_-)_!\Qlb=\; {^{\le n}K_V}
$$ 

Since $q_V^-$ is finite, (\ref{complex_Pp_4.2.8}) is placed in the usual degrees $\le 2\dimrel(\tilde\beta\times\tilde\beta')-2d'$, and the top cohomology sheaf of (\ref{complex_Pp_4.2.8}) identifies with (\ref{sheaf_le_n_in_the_answer}).
 
 From Proposition~\ref{Pp_graded_result} we see that (\ref{sheaf_le_n_in_the_answer}) admits a filtration by constructible subsheaves as in Proposition~\ref{Pp_graded_result}.  
Assume now $\rk(E)=2n$. In this case we check that this filtration exhausts  the sheaf 
$$
(\wedge^2 E)^{(d)}\boxtimes E^{(d'-d)}
$$
To do so, using Lemma~\ref{Lm_2.3.2}, we check that for any $(D,D')\in X^{(d)}\times X^{(d')}$ with $D\le D'$ the $*$-fibre of (\ref{complex_graded_by_lambda}) at $(D,D')$ identifies with 
\begin{equation}
\label{complex_first_for_proof_Pp_4.2.8}
(\wedge^2 E)^{(d)}_D\otimes (E^{(d'-d)})_{D'-D}
\end{equation}
If $D=dx, D'=d'x$ then this claim is precisely Lemma~\ref{Lm_2.3.2}. 

 For a $\Qlb$-vector space $V$ and $\lambda\in\Lambda^{\succ-}_{2n}$ set $V^{\lambda}=V^{w_0(\lambda)}$, where $w_0$ is the longest element of $S_{d+d'}$. In general, we write $D=\sum_x d_x x, D'=\sum_x d'_x x$, so (\ref{complex_first_for_proof_Pp_4.2.8}) identifies with
$$
\otimes_x \; (\Sym^{d_x}(\wedge^2 E_x)\otimes \Sym^{d'_x-d_x}(E_x))
$$
Set 
$$
\cS_x=\{\lambda_x\in\Lambda^{\succ-}_{2n, d_x+d'_x}\mid \lambda^{odd}_x\in \Lambda^{\succ-}_{n,d_x}, \lambda^{even}_x\in \Lambda^{\succ-}_{n, d'_x}\}
$$ 
and 
$$
\cS=\{\lambda\in\Lambda^{\succ-}_{2n, d+d'}\mid \lambda^{odd}\in \Lambda^{\succ-}_{n,d}, \lambda^{even}\in \Lambda^{\succ-}_{n, d'}\}
$$ 
Recall that $X^{\lambda}_-$ is the scheme of $\Lambda^{\succ-}_{2n}$-valued divisors on $X$ of degree $\lambda$. To a collection $(\lambda_x)\in \prod_x \cS_x$ with $\lambda=\sum_x \lambda_x$ we attach a point $\tilde D=\sum_x \lambda_x x\in X^{\lambda}_-$. By Lemma~\ref{Lm_2.3.2}, (\ref{complex_first_for_proof_Pp_4.2.8})  identifies with
$$
\mathop{\oplus}\limits_{\lambda\in\cS}(
\mathop{\oplus}\limits_{(\lambda_x)\in \prod_x \cS_x, \; \lambda=\sum_x \lambda_x} (E^{\lambda}_-)_{\tilde D}),
$$
where $\tilde D$ is as above. Our claim follows.
\QED

\sssec{Proof of Proposition~\ref{Pp_2.3.3}} 
\label{Sect_4.3.13}
Combine Propositions~\ref{Pp_graded_result} and \ref{Pp_4.2.8}. \QED

\ssec{Proof of Theorem~\ref{Thm_4.2.11}}
\label{Section_proof_of_Thm_4.2.11}

Let us sketch the ideas involved in the proof of Theorem~\ref{Thm_4.2.11}. To prove part i) we stratify the stack $\Sh^{d,d'}_-$ naturally and reduce to an estimation of dimension of the fibres of $\div^{\nu}_-$ intersected with strata. This, in turn, is reduced to Lemma~\ref{Lm_1.2.18}, which is of local nature (the curve $X$ is replaced by a formal disk). Lemma~\ref{Lm_1.2.18} is proved in inductive way by introducing an ad hoc notion of a \select{special transposition} in the symmetric group (for a given element of $\Lambda_{2m}$). 

 The more difficult part is Theorem~\ref{Thm_4.2.11} ii). Our approach here is to use another stratification of $\Sh^{d,d'}_-$ given in Lemma~\ref{Lm_4.4.12}. The definition of this second stratification is  inspired by some constructions of Richardson and Springer from \cite{RS}. Write $\Sigma$ for the set of $d$-element subsets in $I=\{1,\ldots, d+d'\}$. The corresponding strata are locally closed substacks
$$
\cX_{\cJ}\cap\cY_{\cJ'}\subset \Sh^{d,d'}_-
$$ 
indexed by pairs $\cJ,\cJ'\in \Sigma$. The key result is then Proposition~\ref{Pp_4.3.10}. It affirms that $\cX_{\cJ}\cap\cY_{\cJ'}$ does not contribute to the desired highest direct image unless $\cJ\cap\cJ'=\emptyset$ and some additional condition (C) holds. Under the latter conditions it provides a further stratification of $\cX_{\cJ}\cap\cY_{\cJ'}$ and identifies the contribution of these strata. 

Combining the above with general Remarks~\ref{Rem_4.4.13} and \ref{Rem_2.4.6}, we finish the proof of Theorem~\ref{Thm_4.2.11}. 

\sssec{} Recall the scheme $X^{\mu}_+$ defined in Section~\ref{Sect_1.1.1}. For $\mu\in\Lambda_{m,d}^{\succ+}$ denote by $i_{\mu}:X^{\mu}_+\to \Sh^d_0$ the map sending $(D_1,\ldots, D_m)$ to $\cO_{D_1}\oplus\ldots\oplus \cO_{D_m}$. The image of $i_{\mu}$ is a locally closed substack of $\Sh^d_0$ denoted $\Sh^d_{0,\mu}$. As $\mu$ ranges in $\Lambda_{m,d}^{\succ+}$, the stacks $\Sh^d_{0,\mu}$ form a stratification of the open substack $^{\le m}\Sh^d_0$.

\sssec{Proof of Theorem~\ref{Thm_4.2.11} i)} The argument is similar to (\cite{La}, Theorem 3.3.1). 
The stack $\Sh^{d,d'}_-$ classifies: $F\in \Sh^d_0, F'\in\Sh^{d'}_0$ with $D'=\div(F')\ge D=\div(F)$, and a complete flag $0=F_0\subset F_1\subset\ldots\subset F_{d+d'}=F\oplus F'$ of torsion subsheaves on $X$. Given $m\ge 1$, $\mu\in\Lambda_{m, d}^{\succ+}$ and $\mu'\in\Lambda_{m, d'}^{\succ+}$ set 
\begin{equation}
\label{Sh_mu_mu'_minus}
\Sh_{\mu, \mu', -}=\Sh^{d,d'}_-\times_{\Sh_0^d\times \Sh_0^{d'}} (\Sh^d_{0,\mu}\times \Sh^{d'}_{0,\mu'})
\end{equation}
The composition 
\begin{equation}
\label{decomposition_Pp_1.2.17}
\Sh_{\mu, \mu', -}\to (X^{\mu}_+\times X^{\mu'}_+)\times_{X^{(d)}\times X^{(d')}} V^{d,d'}_-\toup{\pr} V^{d,d'}_-
\end{equation}
is the restriction of $\div^{\nu}_-$. 

Fix a $k$-point of 
$
(X^{\mu}_+\times X^{\mu'}_+)\times_{X^{(d)}\times X^{(d')}} V^{d,d'}_-
$ 
given by $(D, D', (x_i))\in V^{d,d',\nu}_-$, so $D\le D'$, 
$$
\sum_{i=1}^{d+d'} x_i=D+D',
$$ 
and $(D_j)\in X^{\mu}_+, (D'_j)\in X^{\mu'}_+$ with 
$D=\sum_j D_j$ and $D'=\sum_j D'_j$.
Write $Y$ for the fibre of the first map in (\ref{decomposition_Pp_1.2.17}) over this point. 
We check that $\dim Y\le -d'$. Besides, by Lemma~\ref{Lm_1.2.18} below, if $m=1$, $\mu=(d), \mu'=(d')$ then $\dim Y=-d'$. 
 
  One has 
$$
F\,\iso\,\cO_{D_1}\oplus\ldots\oplus\cO_{D_m}\;\;\;\mbox{and}\;\;\; F'\,\iso\, \cO_{D'_1}\oplus\ldots\oplus \cO_{D'_m}
$$ 
As in (\cite{La}, Theorem 3.3.1), one checks that 
$$
\dim\Aut(F)=\sum_{i=1}^m (\mu_i-\mu_{i+1})i^2\;\;\;\mbox{and}\;\;\; 
\dim\Aut(F')=\sum_{i=1}^m (\mu'_i-\mu'_{i+1})i^2
$$ 
We have canonically $F\,\iso\, \oplus_{x\in X} F_x$, $F'\,\iso\, \oplus_{x\in X} F'_x$, where $F_x$, $F'_x$ are torsion sheaves supported on a formal neighbourhood of $x$. 
Write 
$$
D_i=\sum_{x\in X} \mu_{i,x} x\;\;\;\mbox{and}\;\;\; D'_i=\sum_{x\in X} \mu'_{i,x} x
$$ 
for $1\le i\le m$. Set $d_x=\sum_i \mu_{i,x}$ and $d'_x=\sum_i \mu'_{i,x}$. Write $CFl(F_x\oplus F'_x)$ for the scheme of complete flags on $F_x\oplus F'_x$. We have 
$$
\prod_{x\in X}\Aut(F_x)\,\iso\,\Aut(F), \;\; \prod_{x\in X}\Aut(F'_x)\,\iso\,\Aut(F')
$$
So,
$$
Y\,\iso\,\prod_{x\in X} (CFl(F_x\oplus F'_x))/(\Aut(F_x)\times\Aut(F'_x)))
$$
By Lemma~\ref{Lm_1.2.18} below, for each $x\in X$,
$$
\dim (CFl(F_x\oplus F'_x))/(\Aut(F_x)\times\Aut(F'_x)))\le -d'_x
$$
Our claim follows. \QED

\begin{Lm} 
\label{Lm_1.2.18}
Let 
$m\ge 1$, $\mu\in\Lambda^{\succ+}_{m, d}, \mu'\in\Lambda^{\succ+}_{m,d'}$. Let 
Let $\cO$ be a complete discrete valuation $k$-algebra. Write $\cO_m=\cO/\gm^m$, where $\gm\subset\cO$ is the maximal ideal.
$F=\cO_{\mu_1}\oplus\ldots\oplus\cO_{\mu_m}$, $F'=\cO_{\mu'_1}\oplus\ldots\oplus\cO_{\mu'_m}$. Let $CFl(F\oplus F')$ be the scheme of complete flags of $\cO$-modules on $F\oplus F'$. Set 
$$
\theta=(\mu'_1, \mu_1, \mu'_2,\mu_2,\ldots, \mu'_m, \mu_m)
$$ 
Let $\eta\in\Lambda^{\succ+}_{2m, d+d'}$ be obtained from $\theta$ by permutation so that the resulting sequence is decreasing. 
Then
\begin{equation}
\label{magic_inequality}
\dim CFl(F\oplus F')/(\Aut(F)\times\Aut(F'))\le -d', 
\end{equation}
and the inequality is strict unless $\theta=\eta$. In the latter case  (\ref{magic_inequality}) is an equality. 
\end{Lm}
\begin{proof} We have $F\oplus F'\,\iso\, \cO_{\eta_1}\oplus\ldots\oplus\cO_{\eta_{2m}}$. As in (\cite{La}, Theorem~3.3.1), one gets 
\begin{equation}
\label{expression_dim_CFl}
\dim CFl(F\oplus F')=\sum_{i=1}^{2m}(\eta_i-\eta_{i+1})\frac{i(i-1)}{2}
\end{equation}
and 
$$
\dim\Aut(F)=\sum_{i=1}^m(\mu_i-\mu_{i+1})i^2, \;\;\dim\Aut(F')=\sum_{i=1}^m (\mu'_i-\mu'_{i+1})i^2
$$ 
Here $\mu_{m+1}=\mu'_{m+1}=\eta_{2m+1}=0$. One has 
$$
\sum_{i=1}^m (\mu_i-\mu_{i+1}) i^2=\sum_{i=1}^m\mu_i(2i-1)\;\;\;\mbox{and}\;\;\; \sum_{i=1}^m (\mu'_i-\mu'_{i+1})i^2=\sum_{i=1}^m\mu'_i(2i-1)
$$
The expression (\ref{expression_dim_CFl}) equals $\sum_{i=2}^{2m} \eta_i(i-1)$. Set $\tau=(0,1,\ldots, 2m-1)\in\Lambda_{2m}$. We must show that
$\<\theta-\eta,\tau\>\ge 0$. In the case $\eta=\theta$ we get the desired equality. 

 Let $\cI=\{1,\ldots, 2m\}$. For a subset $J\subset \cI$ with $\mid J\mid=m$ denote by $\theta_J\in \Lambda^{\succ}_{2m}$ the element obtained as follows. Equip $J$ with the order induced from the standard order on $\cI$, and similarly for $\cI-J$. Then $\theta_J: \cI\to\ZZ_+$ is the map whose restriction to $J$ is $J\,\iso\, \{1,\ldots, m\}\toup{\mu}\ZZ$, and whose restriction to $\cI-J$ is $\cI-J\,\iso\, \{1,\ldots, m\}\toup{\mu'}\ZZ_+$. 
 
Let $\xi\in\Lambda^{\succ}_{2m}$. If for some $1\le i<2m$ we have $\xi_i<\xi_{i+1}$, let $\bar\xi$ be obtained from $\xi$ applying the permutation $(i, i+1)$. Then $\xi-\bar\xi=(0,\ldots, 0,-a, a,0,\ldots, 0)$, where $a=\xi_{i+1}-\xi_i$ appears on $i+1$-th place. In this case we say that $\bar\xi$ is obtained from $\xi$ by a \select{special transposition}. In this case $\<\xi-\bar\xi,\tau\>=a>0$. 
 
 There is a finite sequence $\eta_1, \ldots, \eta_s\in\Lambda^{\succ}_{2m}$ such that $\theta=\eta_1$, $\eta_s=\eta$, each $\eta_{i+1}$ is obtained from $\eta_i$ by a special transposition, and each $\eta_i$
is of the form $\theta_J$ for some $J$. On each step the quantity $\<\eta_i-\eta_{i+1},\tau\>$ is strictly positive. So, $\<\theta-\eta, \tau\>\ge 0$, and the inequality is strict unless $\theta=\eta$.

 Remark by a referee: one has $\tau=(m-\frac{1}{2})(1,\ldots, 1)-\rho_{\GL_{2m}}$, so that $\<\theta-\eta, \tau\>=\<\eta-\theta,\rho_{\GL_{2m}}\>\ge 0$. 
\end{proof}

\begin{Rem} In the notations of the proof of Theorem~\ref{Thm_4.2.11} i), one has $\dim Y< -d'$ unless 
\begin{equation}
\label{condition_Rem_1.2.19}
D'_1\ge D_1\ge D'_2\ge D_2\ge\ldots\ge D'_m\ge D_m
\end{equation}
Indeed, if $x\in X$ then 
$$
F_x\,\iso\, \cO_{\mu_1, x}\oplus\ldots\oplus\cO_{\mu_m, x}, \;\;\;\; F'_x\,\iso\, \cO_{\mu'_1, x}\oplus\ldots\oplus\cO_{\mu'_m, x}
$$ 
and $\deg(F_x)=d_x, \deg(F'_x)=d'_x$. Set $\theta_x=(\mu'_{1,x}, \mu_{1,x}, \mu'_{2,x},\mu_{2,x},\ldots,\mu'_{m,x},\mu_{m,x})$. Then $Y\,\iso\, \prod_{x\in X} Y_x$ and $\dim Y_x<-d'_x$ unless $\theta_x\in \Lambda^{\succ+}_{2m}$. The condition (\ref{condition_Rem_1.2.19}) is equivalent to requiring that for any $x\in X$, $\theta_x\in \Lambda^{\succ+}_{2m}$.
\end{Rem}

\begin{Rem} 
\label{Rem_1.2.20}
Let $m\ge 1$, $\mu\in\Lambda^{\succ+}_{m,d}$, $\cO$ a complete discrete valuation $k$-algebra, $F=\cO_{\mu_1}\oplus\ldots\oplus\cO_{\mu_m}$. Then $CFl(F)/\Aut(F)$ is of dimension $-\sum_{i=1}^m \mu_i i$.
\end{Rem}

\sssec{} Consider $\mu=(d)\in\Lambda^{\succ+}_{1,d}, \mu'=(d')\in\Lambda^{\succ+}_{1,d'}$ for $m=1$.
Consider the open substack $\Sh_{\mu,\mu',-}\subset \Sh^{d,d'}_-$ given by (\ref{Sh_mu_mu'_minus}). Let 
$$
^1\div^{\nu}_-: \Sh_{\mu,\mu',-}\to V^{d,d'}_-
$$ 
be the restriction of $\div^{\nu}_-$. We get the natural map 
\begin{equation}
\label{map_for_Sect_1.2.21}
R^{-2d'}(^1\div^{\nu}_-)_!\Qlb\to \R^{-2d'}(\div^{\nu}_-)_!\Qlb
\end{equation}
It is not an isomorphism over the whole of $V^{d,d'}_-$, as is seen from the proof of Theorem~\ref{Thm_4.2.11} i). 
However, the map (\ref{map_for_Sect_1.2.21}) is an isomorphism over the open subscheme ${^0V^{d,d'}_-}\subset V^{d,d'}_-$ 
defined in Section~\ref{Sect_1.2.8_normalization}. It is easy to see that 
$$
\R^{-2d'}(\div^{\nu}_-)_!\Qlb\,\iso\, \Qlb
$$ 
over  ${^0V^{d,d'}_-}$. Let 
$$
^0(X^{(d)}\times X^{(d'-d)})\subset X^{(d)}\times X^{(d'-d)}
$$ 
be the open subscheme of $(D, D'-D)$ such that $D$ is reduced, $D'-D$ is reduced and $(D'-D)\cap D=\emptyset$. 

 To check the equivariance property of the isomorphism in Theorem~\ref{Thm_4.2.11} ii), it suffices to do it over $^0(X^{(d)}\times X^{(d'-d)})$. Over this locus it follows easily from the corresponding property of the Springer sheaf $\Spr^2_E$. In other words, this equivariance property in general is reduced to the case $d=d'=1$. In the latter case it is easy and left to a reader.

\sssec{Stratifications I} 
\label{Sect_Stratifications_I}
To prove Theorem~\ref{Thm_4.2.11} ii), we introduce a suitable stratification of $\Sh^{d,d'}_-$. For this we need some additional notations. 

 A point of $\Sh^{d,d'}_-$ is written as $(F, F', (F_i))$, where $F_1\subset\ldots\subset F_{d+d'}=F\oplus F'$ is a complete flag of torsion subsheaves. We set $D=\div(F), D'=\div F'$, $x_i=\div(F_i/F_{i-1})$ for $i\in I$.

 Write $\Sigma$ for the set of $d$-element subsets in $I=\{1,\ldots, d+d'\}$. Fix $\cJ\in\Sigma$. For a point of $\Sh^{d,d'}_-$ set $\cF'_i=F'\cap F_i$, so $\cF'_i\subset F_i$ is a subsheaf. Set $\bar F_i=F_i/\cF'_i$ for all $i$. We get a complete flag with repetitions 
$$
\bar F_1\subset \bar F_2\subset\ldots\subset \bar F_{d+d'}=F,
$$
where we identify $\bar F_j$ with its image under the projection $\bar F_j\to (F\oplus F')/F'\,\iso\, F$.
We get also another complete flag with repetitions
$$
\cF'_1\subset\cF'_2\subset\ldots\subset \cF'_{d+d'}=F'
$$
We have the surjection $F_j\to \bar F_j$ for $j\in I$. So,  for $j\in I$ we get a diagram 
\begin{equation}
\label{diag_first_for_4.4.7}
\cF'_j/\cF'_{j-1}\hook{} F_j/F_{j-1}\to \bar F_j/\bar F_{j-1},
\end{equation}
where the first map is the natural injection, the second is the natural 
surjection. 

\sssec{} Denote by $\cX_{\cJ}\subset \Sh^{d,d'}_-$ the locally closed substack given by the property that for $j\in\cJ$ one has $\cF'_j=\cF'_{j-1}$, so that $\bar F_j/\bar F_{j-1}\in\Sh^1_0$. For a point of $\cX_{\cJ}$ the second map in (\ref{diag_first_for_4.4.7}) is an isomorphism for $j\in \cJ$, 
$$
x_j=\div(\bar F_j/\bar F_{j-1})\;\;\mbox{for}\; j\in\cJ,
$$
and $D=\sum_{j\in\cJ} x_j$. For a point of $\cX_{\cJ}$ we get 
$\cF'_j/\cF'_{j-1}\in\Sh^1_0$ for $j\in I-\cJ$, the first map in (\ref{diag_first_for_4.4.7}) is an isomorphism for $j\in I-\cJ$, 
$$
x_j=\div(\cF'_j/\cF'_{j-1}) \;\;\mbox{for}\; j\in I-\cJ,
$$ 
and $D'=\sum_{j\in I-\cJ} x_j$. 
 
\sssec{} For a point of $\Sh^{d,d'}_-$ set similarly $\cF_i=F\cap F_i$, so $\cF_i\subset F_i$ is a subsheaf. Set $\bar F'_i=F_i/\cF_i$ for $i\in I$. 
We get a complete flag with repetitions 
$$
\cF_1\subset\cF_2\subset\ldots\subset \cF_{d+d'}=F
$$
View $\bar F'_i$ as a subsheaf of $F'$ via the natural map $F_i/\cF_i\to (F\oplus F')/F\,\iso\, F'$. This gives a complete flag with repetitions
$$
\bar F'_1\subset \bar F'_2\subset\ldots\subset\bar F'_{d+d'}=F'
$$
For a point of $\Sh^{d,d'}_-$ and $j\in I$ we get a diagram
\begin{equation}
\label{diag_second_for_4.4.7}
\cF_j/\cF_{j-1}\hook{} F_j/F_{j-1}\to \bar F'_j/\bar F'_{j-1},
\end{equation}
where the first map is the natural injection, the second is the natural surjection.
 
\sssec{} Denote by $\cY_{\cJ}\subset \Sh^{d,d'}_-$ the locally closed substack given by the property that for $j\in\cJ$ one has $\cF_j/\cF_{j-1}\in \Sh^1_0$, so that $\bar F'_j=\bar F'_{j-1}$. For a point of $\cY_{\cJ}$ and $j\in\cJ$ the first map in (\ref{diag_second_for_4.4.7}) is an isomorphism, so 
$$
x_j=\div(\cF_j/\cF_{j-1})\;\;\;\mbox{for}\; j\in\cJ,
$$
and $D=\sum_{j\in \cJ} x_j$. 
For a point of $\cY_{\cJ}$ and $j\in I-\cJ$ the second map in (\ref{diag_second_for_4.4.7}) is an isomorphism, 
$$
\div(\bar F'_j/\bar F'_{j-1})=x_j\;\;\;\mbox{for}\; j\in I-\cJ,
$$
and $D'=\sum_{j\in I-\cJ} x_j$. 

\sssec{} On the sheaf $F'$ we get two flags: $(\bar F'_j)$ and $(\cF'_j)$. The relation between them is as follows. For any $k\in I$ we have the natural inclusion $\cF'_k\hook{} F_k$ and a surjection $F_k\to F_k/\cF_k=\bar F'_k$. Their composition $\cF'_k\to \bar F'_k$ is clearly injective for any $k$. We get the diagram
$$
\begin{array}{cccccc}
\bar F'_1 & \subset & \bar F'_2 & \subset\ldots\subset & \bar F'_{d+d'} &=F'\\
\cup &&\cup && \cup\\
\cF'_1 & \subset & \cF'_2 & \subset\ldots\subset & \cF'_{d+d'} &=F'
\end{array}
$$ 

 For a torsion sheaf $\cG$ on $X$ write $\ell(\cG)$ for its length. For $k\in I$ we get $\ell(\cF'_k)\le \ell(\bar F'_k)$. For any $k\in I$, $\ell(\cF'_k)$ is the number of elements $1\le j\le k$ such that $j\notin\cJ$. Besides, $\ell(\bar F'_k)$ is the number of elements $1\le j\le k$ such that $j\notin\cJ'$. For $k\in I$ set 
$$
\cJ_k=\cJ\cap\{1,\ldots, k\}, \;\;\;\; \cJ'_k=\cJ'\cap\{1,\ldots, k\}
$$ 
We see that the pair $\cJ,\cJ'$ satisfies the following condition:
\begin{itemize}
\item[(C)] if $k\in I$ then $\mid\cJ'_k\mid\le \mid\cJ_k\mid$. 
\end{itemize}

 Similarly, on $F$ we get two flags $(\bar F_j)$ and $(\cF_j)$. The relation between them is similar. For $k\in I$ we have an injection $\cF_k\hook{} F_k$ and a surjection $F_k\to \bar F_k$. Their composition $\cF_k\to \bar F_k$ is injective. We get the diagram
\begin{equation}
\label{diag_two_flags_on_F}
\begin{array}{cccccc}
\bar F_1 & \subset & \bar F_2 & \subset\ldots\subset & \bar F_{d+d'} &=F\\
\cup &&\cup && \cup\\
\cF_1 & \subset & \cF_2 & \subset\ldots\subset & \cF_{d+d'} &=F
\end{array}
\end{equation}
 
 The property that for any $k\in I$, $\ell(\cF_k)\le\ell(\bar F_k)$ is nothing but (C). We summarize this discussion in the following.
\begin{Lm} 
\label{Lm_4.4.12}
The stack $\Sh^{d,d'}_-$ is stratified by locally closed substacks $\cX_{\cJ}\cap\cY_{\cJ'}$ indexed by pairs $(\cJ,\cJ')\in \Sigma^2$ such that (C) holds for $(\cJ,\cJ')$.
In particular, if $i\in \cJ$ is the smallest element, $j'\in\cJ'$ is the biggest element then $\cJ\cup\cJ'\subset\{i,\ldots, j'\}$. \QED
\end{Lm} 
\begin{Rem} 
\label{Rem_4.4.13}
If $f: \cZ\to\cZ'$ is a morphism of finite type between  algebraic stacks, $F$ is a smooth $\Qlb$-sheaf on $\cZ$, assume all the fibres of $f$ are of dimension $\le d$. Assume $\cZ$ is stratified by locally closed substacks $i_a: \cZ^a\hook{}\cZ$ indexed by $a\in\cA$. Let $f^a: \cZ^a\to\cZ'$ be the restriction of $f$. Then $\R^{2d}f_!F$ admits a filtration by subsheaves with the successive quotients being $\R^{2d}f^a_!i_a^*F$, $a\in\cA$. We refer to $\R^{2d}f^a_!i_a^*F$ as the {\rm contribution} of the stratum $\cZ^a$ to the highest direct image $\R^{2d}f_!F$.
\end{Rem}

\sssec{Stratifications II} Recall the subset $\Inv\subset S_I$ defined in Section~\ref{Sect_1.2.8_normalization}. An involution $w\in \Inv$ writes as a disjoint product of two-cycles $w=(i_1, j_1)(i_2, j_2)\ldots (i_r, j_r)$ with $i_k< j_k$ for $1\le k\le r$. 
As in (\cite{RS}, Section~5.3), for such an involution $w$ set $\Hi(w)=\{j_1,\ldots, j_r\}$ and $\Lo(w)=\{i_1,\ldots, i_r\}$. As in \select{loc.cit.}, we call $\Hi(w)$ (resp., $\Lo(w)$) the \select{high points} (resp., \select{low points}) of $w$. 

\begin{Pp} 
\label{Pp_4.3.10}
i) For $\cJ,\cJ'\in\Sigma$ the stratum $\cX_{\cJ}\cap\cY_{\cJ'}$ does not contribute to 
\begin{equation}
\label{highest_direct_image_Pp_4.3.10}
R^{-2d'}(\div^{\nu}_-)_!\Qlb
\end{equation}
unless $\cJ\cap\cJ'=\emptyset$ and the condition (C) holds. 

\smallskip\noindent
ii) Assume $\cJ\cap\cJ'=\emptyset$ and the condition (C). Then $\cX_{\cJ}\cap\cY_{\cJ'}$ admits a stratification by locally closed substacks $\cX^w$ indexed by $w\in\Inv$ such that $\Hi(w)=\cJ'$, $\Lo(w)=\cJ$. 

For such $w$ let $\alpha: I\to\bar I$ be the unique element of $\cE$ such that the classes of the corresponding equivalence relation on $I$ are precisely the $w$-orbits. The composition $\cX^w\hook{} \Sh^{d,d'}_-\to V^{d,d'}_-$ factors uniquely through the closed subscheme $\norm_{\alpha}: V_{\alpha}\hook{} V^{d,d'}_-$, and the contribution of $\cX^w$ to (\ref{highest_direct_image_Pp_4.3.10}) is $(\norm_{\alpha})_!\Qlb$.
\end{Pp}

\begin{Lm} 
\label{Lm_coproduct}
Let $F_1, F_2\subset F$ be torsion sheaves on $X$, let $\tilde F$ be the coproduct of the diagram $F_1\gets F_1\cap F_2\to F_2$. Then the induced map $\tilde F\to F$ is injective. \QED
\end{Lm}
\sssec{Proof of Proposition~\ref{Pp_4.3.10}} ii) Consider two subsets $\cJ,\cJ'\in\Sigma$ such that $\cJ\cap\cJ'=\emptyset$, and the condition (C) holds. Set $\bar\cJ=\{1,\ldots, d\}$.
Equip $\cJ$ (resp., $\cJ'$) with the natural order:  $i_1\le i_2$ iff $i_1$ is less of equal to $i_2$. These orders are linear, and so yield bijections that we denote 
$$
b: \bar\cJ\,\iso\, \cJ, \;\; b': \bar\cJ\,\iso\, \cJ'
$$ 
If $i\in\bar\cJ$ then $\ell(\bar F_{b(i)})=i$. If $j\in\bar\cJ$ then $\ell(\cF_{b'(j)})=j$. 

 Denote by $_d^dJ$ the set of matrices $e=(e_i^j)$ with $i,j\in \bar\cJ$, $e_i^j\in\ZZ_+$. For a point of $\cX_{\cJ}\cap\cY_{\cJ'}$ define the matrix $e\in {_d^dJ}$ by 
\begin{equation}
\label{matrix_e_def}
\sum_{k=1}^i \sum_{l=1}^j 
e_k^l=\ell(\bar F_{b(i)}\cap F_{b'(j)})=\ell(\bar F_{b(i)}\cap\cF_{b'(j)})
\end{equation} 
for $i,j\in\bar\cJ$. From Lemma~\ref{Lm_coproduct} one gets by induction that indeed $e_i^j\in\ZZ_+$ for any $i,j\in\bar\cJ$.

 The matrix $e$ has the properties:
\begin{itemize}
\item $e_i^j\in\{0,1\}$.
\item If $j\in\bar\cJ$ then $\sum_{k=1}^d e^j_k=1$. 
If $i\in \bar\cJ$ then $\sum_{j=1}^d e^j_i=1$. So, 1 appears precisely once in each row (resp., each column) of $e$.
\end{itemize} 
We get a unique $\bar w\in S_d$ such that $e^j_{\bar wj}=1$ for all $j\in\bar\cJ$, and $e^j_i=0$ unless $i=\bar wj$. Define the unique involution $w\in \Inv$ by the properties: 
\begin{itemize}
\item[P1)] $w\mid_{\cJ'}=b\bar w (b')^{-1}$; 
\item[P2)] $I-(\cJ\cup\cJ')$ is the set of fixed points of $w$. 
\end{itemize}
Conversely, given an involution $w\in \Inv$ with $\Hi(w)=\cJ', \Lo(w)=\cJ$, there is a unique $\bar w\in S_d$ satisfying the property P1). Define the locally closed substack $\cX^w\subset \cX_{\cJ}\cap \cY_{\cJ'}$ by fixing the corresponding $\bar w\in S_d$, or equivalently, the corresponding matrix $e$ satisfying (\ref{matrix_e_def}). This is the desired stratification of $\cX_{\cJ}\cap\cY_{\cJ'}$ by the locally closed substacks $\cX^w$. 

 Denote by $f^w: \cX^w\to (\Sh^1_0)^{I-\cJ'}$ the map sending $(F,F', (F_i))$ to the collection $(F_i/F_{i-1})_{i\in I-\cJ'}$. This is a generalized affine fibration of rank zero. For $i\in I-\cJ'$ we then let $x_i=\div(F_i/F_{i-1})$. For $i\in\cJ'$ we let $x_i=x_{wi}$.   

 As we have seen in Section~\ref{Sect_Stratifications_I}, for a point of $\cX_{\cJ}$ one has $D=\sum_{j\in\cJ} x_j$ and $D'=\sum_{j\in I-\cJ} x_j$. So, the composition $\cX^w\hook{}\Sh^{d,d'}_-\to V^{d,d'}_-$ factors uniquely through the closed subscheme $V_{\alpha}\hook{} V^{d,d'}_-$, and the contribution of $\cX^w$ to $\R^{-2d'}(\div^{\nu}_-)_!\Qlb$ is $(\norm_{\alpha})_!\Qlb$. 
 
\medskip\noindent 
i) This is obtained from the dimension estimates and is left to a reader. If $\cJ\cap\cJ'\ne\emptyset$, and (C) holds, one may also stratify $\cX_{\cJ}\cap\cY_{\cJ'}$ by fixing the relative position of the two flags $(\bar F_i), (\cF_j)$ on $F$ as in the case $\cJ\cap\cJ'=\emptyset$. \QED

\begin{Rem} 
\label{Rem_4.3.14}
In the situation of Proposition~\ref{Pp_4.3.10} ii) for certain $\bar w\in S_d$ the stack $\cX^w$ is empty. For example, consider the case $d=d'=2$, $\cJ=\{1,3\}, \cJ'=\{2,4\}$. Then $\cX_{\cJ}\cap\cY_{\cJ'}=\cX^w$, where $w=(12)(34)$. It corresponds to $\bar w=\id\in S_2$. Indeed, in this case the stratification is given by the relative position of the two subsheaves $\bar F_1, \cF_2\in\Sh^1_0$ in $F\in\Sh^2_0$. However, from the diagram (\ref{diag_two_flags_on_F}) we know that $\cF_2\subset \bar F_2$, and for a point of $\cX_{\cJ}\cap\cY_{\cJ'}$ we have $\bar F_1=\bar F_2$ by definition. So, $\cF_2=\bar F_1$ for any point of $\cX_{\cJ}\cap\cY_{\cJ'}$. In general, if $\cX^w\ne \emptyset$ then $w$ has to be compatible with the diagram (\ref{diag_two_flags_on_F}).
\end{Rem}

For the convenience of the reader, we give several examples of the stratifications of $\Sh^{d,d'}_-$ appeared in Section~\ref{Sect_Stratifications_I} and Proposition~\ref{Pp_4.3.10}.

\sssec{Example $d=1, d'>1$} This is a first nontrivial example. We have $\ell(F)=1, \ell(F')=d'$. Let $\cJ'=\{j\}, \cJ=\{i\}$. The substack $\cX_{\cJ}\cap\cY_{\cJ'}\subset \Sh^{d,d'}_-$ is given by the properties: 
\begin{itemize} 
\item $F_{i-1}\subset F', F_i\cap F'=F_{i-1}$,
\item $F\cap F_{j-1}=0$, $F\subset F_j$. 
\end{itemize}
(The first conditions garantee that $F\cap F_{i-1}=0$, so if $F\subset F_j$ then $j\ge i$). 

 Let $\bar\cX$ be the stack classifying $F\in \Sh_0^1, F'\in\Sh^{d'}_0$ with $\div(F)\le \div(F')$ and a complete flag $(F'_1\subset\ldots F'_{d'})$ on $F'$. For each $1\le i\le d'+1$ we get an isomorphism 
$$
\bar\cX\,\iso\,\cX_{\cJ}\cap\cY_{\cJ}
$$ 
sending $(F,F', (F'_j))$ to $(F,F',(F_j))$, where $(F_j)$ is the complete flag on $F\oplus F'$ given by 
$$
\left\{
\begin{array}{ll}
F_j=F'_j, & j<i\\
F_j=F_{j-1}\oplus F, & j\ge i
\end{array}
\right.
$$
The composition $\bar\cX\,\iso\, \cX_{\cJ}\cap\cY_{\cJ}\toup{\div^{\nu}_-} V^{d,d'}_-$ sends $(F, F', (F'_j))$ to the collection $x=\div F, D'=\div F'$ and $(x_1,\ldots, x_{i-1}, x, x_i,\ldots, x_{d'})$, where $x_i=\div(F'_i/F'_{i-1})$. Using Remark~\ref{Rem_1.2.20}, one sees that the contribution of the stratum $\cX_{\cJ}\cap\cY_{\cJ}$ to $\R^{-2d'}(\div^{\nu}_-)_!\Qlb$ is zero. 
  
  Assume $1\le i<j\le d'+1$. Let $w=(ij)\in\Inv$ then $\cX_{\cJ}\cap\cY_{\cJ'}=\cX^w$. The map $f^w: \cX^w\to (\Sh_0^1)^{I-\cJ'}$ is a generalized affine fibration. Indeed, taking the quotient by $F_{i-1}$ and replacing $F'$ by $F_j$ this claim is reduced to the special case $i=1, j=d'+1$. In the latter case the map $f^w$ is described as follows.
\begin{Lm} Let $d=1, d'>1$, $i=1, j=d'+1$ and $w=(ij)$.
The stack $\cX^w$ classifies: $F, F_1\in \Sh^1_0$, an isomorphism $F\,\iso\, F_1$, a complete flag $\cF'_2\subset\ldots\subset\cF'_{d'}$ of torsion sheaves on $\cF'_{d'}\in \Sh^{d'-1}_0$ as in Section~\ref{Sect_Stratifications_I}. So, $\ell(\cF'_k)=k-1$ for $2\le k\le d'$. Here for $2\le i\le d'$ one has canonically
$$
F_i/F_{i-1}\,\iso\, \cF'_i/\cF'_{i-1}
$$ 
Besides, $F\,\iso\, F_1\,\iso\, F_{d'+1}/F_{d'}$. 
\end{Lm}
\begin{proof}
The stack $\cX_{\cJ}\cap\cY_{\cJ'}$ is given by two conditions: $F_1\cap F'=0, F\cap F_{d'}=0$. Using the notations from Section~\ref{Sect_Stratifications_I}, we get $F_k=F_1\oplus \cF'_k$ for $2\le k\le d'$, and $F_{d'+1}=F_{d'}\oplus F$. So, $F\oplus F_1\oplus \cF'_{d'}=F\oplus F'$. Consider the natural projection 
$$
F'=\cF'_{d'+1}\to (F\oplus F_1\oplus \cF'_{d'})/\cF'_{d'}\,\iso\, F\oplus F_1
$$
Both its components are nonzero. Thus, the induced maps $\cF'_{d'+1}/\cF'_{d'}\to F_1$ and $\cF'_{d'+1}/\cF'_{d'}\to F$ are both nonzero, hence isomorphisms. This defines the desired isomorphism $F\,\iso\, F_1$, and $F'$ as the unique subsheaf in $F\oplus F_1\oplus \cF'_{d'}$ containing $\cF'_{d'}$ and mapping diagonally to $F\oplus F_1$.
\end{proof} 
  
\sssec{Example $d=d'=2$} Only the strata $\cX_{\{1,3\}}\cap\cY_{\{2,4\}}$ and $\cX_{\{1,2\}}\cap\cY_{\{3,4\}}$ contribute to (\ref{highest_direct_image_Pp_4.3.10}). One gets $\cX_{\{1,3\}}\cap\cY_{\{2,4\}}=\cX^w$ for $w=(12)(34)$ as in Remark~\ref{Rem_4.3.14}. The locus $\cX_{\{1,2\}}\cap\cY_{\{3,4\}}$ consists of two strata $\cX^w, \cX^{w'}$ for $w=(13)(24)$ and $w'=(14)(23)$. The substack $\cX^w\subset \cX_{\{1,2\}}\cap\cY_{\{3,4\}}$ is given by $\bar F_1\subset F_3$, and $\cX^{w'}$ is its complement in $\cX_{\{1,2\}}\cap\cY_{\{3,4\}}$. 
 
 The stack $\cX^w$ classifies: $F_1\in\Sh^1_0, F_2/F_1\in\Sh^1_0$ with divisors $x_1, x_2$, for which we set $x_3=x_1, x_4=x_2$, and an extension $0\to F_1\to F_2\to F_2/F_1\to 0$ on $X$. 
 
 The stack $\cX^{w'}$ classifies: $\bar F_1, \cF_3\in\Sh^1_0$ with $x_1=\div(\bar F_1), x_2=\div(\cF_3)$, for which we set $x_3=x_2$, $x_4=x_1$. So, the map $f^{w'}: \cX^{w'}\to (\Sh^1_0)^2$ sending the above point to $(F_1, F_2/F_1)$ is an isomorphism.
 
\sssec{End of the proof of Theorem~\ref{Thm_4.2.11}} 
\label{Sect_4.3.18}
The normalization of $V^{d,d'}_-$ is described in Section~\ref{Sect_1.2.8_normalization}. We first stratify $\Sh^{d,d'}_-$ by locally closed substacks $\cX_{\cJ}\cap\cY_{\cJ'}$ indexed by pairs $\cJ,\cJ'\in\Sigma$. By Section~\ref{Sect_Stratifications_I}, this locus is empty unless the condition (C) holds. By Proposition~\ref{Pp_4.3.10}, only the strata with $\cJ\cap\cJ'=\emptyset$ contribute to (\ref{highest_direct_image_Pp_4.3.10}). For $\cJ,\cJ'\in\Sigma$ with $\cJ\cap\cJ'=\emptyset$ such that (C) holds, 
we stratify $\cX_{\cJ}\cap\cY_{\cJ'}$ by locally closed substacks $\cX^w$ as in Proposition~\ref{Pp_4.3.10} indexed by $w\in\Inv$ such that $\Hi(w)=\cJ'$, $\Lo(w)=\cJ$. Let $\alpha: I\to\bar I$ be the surjection such that the classes of the corresponding equivalence relation on $I$ are precisely $w$-orbits. By Proposition~\ref{Pp_4.3.10}, the contribution of the stratum $\cX^w$ to (\ref{highest_direct_image_Pp_4.3.10}) identifies with $(\norm_{\alpha})_!\Qlb$. Now $(\norm_{\alpha})_!\Qlb[d']$ is an irreducible perverse sheaf. So, $(\R^{-2d'}(\div^{\nu}_-)_!\Qlb)[d']$ is perverse and admits a filtration with the associated graded
$$
\mathop{\oplus}\limits_{\alpha\in\cE} (\norm_{\alpha})_!\Qlb[d']\,\iso\, \IC(V^{d,d'}_-)
$$
Any such filtration splits canonically into a direct sum, because the summands are irreducible perverse sheaves supported on different irreducible components (cf. Remark~\ref{Rem_2.4.6}). \QED

\medskip

 Finally, the proof of Proposition~\ref{Pp_2.2.5} is complete. So, Theorem~\ref{Th_2.1.2} is proved.

\ssec{Comparison with \cite{RS}} 
\label{Sect_4.4} Section~\ref{Sect_4.4} is not used in the paper and may be skipped. We include it to give another description of the stratification used in Proposition~\ref{Pp_4.3.10}. Consider the stack $\cD$ classifying vector spaces $V,V'$ of dimensions $d,d'$ and a complete flag $V_1\subset\ldots\subset V_{d+d'}=V\oplus V'$ of vector subspaces. The set of isomorphism classes of $\cD$ is finite and was described in \cite{RS}. In this section we introduce a morphism $f: \Sh^{d,d'}_-\to \cD$ and a stratification of $\cD$ such that the stratification of $\Sh^{d,d'}_-$ used in Proposition~\ref{Pp_4.3.10} is obtained as the pull-back by $f$. 

\sssec{} Consider the stack $\cD$ classifying vector spaces $V,V'$ of dimensions $d,d'$ and a complete flag $V_1\subset\ldots\subset V_{d+d'}=V\oplus V'$ of vector subspaces. The set of points of $\cD$ is finite and described in (\cite{RS}, Section~5.3). It is naturally in bijection with
$$
\bar\cE=\{(w, \cJ)\in \Inv\times\Sigma \mid \Hi(w)\subset \cJ, \Lo(w)\cap \cJ=\emptyset\}
$$
Consider the map $f: \Sh^{d,d'}_-\to \cD$ sending $(F,F',(F_i))$ to the collection $(V, V', (V_i))$ with 
$V=\H^0(X,F), V'=\H^0(X, F')$, and $V_i=\H^0(X, F_i)$. 
There is a stratification of $\cD$ whose preimage under $f$ gives
the stratification of $\Sh^{d,d'}_-$ that we used in Section~\ref{Sect_Stratifications_I} and Proposition~\ref{Pp_4.3.10}. 
However, this is not the stratification by the classes of isomorphisms of points of $\cD$. We give some details for the convenience of the reader.

We define stratifications of $\cD$ by locally closed substacks $\cD^{\cJ}$ with $\cJ\in\Sigma$ (resp., by locally closed substacks $\cD_{\cJ}$ with $\cJ\in\Sigma$) such that $f^{-1}(\cD^{\cJ})=\cY_{\cJ}$ and $f^{-1}(\cD_{\cJ})=\cX_{\cJ}$ for any $\cJ\in\Sigma$. 
\sssec{}  
\label{Sect_4.3.20} 
If we pick bases $\{e_1,\ldots, e_d\}$ in $V$ and $\{e_{d+1},\ldots, e_{d+d'}\}$ in $V'$ then we get a complete flag in $k^{d+d'}$ given by some $gB_{d+d'}\in \GL_{d+d'}/B_{d+d'}$. 
Set $K=\GL_d\times\GL_{d'}$. Forgetting these bases, one gets an isomorphism $\cD\,\iso\, K\backslash \GL_{d+d'}/B_{d+d'}$ with the the stack quotient.

 The variety of decompositions $k^{d+d'}=\oplus_{i=1}^{d+d'} W_i$ into 1-dimensional vector subspaces identifies naturally with $\GL_{d+d'}/T_{d+d'}$. Here $T_{d+d'}$ is the torus of diagonal matrices. Namely, to $gT_{d+d'}$ we associate $W_i=g\Vect(e_i)$. 

 If $V,V'$ are vector spaces of dimensions $d,d'$ and $\oplus_{i=1}^{d+d'}W_i=V\oplus V'$ is a decomposition into 1-dimensional subspaces, it gives rise to a torus $T_W=\{g\in \GL(V\oplus V')\mid g(W_i)=W_i\}$ for any $i$. Let $\theta$ be the automorphism of $V\oplus V'$ acting as $1$ on $V$ and as $-1$ on $V'$. The torus $T_W$ is $\theta$-stable if $\theta T_W\theta^{-1}=T_W$. This means that there is an involution $w\in S_{d+d'}$ such that $\theta(W_i)=W_{wi}$ for any $i$. 
 
 Write $\wt\cD$ for the stack classifying vector spaces  $V,V'$ of dimensions $d,d'$, and a decomposition into 1-dimensional vector subspaces $V\oplus V'=\oplus_{i=1}^{d+d'} W_i$. If we pick bases of $V,V'$ as above, a point of $\wt\cD$ together with these bases gives an element of $\GL_{d+d'}/T_{d+d'}$. Forgetting the bases, one gets an isomorphism with the stack quotient $
 K\backslash \GL_{d+d'}/T_{d+d'}\,\iso\, \wt\cD$. 
 
 Denote by $\bar\cD\subset\wt\cD$ the closed substack given by the property that $T_W$ is $\theta$-stable. For an involution $w\in S_{d+d'}$ denote by $\bar\cD^w\subset \bar\cD$ the locally closed substack given by the property that $\theta(W_i)=W_{wi}$ for any $i\in I$. This is a stratification of $\bar\cD$ indexed by the set $\Inv$ of involutions in $S_{d+d'}$. 
 
 Let $c\in\GL_{d+d'}$ be the automorphism acting as 1 on $\Vect(e_1,\ldots, e_d)$ and as $-1$ on $\Vect(e_{d+1},\ldots, e_{d+d'})$. Let $\theta_c$ be the inner automorphism of $\GL_{d+d'}$ sending $g$ to $cgc^{-1}$. Then $K=(\GL_{d+d'})^{\theta_c}$. As in (\cite{RS}, Section~1.2), set 
$$
\cV=\{g\in\GL_{d+d'}\mid g^{-1}\theta_c(g)\in N_{d+d'}\},
$$ 
here $N_{d+d'}$ is the normalizer of $T_{d+d'}$ in $\GL_{d+d'}$. Then $\cV$ is stable under the action of $T_{d+d'}$ by right translations. Now 
$$
\cV/T_{d+d'}\subset \GL_{d+d'}/T_{d+d'}
$$ 
is the closed subvariety of those decompositions $k^{d+d'}\,\iso\, \oplus_{i=1}^{d+d'} W_i$
for which $T_W\subset\GL_{d+d'}$ is $\theta_c$-stable. Note that $\cV$ is stable under the action of $K$ by left translations. Picking bases of $V, V'$ as above, this gives an isomorphism with the stack quotient
\begin{equation}
\label{iso_for_bar_cD}
K\backslash\cV/T_{d+d'}\,\iso\, \bar\cD
\end{equation}

 For $x\in N_{d+d'}$, $w\in S_{d+d'}$ the image of $x$ in $S_{d+d'}$ is $w$ iff for any $i\in I$, $x\Vect(e_i)=\Vect(e_{wi})$. Let $N^w_{d+d'}\subset N_{d+d'}$ be the closed subscheme of $x\in N_{d+d'}$ over $w$. Let also $\cV^w\subset\cV$ be the closed subscheme given by requiring that $g^{-1}\theta_c(g)\in N^w_{d+d'}$. Then (\ref{iso_for_bar_cD}) restricts for any $w\in \Inv$ to
an isomorphism 
$$
K\backslash \cV^w/T_{d+d'}\,\iso\, \bar\cD^w
$$
   
   By (\cite{RS}, Proposition~1.2.2), number of points of $\bar\cD$ and $\cD$ is finite. Consider the map $\wt\cD\to \cD$ sending $(V,V', (W_i))$ to $(V,V', (V_i))$, where $V_i=\oplus_{j=1}^i W_j$. Let $f_{\cD}: \bar\cD\to\cD$ be its restriction to $\bar\cD$. Now  (\cite{RS}, Proposition~1.2.2) implies that $f_{\cD}$ induces a bijection on the set of points of $\bar\cD$ and $\cD$.  
   
 For $w\in \Inv$ the stack $\bar\cD^w$ has a finite number of points, this number could be bigger than one. For $w\in\Inv$ set $\cD^w=f_{\cD}(\bar\cD^w)$ viewed a locally closed substack.

\sssec{} Let $P\subset \GL_{d+d'}$ be the parabolic preserving $\Vect(e_1,\ldots, e_d)$ in $k^{d+d'}$. So, $K$ is the standard Levi of $P$, and $B_{d+d'}\subset P$. Let $\cD_P$ be the stack classifying 
a complete flag of vector spaces $V_1\subset\ldots\subset V_{d+d'}$ and a subspace $V\subset V_{d+d'}$. As above, we get an isomorphism 
$$
\cD_P\,\iso\, P\backslash \GL_{d+d'}/B_{d+d'}
$$ 
Let $\eta_P: \cD\to\cD_P$ be the map forgetting $V'$.  

  Set $\bar\cJ=\{1,\ldots, d\}\subset I$. Let $W_{\bar\cJ}\subset W=S_{d+d'}$ be the stabilizer of $\bar\cJ$. We have a bijection $W/W_{\bar \cJ} \,\iso\, \Sigma$ sending $wW_{\bar\cJ}\in W/W_{\bar \cJ}$ to $w\cJ\in\Sigma$. The set of points of $\cD_P$ is in bijection with $\Sigma$. Namely, to $wW_{\bar\cJ}\in W/W_{\bar \cJ}$ with $\cJ=w\bar\cJ\in\Sigma$ we associate $Pw^{-1}B_{d+d'}$. Write $\cD_P^{\cJ}$ for the corresponding locally closed substack of $\cD_P$. For $\cJ\in \Sigma$ let $\cD^{\cJ}$ be the preimage of $\cD_P^{\cJ}$ under $\eta_P$. 
  
   According to (\cite{RS}, Section~5.3), for $w\in Inv, \cJ\in\Sigma$ the stack $\cD^{w,\cJ}=\cD^w\cap\cD^{\cJ}$ is either empty or has precisely one point. Moreover, it is nonempty iff
$$
\Hi(w)\subset \cJ, \Lo(w)\cap \cJ=\emptyset   
$$
This way one gets a bijection between the set of points of $\cD$ and  $\bar\cE$. 
For $\cJ\in\Sigma$ one has a stratification
$$
\cD^{\cJ}=\mathop{\sqcup}\limits_{(w,\cJ)\in\bar\cE} \cD^{w,\cJ}
$$ 
\sssec{} For $w\in S_{d+d'}$ write $\Fix(w)$ for the set of fixed points of $w$ on $I$. For $w\in\Inv, \cJ\in \Sigma$ let $\cD^{w,\cJ}$ be nonempty. Then the unique point of $\cD^{w,\cJ}$ admits the following presentation $(V, V', (V_i))$. There is a decomposition $V\oplus V'=\oplus_{i\in I} W_i$ with $\theta(W_i)=W_{wi}$ for all $i$ with the propeties: 
\begin{itemize} 
\item $V_i=W_1\oplus\ldots\oplus W_i$ for all $i\in I$; 
\item if $j\in \Fix(w)\cap \cJ$ then $W_j\subset V$;
\item if $j\in \Fix(w)- \cJ$ then $W_j\subset V'$.
\end{itemize}

 In addition, let $1\le i<j\le d+d'$ with $w(i)=j$. Pick $0\ne w_i\in W_i$, let $w_j=\theta(w_i)$. Then $W_i\oplus W_j$ is $\theta$-stable with a base $\{w_i+w_j, w_i-w_j\}$. Here $w_i+w_j\in V, w_i-w_j\in V'$. 

\sssec{} Consider also the parabolic $P^-\subset\GL_{d+d'}$ preserving $\Vect(e_{d+1},\ldots, e_{d+d'})$ in $k^{d+d'}$. Let $\cD_{P^-}$ be the stack classifying a complete flag of vector spaces $(V_1\subset\ldots\subset V_{d+d'})$ and a subspace $V'\subset V_{d+d'}$. We have a projection $\eta_{P^-}: \cD\to \cD_{P^-}$ forgetting $V$. As above, we have an isomorphism with the stack quotient $\cD_{P^-}\,\iso\, P^-\backslash \GL_{d+d'}/B_{d+d'}$. 
 
 For $wW_{\bar\cJ}\in W/W_{\bar\cJ}$ with $\cJ=w\bar\cJ\in\Sigma$ the orbit $P^-w^{-1}B_{d+d'}$ defines a locally closed substack $\cD_{P^-}^{\cJ}\subset\cD_{P^-}$. Write $\cD_{\cJ}$ for the preimage of $\cD_{P^-}^{\cJ}$ under $\eta_{P^-}$. 
 
 For $w\in \Inv, \cJ\in\Sigma$ set $\cD^w_{\cJ}=\cD^w\cap \cD_{\cJ}$. One checks that if $\cD^w_{\cJ}$ is nonempty then it consists of one point. Moreover, for $w\in \Inv, \cJ\in\Sigma$ the stack $\cD^w_{\cJ}$ is nonempty iff
$$
\Hi(w)\cap\cJ=\emptyset, \Lo(w)\subset\cJ
$$
 
\section{Proof of Theorem~\ref{Thm_2.2.2}}

\sssec{} We have the diagram extending (\ref{diag_linear_periods_main})
$$
\begin{array}{ccccc}
\cZ_n & \toup{q_{\cZ}} & (\Bun_n\times\Bun_n) & \toup{\det\times\det} (\Pic X\times\Pic X)\\
\downarrow\lefteqn{\scriptstyle \nu_{\cZ}} &&  \downarrow\lefteqn{\scriptstyle \nu_n}\\
\cM_{2n} & \toup{q_{2n}} & \Bun_{2n},
\end{array}
$$
where $q_{\cZ}$ sends $(\Omega^{2n-1}\hook{} L, L')$ to $(L,L')$, and $q_{2n}$ sends $(\Omega^{2n-1}\hook{} M)$ to $M$. 
 
 Let $\cU_{2n}\subset\Bun_{2n}$ be the open substack given by the property that $\Ext^1(\Omega^{2n-1}, M)=0$ for $M\in \Bun_{2n}$. 
Let $\cU_{2n}^n\subset\Bun_n$ be the open substack given by $\Ext^1(\Omega^{2n-1}, L)=0$ for $L\in \Bun_n$. For $d\in\ZZ$ let $\Bun_{2n, d}\subset \Bun_{2n}$ be the component given by $\deg M=d+\deg(\Omega^{(2n-1)+\ldots+1})$ for $M\in\Bun_{2n}$. So, $q_{2n}: {\cM_{2n, d}}\to \Bun_{2n, d}$. 

\sssec{} Assume $d+d'$ large enough, so that $\Aut_E$ over $\Bun_{2n, d+d'}$ is the extension by zero under $\cU_{2n}\hook{} \Bun_{2n}$. Indeed, by (\cite{FGV}, Lemma~3.3), any $M\in \Bun_{2n, d+d'}$ not lying in $\cU_{2n}$ is very unstable in the sense of \select{loc.cit.}, and the $*$-fibre of $\Aut_E$ at $M$ vanishes. If $(L, L')\in \nu_n^{-1}(\cU_{2n})$ then $L\in\cU_{2n}^n$. So, we get an exact triangle
$$
(q_{\cZ})_!\Qlb\to \Qlb[-2M]\to \Qlb
$$
over $\nu_n^{-1}(\cU_{2n})\cap (\det\times\det)^{-1}((\Pic X\times\Pic X)^{d,d'})$, here $M=d+(g-1)(n-2n^2)$ is the relative dimension of $q_{\cZ}$ over $\cZ_n^{d,d'}$. This yields an exact triangle
\begin{equation}
\label{triangle_main}
\pi_!\nu_{\cZ}^*\cK_{2n, E}[\dimrel(\nu_{\cZ})+M]\to Per^{d,d'}_E\to Per^{d,d'}_E[2M]
\end{equation}
on $(\Pic X\times\Pic X)^{d,d'}$. 

 If $d>d'$ and $d$ is large enough then, by Theorem~\ref{Th_2.1.2}, $\pi_!\nu_{\cZ}^*\cK_{2n, E}$ vanishes. 
Under these assumptions then (\ref{triangle_main}) shows that $Per^{d,d'}_E=0$. Using the $S_2$-symmetry of $Per_E$ and the equivariance property (\ref{equiv_property_P_E}), one obtains i). 
 
\sssec{} From i) and the equivariance property (\ref{equiv_property_P_E}), one gets $N\in\ZZ$ such that for any $d,d'\in\ZZ$, $Per^{d,d'}_E$ is placed in usual degrees $\le N$. The second claim follows now from the exact triangle (\ref{triangle_main}) and Theorem~\ref{Th_2.1.2}. Theorem~\ref{Thm_2.2.2} is proved. \QED

\section{Proof of Theorem~\ref{Thm_2.6.2}}

\sssec{} Let $P\subset \GG$ be the Siegel parabolic. Then $\Bun_P$ is the stack classifying $L\in\Bun_2,\cA\in\Bun_1$ and an exact sequence $0\to \Sym^2 L\to ?\to \cA\to 0$ on $X$ (cf. Remark~\ref{Rem_very_last}). Let $\cS_P$ be the stack classifying $L\in\Bun_2,\cA\in\Bun_1$ and a section $s: \Sym^2 L\to \cA\otimes\Omega$ on $X$. Let $h: \Bun_1\times\Bun_1\to\cS_P$ be the map sending $(L_1, L_2)$ to $L=L_1\oplus L_2$ with the natural symmetric form $\Sym^2 L\to \L_1\otimes L_2=\cA\otimes\Omega$, so $L_i$ are isotropic in $L$. 

 One gets the diagram
$$
\begin{array}{ccccccc}
\Bun_P &&&&\cS_P & \getsup{h} & \Bun_1\times\Bun_1\\
\downarrow\lefteqn{\scriptstyle p_{\GG}} &\searrow\lefteqn{\scriptstyle f} &&\swarrow\lefteqn{\scriptstyle f^{\vee}}\\
\Bun_{\GG} && \Bun_2\times\Bun_1,
\end{array}
$$
where $f,f^{\vee}$ are the projections sending the above points to $(L,\cA)$, and $p_{\GG}$ is the extension of scalars map. The maps $f,f^{\vee}$ are dual generalized vector bundles, we denote by $\Four_{\psi}: \D^{\prec}(\Bun_P)\to \D^{\prec}(\cS_P)$ the corresponding Fourier transform.

  Recall that $\HH\,\iso\, \GL_4/\mu_2$. Set $R=(\GL_2\times\GL_2)/\mu_2$, where $\mu_2$ is included diagonally. So, $R$ is a subgroup of $\HH$. The homomorphism $\det\times\det: \GL_2\times\GL_2\to\Gm\times\Gm$ factors through $R\to \Gm\times\Gm$. Consider the diagram
$$
\Bun_{\HH}\getsup{\alpha} \Bun_R\toup{\beta} \Bun_1\times\Bun_1\toup{\varrho\times\varrho} \Bun_1\times\Bun_1
$$
where $\alpha,\beta$ are the corresponding extension of scalars maps, and $\varrho: \Bun_1\to\Bun_1$ sends $\cB$ to $\cB^*\otimes\Omega$. 
  
  The following is established in (\cite{L1}, Proposition~11).
\begin{Pp} 
\label{Pp_6.0.2}
For any $F\in\D^-(\Bun_{\HH})_!$ there is an isomorphism (up to a cohomological shift)
$$
h^*\Four_{\psi}(p_{\GG}^*F_{\GG}(F))\,\iso\, (\varrho\times\varrho)_!\beta_!\alpha^*F
$$
\QED
\end{Pp}

 The following diagram commutes, where the square is cartesian
$$
\begin{array}{ccccc}
\Bun_4 & \getsup{\bar\alpha} & \Bun_2\times\Bun_2\\ 
\downarrow\lefteqn{\scriptstyle\rho} && \downarrow &\searrow\lefteqn{\scriptstyle \det\times\det}\\ 
\Bun_{\HH} & \getsup{\alpha} & \Bun_R & \!\!\!\toup{\beta} & \Bun_1\times\Bun_1,
\end{array}
$$
and the vertical arrows are the natural extension of scalars maps.
  
  By Proposition~\ref{Pp_6.0.2}, it suffices to show that $\beta_!\alpha^*K_{E^*,\chi^*, \HH}$ is not zero. By definition, $\rho^*K_{E,\chi, \HH}\,\iso\, \Aut_E$ over $\Bun_4$. Since $E$ admits a symplectic form, Theorem~\ref{Thm_2.2.2} implies that $(\det\times\det)_!\bar\alpha^*\Aut_E$ in nonzero on $\Bun_1\times\Bun_1$. This concludes the proof of Theorem~\ref{Thm_2.6.2}. \QED
  
\begin{Rem} 
\label{Rem_very_last}
Let $S$ be a scheme. Assume given a semidirect product $P=M\rtimes U$ of affine group schemes on $S$, where $U\subset P$ is a normal closed subgroup, $M$ acts on $U$ by conjugation. Then the groupoid of $P$-torsors on $S$ is canonically equivalent to the groupoid of pairs: an $M$-torsor $\cF_M$ on $S$, and a $U_{\cF_M}$-torsor on $S$. Here $U_{\cF_M}$ is the group scheme on $S$, the twist of $U$ by the $M$-torsor $\cF_M$. 
\end{Rem}
  
\appendix  
\section{Antistandard Levi subgroups in reductive groups}
\label{Section_AppendixA}

\sssec{} In this section we introduce a notion of an antistandard Levi subgroup $M$ in a reductive group $G$. Our purpose is to prove Propositions~\ref{Pp_2.1.1} and \ref{Pp_2.1.15} below. The first one is a property of the restriction of spherical perverse sheaves under $\Gr_M\to \Gr_G$ for such Levi subgroup $M$. The second one is an application of the Casselman-Shalika formula from \cite{FGV}
for such Levi subgroups.

\sssec{} Work over an algebraically closed field $k$. Let $\cO=k[[t]]\subset k((t))=F$. Let $G$ be a connected reductive group over $k$, $T\subset B\subset G$ a maximal torus and Borel, $\Lambda$ the coweights lattice $\Lambda^+$ the dominant coweights. Let $M$ be a Levi with $T\subset M$, $\Lambda^+_M$ be the dominant coweights for $M$, $\Lambda^{pos}_M$ the $\ZZ_+$-span of simple coroots of $M$. Let $U\subset B$ be the unipotent radical, $B_M=B\cap M$. Let $U_M\subset B_M$ be the unipotent radical. Write $W$ (resp., $W_M$) for the Weyl group of $G$ (resp., of $M$). Our choices of $B$ and $B_M$ yield choices of positive coroots for $G$ and $M$. Write $w_0$ (resp., $w_0^M$) for the longest element of $W$ (resp., of $W_M$). Let $\check{\rho}$ be the half sum of positive roots for $G$, write $\le$ for the standard order on $\Lambda$. We denote by $\check{\rho}_M$ and $\le_M$ for the corresponding objects for $M$. We ignore the Tate twists everywhere.

 We do not assume that $M$ is standard, so a simple coroot of $M$ is not necessarily simple for $G$. 
 
\begin{Def} Say that $M$ is antistandard if for any simple coroot $\alpha$ of $M$ one has $\<\alpha, \check{\rho}-\check{\rho}_M\>>0$. So, such an $\alpha$ is a positive coroot for $G$ but not simple.
\end{Def} 

\sssec{}  For $\lambda\in\Lambda$ write $S^{\lambda}$ for the $U(F)$-orbit on $\Gr_G$ through $t^{\lambda}$, $S^{\lambda}_M$ is the $U_M(F)$-orbit on $\Gr_M$ through $t^{\lambda}$. For $\nu\in\Lambda^+$ we write $\Gr_G^{\nu}$ for the corresponding $G(\cO)$-orbit, $\cA_{\nu}$ for the corresponding $G(\cO)$-equivariant perverse sheaf on $\Gr_G$. Let $i: \Gr_M\to\Gr_G$ be the natural map. Write $i_{\lambda}: \Spec k\to \Gr_M$ for the point $t^{\lambda}$.
  
\begin{Pp} 
\label{Pp_2.1.1}
Let $\lambda\in\Lambda$, $\nu\in\Lambda^+$. Then \\
1) the complex $\RG_c(S^{\lambda}_M, i^*\cA_{\nu})$ is placed in degrees $\le \<\lambda, 2\check{\rho}\>$. 

\smallskip\noindent
2) Assume in addition that $M$ is antistandard. Then the above inequality is strict unless $\lambda\in -\Lambda^+_G$ and $\nu=w_0(\lambda)$. In the latter case the natural map
$$
\RG_c(S^{\lambda}_M, i^*\cA_{\nu})\to i_{\lambda}^*i^*\cA_{\nu}\,\iso\, \Qlb[-\<\lambda, 2\check{\rho}\>]
$$
is an isomorphism.
\end{Pp}  

\sssec{\bf Question} For future research we raise the following question. In the situation of Proposition~\ref{Pp_2.1.1}, describe the top cohomology group of $\RG_c(S^{\lambda}_M, i^*\cA_{\nu})$ in the degree $\<\lambda, 2\check{\rho}\>$ in general.

\ssec{Proof of Proposition~\ref{Pp_2.1.1}}

\sssec{} Now we will parametrize the spherical orbits in $\Gr_M$ which intersect both the support of $\cA_{\nu}$ and the semi-infinite orbit $S^{\lambda}_M$ and thereby obtain the set $\cI(\lambda,\nu)$ below.

For $\lambda,\nu$ as in Proposition~\ref{Pp_2.1.1} write $\cI(\lambda,\nu)$ for the set of $\mu\in\Lambda^+_M$ such that for any $w\in W_M$, $w\lambda\le_M\mu$, and for any $w\in W$, $w\mu\le\nu$. Let $f: \cI(\lambda,\nu)\to\ZZ$ be the function 
$$
f(\mu)=\<\mu, 2\check{\rho}_M\>-\<\mu', 2\check{\rho}\>,
$$
where $\mu'\in\Lambda^+$ is the unique element lying in $W\mu$. 
We derive Proposition~\ref{Pp_2.1.1} from the following.

\begin{Pp} 
\label{Pp_2.1.3}
Let $\lambda\in\Lambda$, $\nu\in\Lambda^+$. Then for any $\mu\in\cI(\lambda,\nu)$ one has
\begin{equation}
\label{inequality_mistery}
f(\mu)\le \<\lambda, 2\check{\rho}-2\check{\rho}_M\>
\end{equation}
Assume in addition $M$ antistandard. Then the above inequality is strict unless $\lambda\in -\Lambda^+_G$, $w_0^M(\lambda)=\mu$ and $\mu'=w_0(\lambda)$.  
\end{Pp}
\sssec{Proof of Proposition~\ref{Pp_2.1.1}}
Let $\bar\lambda$ be the image of $\lambda$ in $\pi_1(M)$, $\Gr_M^{\bar\lambda}$ the corresponding connected component of $\Gr_M$. Consider the scheme $Y\subset \Gr_M^{\bar\lambda}$, which is a union of $M(\cO)$-orbits for $\mu\in \cI(\lambda,\nu)$ (the scheme structure is not important here). Note that $Y\subset  \Gr_M^{\bar\lambda}\times_{\Gr_G} \ov{\Gr}_G^{\nu}$. Let $\bar Y$ be the closure of $Y$ in $\Gr_M^{\bar\lambda}$. Let $j: Y\hook{} \bar Y$ be the natural immersion. 

  We claim that the $*$-restriction $\cA_{\nu}\mid_Y$ is placed in perverse degrees $\le \<\lambda, 2\check{\rho}-2\check{\rho}_M\>$.
In addition, if $M$ is antistandard, this inequality is strict unless $\lambda\in -\Lambda^+_G$ and $\nu=w_0(\lambda)$.
 
  To see this, we must show that for $\mu\in\cI(\lambda,\nu)$, $\cA_{\nu}\mid_{\Gr_M^{\mu}}$ is placed in usual degrees $\le \<\lambda, 2\check{\rho}-2\check{\rho}_M\>-\<\mu, 2\check{\rho}_M\>$. 
Besides, if $M$ is antistandard, this inequality is strict unless $\lambda\in -\Lambda^+_G$ and $\nu=w_0(\lambda)$. Let $\mu'\in W\mu$ such that $\mu'\in\Lambda^+$. Then $\cA_{\nu}\mid_{\Gr_G^{\mu'}}$ is placed in usual degrees $\le -\<\mu', 2\check{\rho}\>$, and the inequality is strict unless $\mu'=\nu$. 
So, our claim follows from  Proposition~\ref{Pp_2.1.3}.  
 
 Consider now $j_!(\cA_{\nu}\mid_Y)$, it is placed in perverse degrees $\le \<\lambda, 2\check{\rho}-2\check{\rho}_M\>$ over $\Gr_M$. Besides, if $M$ is antistandard, this inequality is strict unless $\lambda\in -\Lambda^+_G$ and $\nu=w_0(\lambda)$. So, by (\cite{MV}, Lemma~3.9), $\RG_c(S^{\lambda}_M, j_!(\cA_{\nu}\mid_Y))$ is placed in degrees $\le \<\lambda, 2\check{\rho}\>$. Now
$$
\RG_c(S^{\lambda}_M, j_!(\cA_{\nu}\mid_Y))\,\iso\, \RG_c(S^{\lambda}_M, i^*\cA_{\nu})
$$
because $S^{\lambda}_M\cap (\bar Y-Y)=\emptyset$ by (\cite{MV}, Theorem~3.2). Proposition~\ref{Pp_2.1.1} is proved. \QED

\sssec{Proof of Proposition~\ref{Pp_2.1.3}} 

 Let $\mu\in \cI(\lambda,\nu)$. By (\cite{MV}, Theorem~3.2), one has 
$$
\dim (S^{\lambda}_M\cap \ov{\Gr}^{\mu}_M)=\<\lambda+\mu, \check{\rho}_M\>,  \;\; \dim (S^{\lambda}\cap\ov{\Gr}^{\mu'}_G)=\<\lambda+\mu', \check{\rho}\>, 
$$
and these schemes are nonempty. So, the inclusion $S^{\lambda}_M\cap \ov{\Gr}^{\mu}_M\subset S^{\lambda}\cap\ov{\Gr}^{\mu'}_G$ yields the desired inequality
\begin{equation}
\label{inequality_for_proof_of_Pp_2.1.3}
\<\lambda+\mu, 2\check{\rho}_M\>\le \<\lambda+\mu', 2\check{\rho}\>
\end{equation}
Assume in addition $M$ antistandard. 
 The inclusion 
$$
S^{w_0^M(\mu)}_M\cap \ov{\Gr}^{\mu}_M\subset S^{w_0^M(\mu)}\cap\ov{\Gr}^{\mu'}_G
$$ 
gives $\<w_0^M(\mu)+\mu, 2\check{\rho}_M\>\le \<w_0^M(\mu)+\mu', 2\check{\rho}\>$, so $f(\mu)\le \<w_0^M(\mu), 2\check{\rho}-2\check{\rho}_M\>$. Since $M$ is antistandard,
$$
\<\lambda-w_0^M(\mu), 2\check{\rho}-2\check{\rho}_M\>\ge 0,
$$ 
and this inequality is strict unless $\lambda\in -\Lambda^+_M$ and 
$\mu=w_0^M(\lambda)$. Assume now $\lambda\in -\Lambda^+_M$ and $\mu=w_0^M(\lambda)$. Since $w_0(\check{\rho})=-\check{\rho}$,  the inequality (\ref{inequality_for_proof_of_Pp_2.1.3}) becomes
$0\le \<-w_0(\lambda)+\mu', 2\check{\rho}\>$, and it is strict unless $\lambda\in -\Lambda^+_G$ and $\mu'=w_0(\lambda)$. 
We are done. \QED

\ssec{Application of the Casselman-Shalika formula} 

\sssec{} Recall the following from (\cite{FGV2}, 7.1). Let $\eta$ be a coweight of $G_{ad}$. A choice of an isomorphism
$\epsilon_{\alpha}: \cO(\<\eta, \check{\alpha}\>)\,\iso\,\Omega$ for each simple root $\check{\alpha}$ of $G$ gives rise to a homomorphism $\chi: U(F)\to \A^1$ given as the sum over simple roots $\check{\alpha}$ of the maps
$$
U(F)\to U/[U,U](F)\toup{\check{\alpha}} F\toup{\epsilon_{\alpha}} \Omega(F)\toup{\Res}\A^1
$$
The character $\chi$ is called admissible of conductor $\eta$. 

 Pick an admissible character $\chi_{-\lambda}: U(F)\to\A^1$ of conductor $-\pr(\lambda)$, here $\pr: \Lambda\to \Lambda_{ad}$ is the projection, and $\Lambda_{ad}$ is the coweight lattice of $G_{ad}$. By (\cite{FGV2}, Lemma~7.1.5) there is a map $\chi^{\lambda}_{-\lambda}: S^{\lambda}\to \A^1$, which is $(U(F), \chi_{-\lambda})$-equivariant. The map $\chi^{\lambda}_{-\lambda}$ is unique up to an additive constant.
 
 Let $i_M: S^{\lambda}_M\hook{} S^{\lambda}$ be the natural inclusion.
\begin{Pp} 
\label{Pp_2.1.15}
Assume $M$ is antistandard. Let $\lambda\in\Lambda, \nu\in\Lambda^+$. 
The natural map  
$$
\RG_c(S^{\lambda}, \cA_{\nu}\otimes (\chi^{\lambda}_{-\lambda})^*\cL_{\psi})\to \RG_c(S^{\lambda}_M, i_M^*(\cA_{\nu}\otimes (\chi^{\lambda}_{-\lambda})^*\cL_{\psi}))
$$
induces an isomorphism in the cohomological degree $\<\lambda, 2\check{\rho}\>$, and both complexes are placed in the cohomological degrees $\le \<\lambda, 2\check{\rho}\>$.
\end{Pp}
\begin{proof}
By (\cite{FGV}, Theorem~1),
$$
\RG_c(S^{\lambda}, \cA_{\nu}\otimes (\chi^{\lambda}_{-\lambda})^*\cL_{\psi})
$$
is given by the Casselman-Shalika formula, namely, it vanishes unless $-\lambda\in\Lambda^+_G$ and $\lambda=w_0(\nu)$. In the latter case it identifies with $\Qlb[-\<\lambda, 2\check{\rho}\>]$. More precisely, if $\lambda=w_0(\nu)$  then the natural map 
$$
\RG_c(S^{\lambda}, \cA_{\nu}\otimes (\chi^{\lambda}_{-\lambda})^*\cL_{\psi})\to i_{\lambda}^*(\cA_{\nu}\otimes (\chi^{\lambda}_{-\lambda})^*\cL_{\psi}))\,\iso\, \Qlb[-\<\lambda, 2\check{\rho}\>]
$$
is an isomorphism, where $i_{\lambda}:\Spec k\to S^{\lambda}$ is the point $t^{\lambda}$. 

 Since $M$ is antistandard, $i_M^*(\chi^{\lambda}_{-\lambda})^*\cL_{\psi}\,\iso\,\Qlb$. Our claim follows now from Proposition~\ref{Pp_2.1.1}. 
\end{proof}

\end{document}